\documentclass[12pt]{amsart}

\usepackage[left=2.5cm,right=2.5cm]{geometry}
\usepackage[mathscr]{euscript}
\usepackage{bookmark}
\usepackage{mathabx}
\usepackage{amssymb}

\newtheorem{theorem}{Theorem}[section]
\newtheorem{lemma}[theorem]{Lemma}
\newtheorem{corollary}[theorem]{Corollary}

\theoremstyle{definition}
\newtheorem{definition}[theorem]{Definition}
\newtheorem{example}[theorem]{Example}

\theoremstyle{remark}
\newtheorem*{remark}{Remark}

\numberwithin{equation}{section}

\DeclareMathOperator{\KR}{KR}
\DeclareMathOperator{\Hom}{Hom}
\DeclareMathOperator{\ch}{ch}
\DeclareMathOperator{\res}{res}

\newcommand{\normord}[1]{:\mathrel{#1}:}

\begin{document}

\title[Quantum \textit{Q}-systems and fermionic sums]{Quantum \textit{Q}-systems and fermionic sums -- the non-simply laced case}

\begin{abstract}
In this paper, we seek to prove the equality of the $q$-graded fermionic sums conjectured by Hatayama \emph{et al.} \cite{HKOTY99} in its full generality, by extending the results of Di Francesco and Kedem \cite{DFK14} to the non-simply laced case. To this end, we will derive explicit expressions for the quantum $Q$-system relations, which are quantum cluster mutations that correspond to the classical $Q$-system relations, and write the identity of the $q$-graded fermionic sums as a constant term identity. As an application, we will show that these quantum $Q$-system relations are consistent with the short exact sequence of the Feigin-Loktev fusion product of Kirillov-Reshetikhin modules obtained by Chari and Venkatesh \cite{CV15}.
\end{abstract}


\author{Mingyan Simon Lin}

\maketitle

\section{Introduction}\label{Section 1}

Kirillov-Reshetikhin (KR-) modules of Yangians first appeared in \cite{KR87}. These are simple, finite-dimensional modules over the Yangian $Y(\mathfrak{g})$ of a simple Lie algebra $\mathfrak{g}$. Their $\mathfrak{g}$-characters satisfy a family of functional relations known as the $Q$-system \cite{KR87, Kirillov89}, which is a system of recurrence relations that comes from the fusion procedure for transfer matrices in the generalized Heisenberg spin chains \cite{KNS94}. Moreover, Kirillov and Reshetikhin \cite{KR87} gave fermionic formulas for the multiplicities of an irreducible $\mathfrak{g}$-module in the tensor product of KR-modules over the Yangian. In the same paper, Kirillov and Reshetikhin (and later generalized by Hatayama \emph{et al.} \cite{HKOTY99}) showed that these fermionic sums have a natural deformation to the $q$-graded case, and this grading was defined combinatorially. It was conjectured by Hatayama \emph{et al.} \cite{HKOTY99}, and proved shortly after by Kirillov \emph{et al.} \cite{KSS02} that the $q$-grading arises in the context of crystals of tensor products of KR-modules over $U_q(\widehat{\mathfrak{sl}}_{r+1})$, by establishing a bijection between rigged configurations and crystal paths in type $A$.

In a subsequent development, KR-modules over the quantum affine algebra $U_q(\widehat{\mathfrak{g}})$ were defined for all types, and these modules are defined in terms of their Drinfeld polynomials \cite{Chari01}. Later, Chari and Moura defined KR-modules for the current algebra $\mathfrak{g}[t]$ in terms of generators and relations \cite{CM06}. Morover, Hatayama \emph{et al.} \cite{HKOTY99} showed that if the $U_q(\mathfrak{g})$-characters of the KR-modules over $U_q(\widehat{\mathfrak{g}})$ satisfy the $Q$-system relations, along with some other asymptotic conditions, then the multiplicity of an irreducible $U_q(\mathfrak{g})$-module in a tensor product of KR-modules over $U_q(\widehat{\mathfrak{g}})$ is expressed in terms of an extended fermionic sum. Subsequently, Nakajima showed that the $q$-characters of the KR-modules over $U_q(\widehat{\mathfrak{g}})$ satisfy the $T$-system relations in the simply-laced case \cite{Nakajima03}, and Hernandez proved, using different methods from \cite{Nakajima03}, that the $q$-characters of the KR-modules over $U_q(\widehat{\mathfrak{g}})$ satisfy the $T$-system relations in the non-simply laced case \cite{Hernandez06}. As the $Q$-system relations are obtained from the $T$-system relations by forgetting the dependence on the spectral parameters, it follows that the $U_q(\mathfrak{g})$-characters of the KR-modules over $U_q(\widehat{\mathfrak{g}})$ satisfy the $Q$-system relations.

In a separate development, Feigin and Loktev introduced a $\mathfrak{g}$-equivariant grading for the tensor product of localized $\mathfrak{g}[t]$-modules \cite{FL99}, and subsequently it was shown in \cite{AK07, DFK08} that the graded multiplicity of an irreducible $\mathfrak{g}$-module in the Feigin-Loktev graded tensor product \cite{FL99} of KR-modules over $\mathfrak{g}[t]$, is given by the $q$-graded fermionic sums. A crucial ingredient in this interpretation is the polynomiality property of the solutions of the $Q$-system \cite{DFK08}. Subsequently, Di Francesco and Kedem \cite{DFK09} showed that the polynomiality property is a consequence of the cluster algebraic formulation of the $Q$-system relations. As cluster algebras admit natural quantum deformations \cite{BZ05}, they were able to use these quantum cluster algebras to prove a conjecture of Hatayama \emph{et al.} concerning the equality of $q$-graded fermionic sums \cite[Conjecture 4.3]{HKOTY99} in the simply-laced case.

In the present paper, we seek to extend the results in \cite{DFK14}, and show that \cite[Conjecture 4.3]{HKOTY99} holds in the non-simply laced case. More precisely, we will first derive the quantum $Q$-system relations corresponding to the non-simply laced quantum $Q$-system cluster algebras. By extending the tools and techniques in \cite{DFK08, DFK14}, the results in the current paper, along with that in \cite{DFK14}, would show that \cite[Conjecture 4.3]{HKOTY99} holds in its entire generality.

\subsection{Main results}

Let us summarize our results briefly. KR-modules over $\mathfrak{g}[t]$ are parameterized by $\alpha\in I_r$ (where $I_r=\{1,\cdots,r\}$ is the set of labels of the simple roots of $\mathfrak{g}$), $m\in\mathbb{Z}_+$, along with a non-zero localization parameter $z\in\mathbb{C}^*$, and they are denoted by $\KR_{\alpha,m}(z)$. As noted above, the $\mathfrak{g}$-characters of these KR-modules over $\mathfrak{g}[t]$ satisfy the $Q$-system relations. More precisely, if we let $Q_{\alpha,k}=\ch\res_{\mathfrak{g}}^{\mathfrak{g}[t]}\KR_{\alpha,k}(z)$ for all $\alpha\in I_r$ and $k\in\mathbb{Z}_+$, then the $Q_{\alpha,k}$'s satisfy the following relation for all $\alpha\in I_r$ and $k\in\mathbb{N}$ \cite{KR87, KNS94}:
\begin{equation}\label{eq:1.1}
Q_{\alpha,k+1}Q_{\alpha,k-1}=Q_{\alpha,k}^2-\prod_{\beta\sim\alpha}\prod_{i=0}^{|C_{\alpha\beta}|-1}Q_{\beta,\lfloor t_{\beta}(k+i)/t_{\alpha}\rfloor}.
\end{equation}
Here, $\beta\sim\alpha$ means that $\beta$ is connected to $\alpha$ in the Dynkin diagram, $C_{\alpha\beta}$ are the entries of the Cartan matrix $C$ of $\mathfrak{g}$, and $t_{\alpha}$ are the integers that satisfy $\min_{\alpha\in I_r}t_{\alpha}=1$ and $C_{\alpha\beta}t_{\beta}=C_{\beta\alpha}t_{\alpha}$ for all $\alpha,\beta\in I_r$. As previously mentioned, the $Q$-system relations \eqref{eq:1.1} could be interpreted as cluster algebra mutations \cite{Kedem08, DFK09}. As cluster algebras admit natural quantum deformations \cite{BZ05}, we can extract the corresponding quantum $Q$-system relations in the quantum $Q$-system cluster algebra. More precisely, we have:

\begin{theorem}\label{1.1}
The quantum $Q$-system relations for $\mathfrak{g}$ are given by 
\begin{equation*}
q^{\frac{1}{\delta}\Lambda_{\alpha\alpha}}\widehat{Q}_{\alpha,k+1}\widehat{Q}_{\alpha,k-1}=\widehat{Q}_{\alpha,k}^2-\normord{\prod_{\beta\sim\alpha}\prod_{i=0}^{|C_{\alpha\beta}|-1}\widehat{Q}_{\beta,\lfloor t_{\beta}(k+i)/t_{\alpha}\rfloor}},
\end{equation*}
with commutation relations
\begin{equation*}
\widehat{Q}_{\alpha,i}\widehat{Q}_{\beta,j}=q^{\frac{1}{\delta}(\Lambda_{\beta\alpha}i-\Lambda_{\alpha\beta}j)}\widehat{Q}_{\beta,j}\widehat{Q}_{\alpha,i}
\end{equation*}
for all $\alpha,\beta\in I_r$ and $i,j\in\mathbb{Z}$ such that $\widehat{Q}_{\alpha,i}$ and $\widehat{Q}_{\beta,j}$ are in the same quantum cluster. Here, $\delta=\det(C)$, $\Lambda_{\alpha\beta}$ are the entries of the matrix $\delta C^{-1}$, and if $y_1,\ldots,y_k$ are pairwise $q$-commuting elements with $y_iy_j=q^{C(y_i,y_j)}y_jy_i$, then the ordered product $\normord{y_1y_2\cdots y_k}$ is given by 
\begin{equation*}
\normord{y_1y_2\cdots y_k}=q^{-\frac{1}{2}\sum_{1\leq i<j\leq n}C(y_i,y_j)}y_1\cdots y_k.
\end{equation*}
\end{theorem}

Here, we would like to emphasize that there isn't a unique quantization of a cluster algebra (and hence the $Q$-system relations). The quantization of the cluster algebra chosen here is the one that gives rise to the Poisson structure on the cluster algebra defined in \cite{GSV03}. 

Our next result concerns the identity of graded sums $M_{\lambda,\mathbf{n}}(q)=N_{\lambda,\mathbf{n}}(q)$. To begin, we let $\mathbf{n}=(n_{\alpha,i})_{\alpha\in I_r,i\in\mathbb{N}}$ be a vector of nonnegative integers that parameterizes a finite set of KR-modules over $\mathfrak{g}[t]$, where $n_{\alpha,i}$ is the number of KR-modules of type $\KR_{\alpha,i}(z)$, and let $\mathcal{F}_{\mathbf{n}}^*$ denote the corresponding Feigin-Loktev graded tensor product of KR-modules parameterized by $\mathbf{n}$ equipped with a $\mathfrak{g}$-equivariant grading, which we will call the fusion product of KR-modules parameterized by $\mathbf{n}$. The graded components $\mathcal{F}_{\mathbf{n}}^*[m]$ of $\mathcal{F}_{\mathbf{n}}^*$ are $\mathfrak{g}$-modules for all $m\in\mathbb{Z}_+$, and we define the generating function $\mathcal{M}_{\lambda,\mathbf{n}}(q)$ for the graded multiplicities of the irreducible $\mathfrak{g}$-module $V(\lambda)$ (where $\lambda$ is a dominant integral weight of $\mathfrak{g}$) in $\mathcal{F}_{\mathbf{n}}^*$ by
\begin{equation}\label{eq:1.2}
\mathcal{M}_{\lambda,\mathbf{n}}(q)=\sum_{m=0}^{\infty}\dim\Hom_{\mathfrak{g}}(\mathcal{F}_{\mathbf{n}}^*[m],V(\lambda))q^m.
\end{equation}
Here, $\Hom_{\mathfrak{g}}(\mathcal{F}_{\mathbf{n}}^*[m],V(\lambda))$ denotes the space of $\mathfrak{g}$-equivariant maps from $\mathcal{F}_{\mathbf{n}}^*[m]$ to $V(\lambda)$. The graded character $\chi_{\mathbf{n}}(q;\mathbf{z})$ of $\mathcal{F}_{\mathbf{n}}^*$ is defined by 
\begin{equation}\label{eq:1.3}
\chi_{\mathbf{n}}(q;\mathbf{z})=\sum_{\lambda}\mathcal{M}_{\lambda,\mathbf{n}}(q)\ch_{\mathbf{z}}V(\lambda),
\end{equation}
where the sum is over all dominant weights $\lambda$, and $\mathbf{z}=(z_1,\cdots,z_r)$ with $z_{\alpha}=e^{\omega_{\alpha}}$, and $\omega_{\alpha}$ is the fundamental weight corresponding to $\alpha$ for all $\alpha\in I_r$. 

It was shown in \cite{AK07, DFK08} that the graded multiplicities $\mathcal{M}_{\lambda,\mathbf{n}}(q^{-1})$ could be given in terms of a $q$-graded fermionic formula $M_{\lambda,\mathbf{n}}(q^{-1})$. In order to define $M_{\lambda,\mathbf{n}}(q^{-1})$ precisely, we need some extra notations. For any vector $\mathbf{m}=(m_{\alpha,i})_{\alpha\in I_r,i\in\mathbb{N}}$ of nonnegative integers and $\alpha\in I_r$, we define the total spin $q_{\alpha}$ as follows:
\begin{equation}\label{eq:1.4}
q_{\alpha}=\ell_{\alpha}+\sum_{\beta\in I_r,j\in\mathbb{N}}jC_{\alpha\beta}m_{\beta,j}-\sum_{j=1}^{\infty}jn_{\alpha,j}.
\end{equation}
Next, for any $\alpha\in I_r$ and $i\in\mathbb{N}$, we define the vacancy numbers $p_{\alpha,i}$ and the quadratic form $Q(\mathbf{m},\mathbf{n})$ as follows:
{\allowdisplaybreaks
\begin{align}
p_{\alpha,i}
&=\sum_{j=1}^{\infty}\min(i,j)n_{\alpha,j}-\sum_{\beta\in I_r,j\in\mathbb{N}}\frac{C_{\alpha\beta}}{t_{\alpha}}\min(t_{\alpha}j,t_{\beta}i)m_{\beta,j},\\
Q(\mathbf{m},\mathbf{n})
&=\frac{1}{2}\sum_{\alpha,\beta\in I_r,i,j\in\mathbb{N}}\frac{C_{\alpha\beta}}{t_{\alpha}}\min(t_{\alpha}j,t_{\beta}i)m_{\alpha,i}m_{\beta,j}-\sum_{\alpha\in I_r,i,j\in\mathbb{N}}\min(i,j)m_{\alpha,i}n_{\alpha,j}.
\end{align}
The $M$-sum $M_{\lambda,\mathbf{n}}(q^{-1})$ \cite[(4.3)]{HKOTY99} is then given by
}
\begin{equation}\label{eq:1.7}
M_{\lambda,\mathbf{n}}(q^{-1})=\sum_{\substack{\mathbf{m}\geq\mathbf{0}\\q_{\alpha}=0,p_{\alpha,i}\geq0}}q^{Q(\mathbf{m},\mathbf{n})}\prod_{\alpha\in I_r,i\in\mathbb{N}}\begin{bmatrix}m_{\alpha,i}+p_{\alpha,i}\\m_{\alpha,i}\end{bmatrix}_q,
\end{equation}
where 
\begin{equation*}
\begin{bmatrix}m+p\\m\end{bmatrix}_q=\frac{(q^{p+1};q)_{\infty}(q^{m+1};q)_{\infty}}{(q;q)_{\infty}(q^{p+m+1};q)_{\infty}},\quad (a;q)_{\infty}=\prod_{j=0}^{\infty}(1-aq^j).
\end{equation*}
The $N$-sum \cite[(4.16)]{HKOTY99} is defined similarly:
\begin{equation}\label{eq:1.8}
N_{\lambda,\mathbf{n}}(q^{-1})=\sum_{\substack{\mathbf{m}\geq\mathbf{0}\\q_{\alpha}=0}}q^{Q(\mathbf{m},\mathbf{n})}\prod_{\alpha\in I_r,i\in\mathbb{N}}\begin{bmatrix}m_{\alpha,i}+p_{\alpha,i}\\m_{\alpha,i}\end{bmatrix}_q.
\end{equation}
Here, the $N$-sum differs from the $M$-sum in that the summands on the right hand side of equation \eqref{eq:1.8} are not bounded by the constraint $p_{\alpha,i}\geq0$.

In order to show that $\mathcal{M}_{\lambda,\mathbf{n}}(q^{-1})=M_{\lambda,\mathbf{n}}(q^{-1})$, it suffices to show that $\mathcal{M}_{\lambda,\mathbf{n}}(1)=M_{\lambda,\mathbf{n}}(1)$, from which the desired identity would follow from the positivity of the $q$-graded sums. In order to prove that the latter identity holds, it was necessary to show that $\mathcal{M}_{\lambda,\mathbf{n}}(1)=N_{\lambda,\mathbf{n}}(1)$, which follows from \cite{HKOTY99, Nakajima03, Hernandez06}, and then show that $M_{\lambda,\mathbf{n}}(1)=N_{\lambda,\mathbf{n}}(1)$, which was shown by Di Francesco and Kedem in \cite{DFK08}. Subsequently, Di Francesco and Kedem \cite{DFK14} showed that $M_{\lambda,\mathbf{n}}(q^{-1})=N_{\lambda,\mathbf{n}}(q^{-1})$ in the simply-laced case. The main bulk of the paper is to show that the last identity holds in the non-simply laced case as well:

\begin{theorem}\label{1.2}
Let $\mathfrak{g}$ be a non-simply laced simple Lie algebra, let $\lambda$ be a dominant integral weight of $\mathfrak{g}$, and $\mathbf{n}=(n_{\alpha,i})_{\alpha\in I_r,i\in\mathbb{N}}$ be a vector of nonnegative integers that parameterizes a finite set of KR-modules over $\mathfrak{g}[t]$. Then we have 
\begin{equation*}
M_{\lambda,\mathbf{n}}(q^{-1})=N_{\lambda,\mathbf{n}}(q^{-1}).
\end{equation*}
\end{theorem}

Our final result concerns an identity of graded characters of fusion product of KR-modules over $\mathfrak{g}[t]$ that extend \eqref{eq:1.1}:

\begin{theorem}\label{1.3}
Let 
\begin{equation*}
K_{\alpha,m}=\bigotimes_{\beta\sim\alpha}\bigotimes_{i=0}^{|C_{\alpha,\beta}|-1}\KR_{\beta,\left\lfloor\frac{t_{\beta}(m+i)}{t_{\alpha}}\right\rfloor}
\end{equation*}
for all $\alpha\in I_r$ and $m\in\mathbb{Z}_+$, and let $K_{\alpha,m}^{\star}$ be the fusion product corresponding to the tensor product $K_{\alpha,m}$ of KR-modules over $\mathfrak{g}[t]$. Then the graded characters of the fusion products of the KR-modules satisfy the following identity for all $\alpha\in I_r$ and $m\in\mathbb{N}$:
\begin{equation*}
\ch_q\KR_{\alpha,m+1}\star\KR_{\alpha,m-1}
=\ch_q \KR_{\alpha,m}\star\KR_{\alpha,m}
-q^m\ch_q K_{\alpha,m}^{\star}.
\end{equation*}
\end{theorem}

Here, we briefly remark that a stronger version of Theorem \ref{1.3} was proved by Chari and Venkatesh in \cite{CV15} for special cases. More precisely, they showed that there exists a short exact sequence of fusion product of KR-modules that extends the $Q$-system relations \eqref{eq:1.1} in certain cases, which we will explain in detail at the end of Section \ref{Section 5}.

The paper is organized as follows: we will first review some properties of the KR-modules over the current algebra and the Feigin-Loktev fusion product in Section \ref{Section 2}. We will then review the construction of the $Q$-system cluster algebras in Section \ref{Section 3}, and derive the quantum $Q$-system relations for all simple Lie algebras. In Section \ref{Section 4}, we will define the restricted versions of the $M$- and $N$-sums $M_{\lambda,\mathbf{n}}(q^{-1})$ and $N_{\lambda,\mathbf{n}}(q^{-1})$, and introduce quantum generating functions in the non-simply laced case that specializes to the $N$-sum. We will then prove factorization properties of these generating functions analogous to those in \cite{DFK08, DFK14}, and use these factorization properties, along with the Laurent polynomiality property of the solutions of the quantum $Q$-system, to prove the identity $M_{\lambda,\mathbf{n}}(q^{-1})=N_{\lambda,\mathbf{n}}(q^{-1})$ in the non-simply laced case. As an immediate consequence, we will use the identity to prove Theorem \ref{1.3} in Section \ref{Section 5}. 

\subsection*{Acknowledgements}

The author would like to thank Rinat Kedem for introducing the problem to the author, and for her guidance throughout this project. The author would also like to thank Phillipe Di Francesco for his helpful discussions and clarifications, and Travis Scrimshaw for his helpful clarifications on the references. Finally, the author would like to thank the referee for his careful reading of the manuscript and helpful remarks.
 The author is supported by graduate fellowship from A*STAR (Agency for Science, Technology and Research, Singapore), and this work is supported in part by the US National Science Foundation (DMS-1802044).
\section{Graded tensor products of cyclic modules over the current algebra}\label{Section 2}

\subsection{Preliminaries}

Let $\mathfrak{g}$ be a simple complex Lie algebra of rank $r$, and let $I_r=\{1,\ldots,r\}$ be labels of the simple roots of $\mathfrak{g}$. Let $C$ denote the Cartan matrix of $\mathfrak{g}$, $t_{\alpha}$ ($\alpha\in I_r$) be the integers that satisfy $t_{\beta}C_{\alpha\beta}=t_{\alpha}C_{\beta\alpha}$ for all $\alpha,\beta\in I_r$, and $t_0=\max_{\alpha\in I_r}t_{\alpha}$. Then we have 
\begin{itemize}
\item $\Pi_{\circ}=\{i\in I_r\mid i\text{ corresponds to a long root of }\mathfrak{g}\}$;
\item $\Pi_{\bullet}=\{i\in I_r\mid i\text{ corresponds to a short root of }\mathfrak{g}\}$.
\end{itemize}
The following table lists all simple Lie algebras $\mathfrak{g}$ and their corresponding $t_0$, $\Pi_{\circ}$ and $\Pi_{\bullet}$:
\begin{center}
  \begin{tabular}{|c|c|c|c|}
  \hline
  $\mathfrak{g}$ & $t_0$ & $\Pi_{\circ}$ & $\Pi_{\bullet}$ \\ \hline
  $A_r, D_r (r\geq 4), E_6, E_7, E_8$ & $1$ & $I_r$ & $\varnothing$ \\ \hline
  $B_r (r\geq 2)$ & $2$ & $\{1,\ldots,r-1\}$ & $\{r\}$ \\ \hline
  $C_r (r\geq 3)$ & $2$ & $\{r\}$ & $\{1,\ldots,r-1\}$ \\ \hline
  $F_4$ & $2$ & $\{1,2\}$ & $\{3,4\}$ \\ \hline
  $G_2$ & $3$ & $\{1\}$ & $\{2\}$ \\ 
  \hline
  \end{tabular}
\end{center}
We note that we have $I_r=\Pi_{\circ}\sqcup\Pi_{\bullet}$, $t_{\alpha}=1$ and $t_{\beta}=t_0$ for all $\alpha\in\Pi_{\circ}$ and $\beta\in\Pi_{\bullet}$. 

In additition, we will also let $\gamma=r,r-1,2,1$ for $\mathfrak{g}=B_r,C_r,F_4$ and $G_2$ respectively, and $\gamma'=r-1,r,3,2$ for $\mathfrak{g}=B_r,C_r,F_4$ and $G_2$ respectively.

\subsection{\texorpdfstring{$\KR$}{KR}-modules}

While the KR-modules over the Yangian or the quantum affine algebra are defined in terms of their Drinfeld polynomials, the KR-modules over the current algebra $\mathfrak{g}[t]=\mathfrak{g}\otimes\mathbb{C}[t]$ are defined in terms of generators and relations, and are the classical limits of the KR-modules over the quantum affine algebra \cite{CM06, Kedem11}. These modules are parameterized by a non-zero $z\in\mathbb{C}^*$ (which we call the localization parameter), $\alpha\in I_r$, and $m\in\mathbb{Z}_+$, and the corresponding KR-module over $\mathfrak{g}[t]$ is denoted $\KR_{\alpha,m}(z)$. 

When $\mathfrak{g}$ is of type $A$, we have $\KR_{\alpha,m}(z)\cong V(m\omega_{\alpha})$ as $\mathfrak{g}$-modules, so KR-modules over $\mathfrak{sl}_{r+1}[t]$ are irreducible as $\mathfrak{sl}_{r+1}$-modules. In general, KR-modules over $\mathfrak{g}[t]$ need not be irreducible as $\mathfrak{g}$-modules. However, the direct sum decomposition of $\KR_{\alpha,m}(z)$ into irreducible $\mathfrak{g}$-modules has the following form \cite{Chari01}:
\begin{equation*}
\KR_{\alpha,m}(z)\cong V(m\omega_{\alpha})\oplus\left(\bigoplus_{\mu\prec m\omega_{\alpha}}V(\mu)^{\oplus m_{\mu}}\right),
\end{equation*}
where $\prec$ is the usual dominance partial ordering on $P$. This decomposition immediately implies that under the restriction of the action to $\mathfrak{g}$, $\KR_{\alpha,m}(z)$ has a highest weight component isomorphic to $V(m\omega_{\alpha})$.

\subsection{The Feigin-Loktev fusion product}

Let $V_1,\ldots,V_N$ be $\mathfrak{g}[t]$-modules with cyclic vectors $v_1,\ldots,v_N$ that generate $V_1,\ldots V_N$ as $\mathfrak{g}[t]$-modules respectively, and let $z_1,\ldots,z_N\in\mathbb{C}$ be pairwise distinct non-zero localization parameters. It was shown by Feigin and Loktev in \cite{FL99} the tensor product $V_1(z_1)\otimes\cdots\otimes V_N(z_N)$ is generated by action of $\mathfrak{g}[t]$ the tensor product $v_1\otimes\cdots\otimes v_N$ of cyclic vectors $v_1,\ldots,v_N$. Hence, they were able to introduce a $\mathfrak{g}$-equivariant grading on the tensor product $V_1(z_1)\otimes\cdots\otimes V_N(z_N)$. This tensor product of localized $\mathfrak{g}[t]$-modules, along with the $\mathfrak{g}$-equivariant grading, is called the Feigin-Loktev fusion product \cite{FL99}. By definition, the fusion product is commutative.

In general, the fusion product of cyclic $\mathfrak{g}[t]$-modules depends on the choice of the localization parameters. However, it was conjectured by Feigin and Loktev in \cite{FL99} and proved in \cite{AK07, DFK08}, that the fusion product is independent of the localization parameters when the constituent modules involved are of Kirillov-Reshetikhin type. Hence, we can suppress the localization parameters, and parameterize the fusion product of KR-modules by a vector $\mathbf{n}=(n_{\alpha,i})_{\alpha\in I_r,i\in\mathbb{N}}$, where $n_{\alpha,i}$ is the number of KR-modules of type $\KR_{\alpha,i}$. 
\section{Quantum \texorpdfstring{$Q$}{\textit{Q}}-systems}\label{Section 3}

\subsection{\texorpdfstring{$Q$}{\textit{Q}}-systems and cluster algebras}

Let $I$ be a subset of $\{Q_{\alpha,k}\mid \alpha\in I_r,k\in\mathbb{Z}\}$ with $|I|=2r$. We say that $I$ forms valid set of initial data for the $Q$-system \eqref{eq:1.1} if any solution of the $Q$-system \eqref{eq:1.1} could be expressed as a function of the elements in $I$. An important example of an initial data for the $Q$-system is the components of the following vector 
\begin{equation*}
\mathbf{y}_{\vec{0}}=(Q_{\alpha,0},Q_{\alpha,1})_{\alpha\in I_r}.
\end{equation*} 
More generally, we let $\vec{s}_k=(kt_{\alpha})_{\alpha\in I_r}$ for all $k\in\mathbb{Z}$. Then the components of the following vector 
\begin{equation}\label{eq:3.1}
\mathbf{y}_{\vec{s}_k}=(Q_{\alpha,kt_{\alpha}},Q_{\alpha,kt_{\alpha}+1})_{\alpha\in I_r}
\end{equation} 
form a valid set of initial data for the $Q$-system. We call $\mathbf{y}_{\vec{s}_k}$ the $k$-th fundamental initial data for the $Q$-system.

\begin{definition}\label{3.1}
Let $\vec{m}=(m_{\alpha})_{\alpha\in I_r}$ be a vector with integer components. We say that $\vec{m}$ is a generalized Motzkin path if 
\begin{equation}\label{eq:3.2}
-\min(t_{\alpha},t_{\beta})\leq t_{\alpha}m_{\beta}-t_{\beta}m_{\alpha}\leq\max(t_{\alpha},t_{\beta}) 
\end{equation}
whenever $C_{\alpha\beta}=-1$. 
\end{definition}

\begin{example}\label{3.2}
An important example of a generalized Motzkin path is the vector $\vec{s}_k=(kt_{\alpha})_{\alpha\in I_r}$. 
\end{example}

\begin{example}\label{3.3}
Let $\mathfrak{g}$ be of type $B_3$. Then it is easy to verify that the vector $(-1,0,2)$ is a generalized Motzkin path.
\end{example}

In general, a valid set of initial data for the $Q$-system is determined by a generalized Motzkin path.

It was shown by Di Francesco and Kedem in \cite{DFK09} that the $Q$-system relations could be realized as cluster algebra mutations. While the cluster algebra mutation relations are written without any subtractions \cite{FZ02}, the RHS of the $Q$-system relations \eqref{eq:1.1} is written with a subtraction. To avoid the use of cluster algebras with coefficients (as in the Appendix of \cite{DFK09}), we would need to normalize our $Q$-system relations accordingly. Following \cite[Lemma 2.1]{DFK09}, we let $\mu_{\alpha}=\sum_{\beta\in I_r}(C^{-1})_{\beta,\alpha}$, $\epsilon_{\alpha}=e^{i\pi\mu_{\alpha}}$ and $R_{\alpha,k}=\epsilon_{\alpha}Q_{\alpha,k}$ for all $\alpha\in I_r$ and $k\in\mathbb{Z}$. Then it follows that the normalized variables $R_{\alpha,k}$ satisfy the following relation: 
\begin{equation}\label{eq:3.3}
R_{\alpha,k+1}R_{\alpha,k-1}=R_{\alpha,k}^2+\prod_{\beta\sim\alpha}\prod_{i=0}^{|C_{\alpha\beta}|-1}R_{\beta,\lfloor t_{\beta}(k+i)/t_{\alpha}\rfloor},
\end{equation}
We will thereby refer to \eqref{eq:3.3} as the normalized $Q$-system for $\mathfrak{g}$.

The cluster algebra associated to the normalized $Q$-system \eqref{eq:3.3} for $\mathfrak{g}$ is defined from the initial cluster $(\mathbf{x}[0],B)$, where $\mathbf{x}[0]=(R_{1,0},R_{2,0},\ldots,R_{r,0},R_{1,1},\ldots,R_{r,1})$, and
\begin{equation}\label{eq:3.4}
B=
\begin{pmatrix}
C^t-C & -C^t\\
C & 0
\end{pmatrix}.
\end{equation}
In fact, more is true:

\begin{theorem}[Theorems 3.1, 3.4, 3.12, and Lemma 3.10, \cite{DFK09}]\label{3.4}
There exists a cluster graph $\mathcal{G}_{\mathfrak{g}}$, which includes all nodes $(\mathbf{x}[k],B)$ and $(\mathbf{x}[k'],-B)$ labeled by $k\in\mathbb{Z}_+$, with the clusters $\mathbf{x}[k]$ and $\mathbf{x}[k']$ defined as follows:
\begin{align*}
\mathbf{x}[k]&=(R_{1,2t_1k},R_{2,2t_2k},\ldots,R_{r,2t_rk},R_{1,2t_1k+1},\ldots,R_{r,2t_rk+1}),\\
\mathbf{x}[k']&=(R_{1,2t_1k},R_{2,2t_2k},\ldots,R_{r,2t_rk},R_{1,2t_1k-1},\ldots,R_{r,2t_rk-1})
\end{align*}
and mutation matrix $B$ defined as in \eqref{eq:3.4}, such that all cluster algebra mutations in the graph $\mathcal{G}_{\mathfrak{g}}$ are normalized $Q$-system relations \eqref{eq:3.3}.  
\end{theorem}

\subsection{Quantum \texorpdfstring{$Q$}{\textit{Q}}-systems} 

The quantum $Q$-system relations were first defined for type $A$ in \cite{DFK11}, and subsequently for all simply-laced $\mathfrak{g}$ in \cite{DFK14}. Our goal in this subsection is to do the same for all non-simply laced $\mathfrak{g}$. In order to define these quantum $Q$-system relations, we would first need to recall the definition of a quantum cluster algebra \cite{BZ05}.

Let us fix a positive integer $m$ (the rank of the quantum cluster algebra), and a non-singular skew-symmetric $m\times m$ matrix integer $\widetilde{B}$. Let $\widetilde{\Lambda}$ be a $m\times m$ integer matrix that satisfies the following compatibility relation with the exchange matrix $B$ for some positive integer $\delta$:
\begin{equation*}
\widetilde{B}\widetilde{\Lambda}=-\delta I,
\end{equation*}
and let $\mathcal{T}$ be the $\mathbb{Z}[\nu^{\pm1/2}]$-algebra with generators $X_1,X_2,\ldots,X_m$ and relations
\begin{equation}\label{eq:3.5}
X_iX_j=\nu^{\widetilde{\Lambda}_{ij}}X_jX_i,\quad i,j=1,\ldots,m.
\end{equation}
In addition, we will also let $\mathcal{F}$ be the skew-field of fractions of $\mathcal{T}$. 

Next, for any pairwise $\nu$-commuting elements $y_1,y_2,\ldots,y_k\in\mathcal{F}$ with 
\begin{equation}\label{eq:3.6}
y_iy_j=\nu^{C(y_i,y_j)}y_jy_i 
\end{equation}
for any $i,j\in\{1,\ldots,k\}$, we define the ordered product $\normord{y_1y_2\cdots y_k}$ as follows:
\begin{equation}\label{eq:3.7}
\normord{y_1y_2\cdots y_k}=\nu^{\frac{1}{2}\sum_{i>j}C(y_i,y_j)}y_1y_2\cdots y_k.
\end{equation}
Then it is clear that the ordered product $\normord{\cdot}$ is associative and commutative, so we may write $\normord{\prod_{i=1}^k y_i}$ or $\normord{\prod_{i\in[1,k]}y_i}$ in lieu of $\normord{y_1y_2\cdots y_k}$.

We are now ready to define the quantum analogue of mutations. The quantum mutations $\mu_i$ act in the same way on the exchange matrix $\widetilde{B}$ as in the classical case. We write $\mu_i(X_j)=X_j$ if $j\neq i$, and 
\begin{equation}\label{eq:3.8}
\mu_i(X_i)=\normord{\prod_{j=1}^m X_j^{-\delta_{ji}+[\widetilde{B}_{ji}]_+}}+\normord{\prod_{j=1}^m X_j^{-\delta_{ji}+[-\widetilde{B}_{ji}]_+}},
\end{equation}
where $[n]_+=\max(n,0)$.

Now, we would like to derive the normalized quantum $Q$-system relations, by deriving the quantum cluster algebra mutations in the quantum $Q$-system cluster algebra associated to the exchange matrix $B$ defined in equation \eqref{eq:3.4} (that is, $\widetilde{B}=B)$, that correspond to the normalized $Q$-system relations \eqref{eq:3.3}. To this end, we let 
\begin{equation}\label{eq:3.9}
\Lambda=\delta C^{-1},\quad\delta=\det(C),
\end{equation}
The skew-symmetric matrix $\widetilde{\Lambda}$ that gives rise to the commutation relations \ref{eq:3.5} between the quantum cluster variables in the initial cluster $\mathbf{X}[0]=(\widehat{R}_{1,0},\widehat{R}_{2,0},\ldots,\widehat{R}_{r,0},\widehat{R}_{1,1},\ldots,\widehat{R}_{r,1})$ is then given by
\begin{equation}\label{eq:3.10}
\widetilde{\Lambda}=
-\delta B^{-1}
\left(
\begin{array}{c|c}
0 & -\Lambda\\
\hline
\Lambda^t & \Lambda^t-\Lambda
\end{array}
\right).
\end{equation}
In particular, the quantum cluster variables in $\mathbf{X}[0]$ satisfy the following commutation relations:
\begin{equation}\label{eq:3.11}
\widehat{R}_{\alpha,i}\widehat{R}_{\beta,j}=\nu^{\Lambda_{\beta\alpha}i-\Lambda_{\alpha\beta}j}\widehat{R}_{\beta,j}\widehat{R}_{\alpha,i},\quad\alpha,\beta\in I_r,i,j\in\{0,1\}.
\end{equation}
In addition, we will also let $\widehat{R}_{\alpha,i}$ denote the quantum cluster variable in the quantum cluster algebra associated to $B$ that corresponds to $R_{\alpha,i}$ for all $\alpha\in I_r$ and $i\in\mathbb{Z}$.

We first observe that each generalized Motzkin path corresponds to a cluster in the cluster algebra corresponding to the normalized $Q$-system. In particular, it follows from inequality \eqref{eq:3.2} that if 
\begin{equation}\label{eq:3.12}
|t_{\alpha}j-t_{\beta}i|\leq t_{\alpha}+t_{\beta}-\min(t_{\alpha},t_{\beta})\delta_{\alpha\beta}, 
\end{equation}
then the variables $\widehat{R}_{\alpha,i}$ and $\widehat{R}_{\beta,j}$ are in the same quantum cluster, and hence they $\nu$-commute. Therefore, the quantum cluster algebra mutations corresponding to \eqref{eq:3.3} has the following form:
\begin{align}
\widehat{R}_{\alpha,k+1}
&=\normord{\widehat{R}_{\alpha,k-1}^{-1}\widehat{R}_{\alpha,k}^2}+\normord{\widehat{R}_{\alpha,k-1}^{-1}\prod_{\beta\sim\alpha}\prod_{i=0}^{|C_{\alpha\beta}|-1}\widehat{R}_{\beta,\lfloor t_{\beta}(k+i)/t_{\alpha}\rfloor}},\label{eq:3.13}\\
\widehat{R}_{\alpha,k-1}
&=\normord{\widehat{R}_{\alpha,k+1}^{-1}\widehat{R}_{\alpha,k}^2}+\normord{\widehat{R}_{\alpha,k+1}^{-1}\prod_{\beta\sim\alpha}\prod_{i=0}^{|C_{\alpha\beta}|-1}\widehat{R}_{\beta,\lfloor t_{\beta}(k+i)/t_{\alpha}\rfloor}}.\label{eq:3.14}
\end{align}
Using the fact that $B\widetilde{\Lambda}=-\delta I$, it follows from equation \eqref{eq:3.13} that the normalized quantum $Q$-system relations for $\mathfrak{g}$ are given by 
\begin{equation}\label{eq:3.15}
\nu^{-\Lambda_{\alpha\alpha}}\widehat{R}_{\alpha,k+1}\widehat{R}_{\alpha,k-1}=\widehat{R}_{\alpha,k}^2+\nu^{-\frac{\delta}{2}}\normord{\prod_{\beta\sim\alpha}\prod_{i=0}^{|C_{\alpha\beta}|-1}\widehat{R}_{\beta,\lfloor t_{\beta}(k+i)/t_{\alpha}\rfloor}}.
\end{equation}
In addition to the normalized quantum $Q$-system relations for $\mathfrak{g}$ \eqref{eq:3.15}, we would also need to know the relevant commutation relations that are satisfied by the normalized quantum $Q$-system variables $\widehat{R}_{\alpha,i}$:

\begin{lemma}\label{3.5}
Let $\vec{m}$ is a generalized Motzkin path \eqref{eq:3.2}. Then within each valid initial data set $(\widehat{R}_{\alpha,m_{\alpha}},\widehat{R}_{\alpha,m_{\alpha}+1})_{\alpha\in I_r}$, the normalized quantum $Q$-system variables $\widehat{R}_{\alpha,i}$ satisfy the following commutation relations:
\begin{equation}\label{eq:3.16}
\widehat{R}_{\alpha,i}\widehat{R}_{\beta,j}=\nu^{\Lambda_{\beta\alpha}i-\Lambda_{\alpha\beta}j}\widehat{R}_{\beta,j}\widehat{R}_{\alpha,i}.
\end{equation}
\end{lemma}

\begin{proof}
The following proof of \eqref{eq:3.16} uses the same strategy as in the proof of \cite[Lemma 3.2]{DFK09}, where the commutation relations \eqref{eq:3.16} are proved in the simply laced case. We will first show that \eqref{eq:3.16} holds in the case where 
\begin{equation}\label{eq:3.17}
|t_{\alpha}j-t_{\beta}i|\leq\max(t_{\alpha},t_{\beta}).
\end{equation}
Firstly, it follows from Theorem \ref{3.4} and equation \eqref{eq:3.10} that \eqref{eq:3.16} holds whenever we have $\alpha,\beta\in\Pi_{\circ}$ and $|i-j|\leq 1$, and in the case where $\alpha\in\Pi_{\circ}$, $\beta\in\Pi_{\bullet}$, we have the following commutation relations for all $k\in\mathbb{Z}$, $i=0,\pm1$ and $j=\pm(t_0-1),\pm t_0$:
\begin{align}
\widehat{R}_{\alpha,2k}\widehat{R}_{\beta,2t_0k+i}&=\nu^{-i\Lambda_{\alpha\beta}}\widehat{R}_{\beta,2t_0k+i}\widehat{R}_{\alpha,2k},\quad\text{and}\label{eq:3.18}\\
\widehat{R}_{\alpha,2k+1}\widehat{R}_{\beta,t_0(2k+1)+j}&=\nu^{-j\Lambda_{\alpha\beta}}\widehat{R}_{\beta,t_0(2k+1)+j}\widehat{R}_{\alpha,2k+1}.\label{eq:3.19}
\end{align}
Let 
\begin{equation*}
\widehat{\mathscr{T}}_{\alpha,k}=\widehat{R}_{\alpha,k}^{-2}\normord{\prod_{\beta\sim\alpha}\prod_{i=0}^{|C_{\alpha\beta}|-1}\widehat{R}_{\beta,\lfloor t_{\beta}(k+i)/t_{\alpha}\rfloor}}
\end{equation*}
and
\begin{equation*}
c_{\alpha,k}=C\left(\widehat{R}_{\alpha,k},\widehat{R}_{\alpha,k-1}^{-1}\right)-\frac{1}{2}\sum_{\beta\sim\alpha}\sum_{i=0}^{|C_{\alpha\beta}|-1}C\left(\widehat{R}_{\beta,\lfloor t_{\beta}(k+i)/t_{\alpha}\rfloor},\widehat{R}_{\alpha,k-1}^{-1}\right)
\end{equation*}
for all $\alpha\in I_r$ and $k\in\mathbb{Z}$. Then it follows that equation \eqref{eq:3.13} could be rewritten as
\begin{equation*}
\widehat{R}_{\alpha,k+1}
=\nu^{-C\left(\widehat{R}_{\alpha,k},\widehat{R}_{\alpha,k-1}^{-1}\right)}\widehat{R}_{\alpha,k}^2(1+\nu^{c_{\alpha,k}}\widehat{\mathscr{T}}_{\alpha,k})\widehat{R}_{\alpha,k-1}^{-1}.
\end{equation*}
Thus, in order to show that the commutation relation \eqref{eq:3.18} holds for $i=2$, it suffices to show that $\widehat{R}_{\alpha,2k}$ commutes with $\widehat{\mathscr{T}}_{\beta,2t_0k+1}$. Using the fact that $C(\widehat{R}_{\alpha,2k},\widehat{R}_{\omega,2k+i})=-i\Lambda_{\alpha\omega}$ for all $\omega\in\Pi_{\circ}$ and $i\in\{0,1\}$, and $C(\widehat{R}_{\alpha,2k},\widehat{R}_{\omega,2t_0k+1})=-\Lambda_{\alpha\omega}$ for all $\omega\in\Pi_{\bullet}$, it follows that we have 
\begin{equation*}
C\left(\widehat{R}_{\alpha,2k},\widehat{R}_{\omega,\lfloor t_{\omega}(2t_0k+1)/t_{\beta}\rfloor}\cdots\widehat{R}_{\omega,\lfloor t_{\omega}(2t_0k+|C_{\beta\omega}|)/t_{\beta}\rfloor}\right)=-\Lambda_{\alpha\omega}=\Lambda_{\alpha\omega}C_{\omega\beta}
\end{equation*}
for all $\omega\sim\beta$. As we have $C\left(\widehat{R}_{\alpha,2k},\widehat{R}_{\beta,2t_0k+1}^{-2}\right)=2\Lambda_{\alpha\beta}=\Lambda_{\alpha\beta}C_{\beta\beta}$, it follows that we have
\begin{equation*}
C\left(\widehat{R}_{\alpha,2k},\widehat{\mathscr{T}}_{\beta,2t_0k+1}\right)=\Lambda_{\alpha\beta}C_{\beta\beta}+\sum_{\omega\sim\beta}\Lambda_{\alpha\omega}C_{\omega\beta}=\sum_{\omega\in I_r}\Lambda_{\alpha\omega}C_{\omega\beta}=0,
\end{equation*}
where the last equality follows from the fact that $\Lambda C=\delta I$. So this shows that $\widehat{R}_{\alpha,2k}$ commutes with $\widehat{\mathscr{T}}_{\beta,2t_0k+1}$, and hence \eqref{eq:3.18} holds for $i=2$. By repeating the argument as before (and using equation \eqref{eq:3.14} where necessary), it follows that the commutation relations \eqref{eq:3.18} and \eqref{eq:3.19} hold for all $|i|,|j|\leq t_0$. Consequently, the commutation relations \eqref{eq:3.16} hold whenever inequality \eqref{eq:3.17} is satisfied, in the case where either $\alpha\in\Pi_{\circ}$, $\beta\in\Pi_{\bullet}$ or $\alpha\in\Pi_{\bullet}$, $\beta\in\Pi_{\circ}$. The same argument above would also imply that the commutation relations \eqref{eq:3.16} hold whenever inequality \eqref{eq:3.17} is satisfied, in the case where $\alpha,\beta\in\Pi_{\bullet}$. 

Finally, to obtain the rest of the commutation relations, we will proceed by induction on either $|i-j|$ (in the case where either $\alpha,\beta\in\Pi_{\circ}$ or $\alpha,\beta\in\Pi_{\bullet}$), or $|t_{\alpha}j-t_{\beta}i|$ (in the case where either $\alpha\in\Pi_{\circ}$, $\beta\in\Pi_{\bullet}$ or $\alpha\in\Pi_{\bullet}$, $\beta\in\Pi_{\circ}$), and repeat the same argument as before, with the base case(s) following from the fact that the commutation relations \eqref{eq:3.16} hold whenever inequality \eqref{eq:3.17} is satisfied. 
\end{proof}

Having obtained the normalized quantum $Q$-system relations for $\mathfrak{g}$, we will now proceed to renormalize \eqref{eq:3.15} to obtain the quantum $Q$-system relations for $\mathfrak{g}$. We let $a_{\alpha}=\frac{1}{2}\sum_{\beta\in I_r}\Lambda_{\alpha\beta}$ and $\widehat{Q}_{\alpha,k}=\epsilon_{\alpha}^{-1}\nu^{a_{\alpha}}\widehat{R}_{\alpha,k}$ for all $\alpha\in I_r$ and $k\in\mathbb{Z}$. Then it follows from \eqref{eq:3.15} that we have 
\begin{equation}\label{eq:3.20}
\nu^{-\Lambda_{\alpha\alpha}}\widehat{Q}_{\alpha,k+1}\widehat{Q}_{\alpha,k-1}=\widehat{Q}_{\alpha,k}^2-\normord{\prod_{\beta\sim\alpha}\prod_{i=0}^{|C_{\alpha\beta}|-1}\widehat{Q}_{\beta,\lfloor t_{\beta}(k+i)/t_{\alpha}\rfloor}}
\end{equation}
for all $\alpha\in I_r$ and $k\in\mathbb{Z}$. We will thereby refer to \eqref{eq:3.20} as the quantum $Q$-system for $\mathfrak{g}$.

\begin{remark}
The quantum $Q$-system variables $\widehat{Q}_{\alpha,k}$ also satisfy the same commutation relations \eqref{eq:3.16} (with the $\widehat{R}_{\alpha,k}$'s replaced by $\widehat{Q}_{\alpha,k}$'s). 
\end{remark}

For latter convenience, we will let 
\begin{equation}\label{eq:3.21}
\widehat{Y}_{\alpha,k}=\widehat{Q}_{\alpha,k}^{-2}\normord{\prod_{\beta\sim\alpha}\prod_{i=0}^{|C_{\alpha\beta}|-1}\widehat{Q}_{\beta,\lfloor t_{\beta}(k+i)/t_{\alpha}\rfloor}} 
\end{equation}
for all $\alpha\in I_r$ and $k\in\mathbb{Z}$. Then the quantum $Q$-system relations \eqref{eq:3.20} is equivalent to 
\begin{equation}\label{eq:3.22}
\nu^{-\Lambda_{\alpha\alpha}}\widehat{Q}_{\alpha,k+1}\widehat{Q}_{\alpha,k-1}=\widehat{Q}_{\alpha,k}^2(1-\widehat{Y}_{\alpha,k}).
\end{equation}

We will conclude this section with a technical lemma, which will come in handy in the factorization of the quantum generating functions that we will introduce in the next section:

\begin{lemma}\label{3.6}
Let $\widehat{Z}_{\alpha,k}=\widehat{Q}_{\alpha,k}\widehat{Q}_{\alpha,k+1}^{-1}$ for all $\alpha\in I_r$ and $k\in\mathbb{Z}$. Then for all distinct $\alpha,\beta\in I_r$, $i\in\mathbb{Z}$ and $|p|\leq t_0$, we have:
\begin{enumerate}
\item $\widehat{Z}_{\beta,i}$ commutes with $\widehat{Z}_{\alpha,i}$ and $\widehat{Z}_{\alpha,i-1}$ whenever $\alpha,\beta\in\Pi_{\circ}$ or $\alpha,\beta\in\Pi_{\bullet}$,
\item $\widehat{Z}_{\beta,i}$ commutes with $\widehat{Z}_{\alpha,t_0i+p}$ whenever $\alpha\in\Pi_{\bullet}$ and $\beta\in\Pi_{\circ}$,
\item $\widehat{Q}_{\beta,i+2}$ commutes with $\widehat{Z}_{\alpha,i}^{-1}\widehat{Z}_{\alpha,i+1}$ whenever $\alpha,\beta\in\Pi_{\circ}$ or $\alpha,\beta\in\Pi_{\bullet}$,
\item $\widehat{Q}_{\beta,t_0i+1}$ commutes with $\widehat{Z}_{\alpha,i-1}^{-1}\widehat{Z}_{\alpha,i}$ whenever $\alpha\in\Pi_{\circ}$ and $\beta\in\Pi_{\bullet}$,
\item $\widehat{Q}_{\beta,i}$ commutes with $\widehat{Z}_{\alpha,t_0i-2}^{-1}\widehat{Z}_{\alpha,t_0i-1}$ and $\widehat{Z}_{\alpha,t_0i}^{-1}\widehat{Z}_{\alpha,t_0i+1}$ whenever $\alpha\in\Pi_{\bullet}$ and $\beta\in\Pi_{\circ}$,
\item $\widehat{Q}_{\beta,i}$ commutes with $\widehat{Y}_{\alpha,i+1}$ whenever $\alpha,\beta\in\Pi_{\circ}$ or $\alpha,\beta\in\Pi_{\bullet}$,
\item $\widehat{Q}_{\beta,t_0i-1}$ commutes with $\widehat{Y}_{\alpha,i}$ whenever $\alpha\in\Pi_{\circ}$ and $\beta\in\Pi_{\bullet}$,
\item $\widehat{Q}_{\beta,i}$ commutes with $\widehat{Y}_{\alpha,t_0i+p}$ whenever $\alpha\in\Pi_{\bullet}$ and $\beta\in\Pi_{\circ}$.
\end{enumerate}
\end{lemma}

We shall omit the proof of Lemma \ref{3.6}, as all statements follow readily from Lemma \ref{3.5} (and using inequality \eqref{eq:3.12}) where necessary.

Similar to the classical case, the cluster variables in any quantum cluster algebra satisfy a Laurent property, that is, given an initial cluster $\mathbf{X}$ of a quantum cluster algebra, we can express any cluster variable as a (non-commutative) Laurent polynomial in the variables of $\mathbf{X}$. As the quantum $Q$-system relations are obtained from the quantum $Q$-system cluster algebras, it follows that the solutions of the quantum $Q$-system inherit this Laurent property as well:

\begin{lemma}\label{3.7}
For any generalized Motzkin path $\vec{m}$, $\alpha\in I_r$ and $i\in\mathbb{Z}$, $\widehat{Q}_{\alpha,i}$ could be expressed as a (non-commutative) Laurent polynomial in the initial data $\mathbf{Y}_{\vec{m}}=(\widehat{Q}_{\alpha,m_{\alpha}},\widehat{Q}_{\alpha,m_{\alpha}+1})_{\alpha\in I_r}$ with coefficients in $\mathbb{Z}[\nu^{\pm1/2}]$.
\end{lemma} 
\section{The graded tensor product multiplicities and the quantum \texorpdfstring{$Q$}{\textit{Q}}-system: the non-simply laced case}\label{Section 4}

Our goal in this section is to show that Theorem \ref{1.2} holds. As in \cite{DFK08, DFK14}, we will prove a slightly stronger statement, where we fix a positive integer $k$, and restrict the sums $M_{\lambda,\mathbf{n}}(q^{-1})$ and $N_{\lambda,\mathbf{n}}(q^{-1})$ to the $\mathbf{m}$'s that satisfy $m_{\alpha,i}=0$ for all $\alpha\in I_r$ and $i>t_{\alpha}k$, and show that these restricted sums are equal to each other. The equality $M_{\lambda,\mathbf{n}}(q^{-1})=N_{\lambda,\mathbf{n}}(q^{-1})$ would subsequently follow when $k$ is sufficiently large.

\subsection{The restricted \texorpdfstring{$M$}{\textit{M}}- and \texorpdfstring{$N$}{\textit{N}}-sums}

Throughout this section, we will assume that $\mathfrak{g}$ is non-simply laced unless otherwise stated. Let us fix a vector $\mathbf{n}=(n_{\alpha,i})_{\alpha\in I_r,i\in\mathbb{N}}$ of nonnegative integers that parameterizes a finite set of KR-modules over $\mathfrak{g}[t]$, and a dominant integral weight $\lambda=\sum_{\alpha\in I_r}\ell_{\alpha}\omega_{\alpha}$ of $\mathfrak{g}$, and a positive integer $k$. For latter convenience, we will let $\ell$ denote the vector $(\ell_{\alpha})_{\alpha\in I_r}$, and $\mathbf{n}^{(j)}$ denote the vector $(n_{\alpha,i})_{\alpha\in I_r,i>t_{\alpha}j}$. Also, for any $j,p\in\mathbb{Z}_+$ satisfying $0\leq j\leq k$ and $p<t_0$, we let $J_{\mathfrak{g}}^{(j,p)}$ be the following indexing set:
\begin{equation*}
J_{\mathfrak{g}}^{(j,p)}=\{(\alpha,i)\mid\alpha\in I_r,j+1\leq i\leq k\,(\alpha\in\Pi_{\circ}),t_0j+p+1\leq i\leq t_0k\, (\alpha\in\Pi_{\bullet})\}.
\end{equation*}
For any vector $\mathbf{m}=(m_{\alpha,i})_{(\alpha,i)\in J_{\mathfrak{g}}^{(0,0)}}$ of nonnegative integers and any $\alpha\in I_r$, we define the restricted total spin $q_{\alpha,0}$ as follows:
\begin{equation}\label{eq:4.1}
q_{\alpha,0}=\ell_{\alpha}+\sum_{(\beta,j)\in J_{\mathfrak{g}}^{(0,0)}}jC_{\alpha\beta}m_{\beta,j}-\sum_{j=1}^{t_{\alpha}k}jn_{\alpha,j}.
\end{equation}
Next, for any $(\alpha,i)\in J_{\mathfrak{g}}^{(0,0)}$, we define the restricted vacancy numbers $p_{\alpha,i}$ and the restricted quadratic form $Q_k(\mathbf{m},\mathbf{n})$ as follows:
\begin{align}
p_{\alpha,i}
&=\sum_{j=1}^{t_{\alpha}k}\min(i,j)n_{\alpha,j}-\sum_{(\beta,j)\in J_{\mathfrak{g}}^{(0,0)}}\frac{C_{\alpha\beta}}{t_{\alpha}}\min(t_{\alpha}j,t_{\beta}i)m_{\beta,j},\label{eq:4.2}\\
Q_k(\mathbf{m},\mathbf{n})
&=\frac{1}{2}\sum_{(\alpha,i),(\beta,j)\in J_{\mathfrak{g}}^{(0,0)}}\frac{C_{\alpha\beta}}{t_{\alpha}}\min(t_{\alpha}j,t_{\beta}i)m_{\alpha,i}m_{\beta,j}-\sum_{\alpha\in I_r}\sum_{i,j=1}^{t_{\alpha}k}\min(i,j)m_{\alpha,i}n_{\alpha,j}.\label{eq:4.3}
\end{align}
The restricted $M$- and $N$- sums $M_{\lambda,\mathbf{n}}^{(k)}(q^{-1})$ and $N_{\lambda,\mathbf{n}}^{(k)}(q^{-1})$ are then defined as follows:
\begin{align}
M_{\lambda,\mathbf{n}}^{(k)}(q^{-1})
&=\sum_{\substack{\mathbf{m}\geq\mathbf{0}\\q_{\alpha,0}=0,p_{\alpha,i}\geq0}}q^{Q_k(\mathbf{m},\mathbf{n})}\prod_{(\alpha,i)\in J_{\mathfrak{g}}^{(0,0)}}\begin{bmatrix}m_{\alpha,i}+p_{\alpha,i}\\m_{\alpha,i}\end{bmatrix}_q,\quad\text{and}\label{eq:4.4}\\
N_{\lambda,\mathbf{n}}^{(k)}(q^{-1})
&=\sum_{\substack{\mathbf{m}\geq\mathbf{0}\\q_{\alpha,0}=0}}q^{Q_k(\mathbf{m},\mathbf{n})}\prod_{(\alpha,i)\in J_{\mathfrak{g}}^{(0,0)}}\begin{bmatrix}m_{\alpha,i}+p_{\alpha,i}\\m_{\alpha,i}\end{bmatrix}_q.\label{eq:4.5}
\end{align}
We note that the integers $q_{\alpha,0}$ were introduced for the purpose of imposing the restriction on the summation variables.

The rest of the section is devoted to proving the following theorem:
\begin{theorem}\label{4.1}
Let $\lambda$ be a dominant integral weight of $\mathfrak{g}$, $\mathbf{n}=(n_{\alpha,i})_{\alpha\in I_r,i\in\mathbb{N}}$ be a vector of nonnegative integers that parameterizes a finite set of KR-modules over $\mathfrak{g}[t]$, and $k$ be a positive integer. Then we have 
\begin{equation*}
M_{\lambda,\mathbf{n}}^{(k)}(q^{-1})=N_{\lambda,\mathbf{n}}^{(k)}(q^{-1}).
\end{equation*}
\end{theorem}

Here, we remark that Theorem \ref{4.1} is proved for all simple $\mathfrak{g}$ in the classical case $q=1$ \cite{DFK08}, and for all simply-laced $\mathfrak{g}$ in the graded case \cite{DFK14}. In both cases, the broad strategy in proving Theorem \ref{4.1} is as follows: firstly, Di Francesco and Kedem defined (quantum) generating functions whose ``constant term evaluation'' is equal to $N_{\lambda,\mathbf{n}}^{(k)}(q^{-1})$ up to a constant depending on $\mathbf{n}$ and $\lambda$, and factorizes into a into a product of (quantum) $Q$-system variables and their inverses. In the former case, Di Francesco and Kedem also defined intermediate generating functions to account for the extra terms corresponding to the short root indices. The next step is to show that the terms in the summation with negative $p_{\alpha,i}$ for some $\alpha\in I_r$, $1\leq i\leq t_{\alpha}k$, do not contribute to the constant term evaluation, from which we would get $M_{\lambda,\mathbf{n}}^{(k)}(q^{-1})=N_{\lambda,\mathbf{n}}^{(k)}(q^{-1})$. 

In a similar fashion, we would employ a largely similar strategy outlined in \cite{DFK14} in proving Theorem \ref{4.1} for the case where $\mathfrak{g}$ is non-simply laced, where as in \cite{DFK08}, we will also define intermediate generating quantum generating functions to account for the extra terms corresponding to the short root indices.

\subsection{The quantum generating functions}

Firstly, we define the modified restricted vacancy numbers $q_{\alpha,i}$ for all $\alpha\in I_r$ and $1\leq i\leq t_{\alpha}k$ as follows:
\begin{align}\label{eq:4.6}
q_{\alpha,i}
&=p_{\alpha,i}+q_{\alpha,0}\nonumber\\
&=\ell_{\alpha}+\sum_{\substack{(\beta,j)\in J_{\mathfrak{g}}^{(0,0)}:\\t_{\alpha}j>t_{\beta}i}}\frac{C_{\alpha\beta}}{t_{\alpha}}(t_{\alpha}j-t_{\beta}i)m_{\beta,j}-\sum_{j=i+1}^{t_{\alpha}k}(j-i)n_{\alpha,j}.
\end{align}
The equation \eqref{eq:4.6} could be written in a more explicit manner as follows: for all $\alpha\in\Pi_{\circ}$ and $0\leq i\leq k$, we have
\begin{equation}\label{eq:4.7}
q_{\alpha,i}
=\ell_{\alpha}+\sum_{j=i+1}^k(j-i)\left(\sum_{\beta\in\Pi_{\circ}}C_{\alpha\beta}m_{\beta,j}-n_{\alpha,j}\right)+\sum_{j=t_0i+1}^{t_0k}(j-t_0i)\sum_{\omega\in\Pi_{\bullet}}C_{\alpha\omega}m_{\omega,j},
\end{equation}
and for all $\beta\in\Pi_{\bullet}$ and $0\leq i\leq t_0k$, we have
\begin{equation}\label{eq:4.8}
q_{\beta,i}
=\ell_{\beta}+\sum_{j=i+1}^{t_0k}(j-i)\left(\sum_{\alpha\in\Pi_{\bullet}}C_{\beta\alpha}m_{\alpha,j}-n_{\beta,j}\right)+\sum_{j=\left\lfloor\frac{i}{t_0}\right\rfloor+1}^k(t_0j-i)\sum_{\omega\in\Pi_{\circ}}C_{\omega\beta}m_{\omega,j}.
\end{equation}
In order to facilitate the definition of the quantum generating functions that arise from partial summations, we will need to make a few extra definitions below the fold. We first observe that for all $\alpha\in\Pi_{\circ}$ and $0\leq i<k$, the conditions $(\beta,j)\in J_{\mathfrak{g}}^{(0,0)}$ and $t_{\alpha}j>t_{\beta}i$ is equivalent to the condition $(\beta,j)\in J_{\mathfrak{g}}^{(i,0)}$, that is, equation \eqref{eq:4.7} is equivalent to 
\begin{equation}\label{eq:4.9}
q_{\alpha,i}
=\ell_{\alpha}+\sum_{(\beta,j)\in J_{\mathfrak{g}}^{(i,0)}}\frac{C_{\alpha\beta}}{t_{\alpha}}(t_{\alpha}j-t_{\beta}i)m_{\beta,j}-\sum_{j=i+1}^{t_{\alpha}k}(j-i)n_{\alpha,j}.
\end{equation}
This leads to the following definition of the ``intermediate" modified vacancy numbers $q_{\alpha,i}^{(p)}$ for all $\alpha\in\Pi_{\circ}$, $0\leq i<k$ and $0<p<t_0$:
\begin{equation}\label{eq:4.10}
q_{\alpha,i}^{(p)}
=\ell_{\alpha}+\sum_{(\beta,j)\in J_{\mathfrak{g}}^{(i,p)}}\frac{C_{\alpha\beta}}{t_{\alpha}}(t_{\alpha}j-t_{\beta}i)m_{\beta,j}-\sum_{j=i+1}^{t_{\alpha}k}(j-i)n_{\alpha,j}.
\end{equation}
As before, the equation \eqref{eq:4.10} could be written explicitly as follows:
\begin{equation}\label{eq:4.11}
q_{\alpha,i}^{(p)}
=\ell_{\alpha}+\sum_{j=i+1}^k(j-i)\left(\sum_{\beta\in\Pi_{\circ}}C_{\alpha\beta}m_{\beta,j}-n_{\alpha,j}\right)+\sum_{j=t_0i+p+1}^{t_0k}(j-t_0i)\sum_{\omega\in\Pi_{\bullet}}C_{\alpha\omega}m_{\omega,j}.
\end{equation}
As in \cite{DFK14}, our first step is to rewrite $Q_k(\mathbf{m},\mathbf{n})$ in terms of the modified vacancy numbers $q_{\alpha,i}$. To this end, we will first need a few notations. Given any vector $\mathbf{v}=(v_{\alpha,i})_{(\alpha,i)\in J_{\mathfrak{g}}^{(0,0)}}$ of non-negative integers, we let $\mathbf{v}_i^{\circ}=(v_{\alpha,i})_{\alpha\in\Pi_{\circ}}$, $\mathbf{v}_i^{\bullet}=(v_{\alpha,i})_{\alpha\in\Pi_{\bullet}}$ and $\mathbf{v}_i=(v_{\alpha,t_{\alpha}i})_{\alpha\in I_r}$. Also, we let $\ell^{\circ}=(\ell_{\alpha})_{\alpha\in \Pi_{\circ}}$, $\ell^{\bullet}=(\ell_{\bullet})_{\alpha\in \Pi_{\circ}}$ and $\mathbf{q}_i^{\circ,(p)}$ denote the vector $(q_{\alpha,i}^{(p)})_{\alpha\in\Pi_{\circ}}$ for all non-negative integers $i<k$ and $0<p<t_0$.

Next, we let $C^{\circ}=(C_{\alpha\beta})_{\alpha,\beta\in\Pi_{\circ}}$, $C^{\bullet}=(C_{\alpha\beta})_{\alpha,\beta\in\Pi_{\bullet}}$, $D=(C_{\alpha\beta})_{\alpha\in\Pi_{\circ},\beta\in\Pi_{\bullet}}$, $\Lambda^{\circ}=(\Lambda_{\alpha\beta})_{\alpha,\beta\in\Pi_{\circ}}$, $\Lambda^{\bullet}=(\Lambda_{\alpha\beta})_{\alpha,\beta\in\Pi_{\bullet}}$, and $A=(\Lambda_{\alpha\beta})_{\alpha\in\Pi_{\circ},\beta\in\Pi_{\bullet}}$. Then it is easy to see (up to a rearrangement of the rows and columns of $C$ and $\Lambda$ in the case where $\mathfrak{g}$ is of type $C$) that the matrices $C$ and $\Lambda$ has the following block form:
\begin{equation}\label{eq:4.12}
C=
\begin{pmatrix}
C^{\circ} & D\\
t_0D^t & C^{\bullet}
\end{pmatrix},\quad
\Lambda=
\begin{pmatrix}
\Lambda^{\circ} & A\\
t_0A^t & \Lambda^{\bullet}
\end{pmatrix}.
\end{equation}
We are now ready to rewrite $Q_k(\mathbf{m},\mathbf{n})$ in terms of the $q_{\alpha,i}$'s:

\begin{lemma}\label{4.2}
Let $\mathbf{m}=(m_{\alpha,i})_{(\alpha,i)\in J_{\mathfrak{g}}^{(0,0)}}, \mathbf{n}=(n_{\alpha,i})_{(\alpha,i)\in J_{\mathfrak{g}}^{(0,0)}}$ be vectors of nonnegative integers. Then we have
\begin{align}
Q_k(\mathbf{m},\mathbf{n})
&=\frac{1}{2\delta}\left[\sum_{j=1}^k[(\mathbf{q}_{j-1}^{\circ}-\mathbf{q}_j^{\circ})\cdot\Lambda^{\circ}(\mathbf{q}_{j-1}^{\circ}-\mathbf{q}_j^{\circ})+2(\mathbf{q}_{j-1}^{\circ}-\mathbf{q}_j^{\circ})\cdot A(\mathbf{q}_{t_0(j-1)}^{\bullet}-\mathbf{q}_{t_0j}^\bullet)]\right.\nonumber\\
&\quad\qquad\left.+\sum_{j=1}^{t_0k}(\mathbf{q}_{j-1}^{\bullet}-\mathbf{q}_j^{\bullet})\cdot\Lambda^{\bullet}(\mathbf{q}_{j-1}^{\bullet}-\mathbf{q}_j^{\bullet})-L_k(\mathbf{n})-\sum_{j=1}^k U_j\right],\label{eq:4.13}
\end{align}
where 
\begin{align}
L_k(\mathbf{n})
&=\sum_{(\alpha,i),(\beta,j)\in J_{\mathfrak{g}}^{(0,0)}}\frac{\Lambda_{\alpha\beta}}{t_{\alpha}}\min(t_{\alpha}j,t_{\beta}i)n_{\alpha,i}n_{\beta,j},\text{ and}\label{eq:4.14}\\
U_j
&=\frac{1}{t_0}\sum_{i=1}^{t_0}\mathbf{e}_{t_0(j-1)+i}\cdot D^t(\Lambda^{\circ}D\mathbf{e}_{t_0(j-1)+i}-2t_0A\mathbf{f}_{t_0(j-1)+i}),\label{eq:4.15}
\end{align}
with
\begin{align}
\mathbf{e}_{t_0(j-1)+i}
&=\sum_{s=1}^{i-1}s\mathbf{m}_{t_0(j-1)+s}^{\bullet}+\sum_{s=i}^{t_0-1}(s-t_0)\mathbf{m}_{t_0(j-1)+s}^{\bullet},\text{ and}\label{eq:4.16}\\
\mathbf{f}_{t_0(j-1)+i}
&=\sum_{s=i}^{t_0-1}\mathbf{n}_{t_0(j-1)+s}^{\bullet}\label{eq:4.17}
\end{align}
for all $1\leq i\leq t_0$ and $1\leq j\leq k$.
\end{lemma}

\begin{proof}
Firstly, we let
\begin{equation*}
M_{\alpha,i}=\sum_{j=i}^{t_{\alpha}k}m_{\alpha,i},\quad 
N_{\alpha,i}=\sum_{j=i}^{t_{\alpha}k}n_{\alpha,i},
\end{equation*}
for all $\alpha\in I_r$ and $1\leq i\leq t_{\alpha}k$, and $\mathbf{M}=(M_{\alpha,i})_{(\alpha,i)\in J_{\mathfrak{g}}^{(0,0)}}$ and $\mathbf{N}=(N_{\alpha,i})_{(\alpha,i)\in J_{\mathfrak{g}}^{(0,0)}}$. 

Our first step is to write $\mathbf{q}_i^{\circ}$ and $\mathbf{q}_i^{\bullet}$ in terms of $\mathbf{m}_j^{\circ}$, $\mathbf{m}_j^{\bullet}$, $\mathbf{n}_j^{\circ}$ and $\mathbf{n}_j^{\bullet}$ using equations \eqref{eq:4.7} and \eqref{eq:4.8} as follows:
\begin{align*}
\mathbf{q}_i^{\circ}&=\ell^{\circ}+\sum_{j=i+1}^k(j-i)[C^{\circ}\mathbf{m}_j^{\circ}-\mathbf{n}_j^{\circ}]+\sum_{j=t_0i+1}^{t_0k}(j-t_0i)D\mathbf{m}_j^{\bullet};\\
\mathbf{q}_i^{\bullet}&=\ell^{\bullet}+\sum_{j=i+1}^{t_0k}(j-i)[C^{\bullet}\mathbf{m}_j^{\bullet}-\mathbf{n}_j^{\bullet}]+\sum_{j=\left\lfloor\frac{i}{t_0}\right\rfloor+1}^k(t_0j-i)D^t\mathbf{m}_j^{\circ}.
\end{align*}
This implies that for all $1\leq j\leq k$ and $t_0(j-1)+1\leq i\leq t_0j$, we have
\begin{subequations}\label{eq:4.18}
\begin{align}
\mathbf{q}_{j-1}^{\circ}-\mathbf{q}_j^{\circ}
&=C^{\circ}\mathbf{M}_j^{\circ}-\mathbf{N}_j^{\circ}+t_0D\mathbf{M}_i^{\bullet}+D\mathbf{e}_i,\\
\mathbf{q}_{i-1}^{\bullet}-\mathbf{q}_i^{\bullet}
&=C^{\bullet}\mathbf{M}_i^{\bullet}-\mathbf{N}_i^{\bullet}+D^t\mathbf{M}_j^{\circ}.
\end{align}
\end{subequations}
Thus, we have
{\allowdisplaybreaks
\begin{align}
&\qquad(\mathbf{q}_{j-1}^{\circ}-\mathbf{q}_j^{\circ})\cdot\Lambda^{\circ}(\mathbf{q}_{j-1}^{\circ}-\mathbf{q}_j^{\circ})\label{eq:4.19}\\
&=\frac{1}{t_0}\sum_{i=t_0(j-1)+1}^{t_0j}(C^{\circ}\mathbf{M}_j^{\circ}-\mathbf{N}_j^{\circ}+t_0D\mathbf{M}_i^{\bullet}+D\mathbf{e}_i)\cdot\Lambda^{\circ}(C^{\circ}\mathbf{M}_j^{\circ}-\mathbf{N}_j^{\circ}+t_0D\mathbf{M}_i^{\bullet}+D\mathbf{e}_i)\nonumber\\
&=\mathbf{M}_j^{\circ}\cdot C^{\circ}\Lambda^{\circ}(C^{\circ}\mathbf{M}_j^{\circ}-2\mathbf{N}_j^{\circ})+\mathbf{N}_j^{\circ}\cdot\Lambda^{\circ}\mathbf{N}_j^{\circ}+\sum_{i=t_0(j-1)+1}^{t_0j}(t_0D\mathbf{M}_i^{\bullet}+2C^{\circ}\mathbf{M}_j^{\circ})\cdot\Lambda^{\circ}D\mathbf{M}_i^{\bullet}\nonumber\\
&\quad+\frac{1}{t_0}\sum_{i=t_0(j-1)+1}^{t_0j}\left[\mathbf{e}_i\cdot D^t\Lambda^{\circ}\left(2C^{\circ}\mathbf{M}_j^{\circ}-2\mathbf{N}_j^{\circ}+2t_0D\mathbf{M}_i^{\bullet}+D\mathbf{e}_i\right)-2t_0\mathbf{M}_i^{\bullet}\cdot D^t\Lambda^{\circ}\mathbf{N}_j^{\circ}\right],\nonumber\\
&\qquad\sum_{i=t_0(j-1)+1}^{t_0j}(\mathbf{q}_{j-1}^{\bullet}-\mathbf{q}_j^{\bullet})\cdot\Lambda^{\bullet}(\mathbf{q}_{j-1}^{\bullet}-\mathbf{q}_j^{\bullet})\label{eq:4.20}\\
&=\sum_{i=t_0(j-1)+1}^{t_0j}(C^{\bullet}\mathbf{M}_i^{\bullet}-\mathbf{N}_i^{\bullet}+D^t\mathbf{M}_j^{\circ})\cdot\Lambda^{\bullet}(C^{\bullet}\mathbf{M}_i^{\bullet}-\mathbf{N}_i^{\bullet}+D^t\mathbf{M}_j^{\circ})\nonumber\\
&=t_0\mathbf{M}_j^{\circ}\cdot D\Lambda^{\bullet}D^t\mathbf{M}_j^{\circ}+\sum_{i=t_0(j-1)+1}^{t_0j}\left(C^{\bullet}\mathbf{M}_i^{\bullet}-2\mathbf{N}_i^{\bullet}+2D^t\mathbf{M}_j^{\circ}\right)\cdot \Lambda^{\bullet}C^{\bullet}\mathbf{M}_i^{\bullet}\nonumber\\
&\quad+\sum_{i=t_0(j-1)+1}^{t_0j}(\mathbf{N}_i^{\bullet}\cdot\Lambda^{\bullet}\mathbf{N}_i^{\bullet}-2\mathbf{M}_j^{\circ}\cdot D\Lambda^{\bullet}\mathbf{N}_i^{\bullet}),\quad\text{and}\nonumber\\
&\qquad 2(\mathbf{q}_{j-1}^{\circ}-\mathbf{q}_j^{\circ})\cdot A(\mathbf{q}_{t_0(j-1)}^{\bullet}-\mathbf{q}_{t_0j}^\bullet)\label{eq:4.21}\\
&=2\sum_{i=t_0(j-1)+1}^{t_0j}(C^{\circ}\mathbf{M}_j^{\circ}-\mathbf{N}_j^{\circ}+t_0D\mathbf{M}_i^{\bullet}+D\mathbf{e}_i)\cdot A(C^{\bullet}\mathbf{M}_i^{\bullet}-\mathbf{N}_i^{\bullet}+D^t\mathbf{M}_j^{\circ})\nonumber\\
&=2t_0\mathbf{M}_j^{\circ}\cdot DA^t(C^{\circ}\mathbf{M}_j^{\circ}-\mathbf{N}_j^{\circ})+2\sum_{i=t_0(j-1)+1}^{t_0j}\left[t_0\mathbf{M}_i^{\bullet}\cdot D^tA(C^{\bullet}\mathbf{M}_i^{\bullet}-\mathbf{N}_i^{\bullet})+\mathbf{N}_i^{\bullet}\cdot A^t\mathbf{N}_j^{\circ}\right]\nonumber\\
&\quad+2\sum_{i=t_0(j-1)+1}^{t_0j}\left[\mathbf{M}_j^{\circ}\cdot\left(C^{\circ}AC^{\bullet}+t_0DA^tD\right)\mathbf{M}_i^{\bullet}-\mathbf{M}_i^{\bullet}\cdot C^{\bullet}A^t\mathbf{N}_j^{\circ}-\mathbf{M}_j^{\circ}\cdot C^{\circ}A\mathbf{N}_i^{\bullet}\right]\nonumber\\
&\quad+2\sum_{i=t_0(j-1)+1}^{t_0j}\mathbf{e}_i\cdot\left(D^tAD^t\mathbf{M}_j^{\circ}+D^tAC^{\bullet}\mathbf{M}_i^{\bullet}-D^tA\mathbf{N}_i^{\bullet}\right).\nonumber
\end{align}
By using the fact that $C\Lambda=\delta I$ and $C\Lambda C=\delta C$, it follows from equations \eqref{eq:4.19}-\eqref{eq:4.21}, along with the block form \eqref{eq:4.12} of the matrices $C$ and $\Lambda$, that we have the RHS of \eqref{eq:4.13} to be equal to
}
\begin{align}
&\quad\frac{1}{2}\sum_{j=1}^k\left[\mathbf{M}_j^{\circ}\cdot(C^{\circ}\mathbf{M}_j^{\circ}-2\mathbf{N}_j^{\circ})+\sum_{i=t_0(j-1)+1}^{t_0j}(\mathbf{M}_i^{\bullet}\cdot(C^{\bullet}\mathbf{M}_i^{\bullet}-2\mathbf{N}_i^{\bullet})+2\mathbf{M}_j^{\circ}\cdot D\mathbf{M}_i^{\bullet})\right]\label{eq:4.22}\\
&\quad+\frac{1}{2\delta}\sum_{j=1}^k\left[\mathbf{N}_j^{\circ}\cdot\Lambda^{\circ}\mathbf{N}_j^{\circ}+\sum_{i=t_0(j-1)+1}^{t_0j}(\mathbf{N}_i^{\bullet}\cdot\Lambda^{\bullet}\mathbf{N}_i^{\bullet}+2\mathbf{N}_j^{\circ}\cdot A\mathbf{N}_i^{\bullet})\right]-\frac{1}{2\delta}L_k(\mathbf{n})\nonumber\\
&\quad+\frac{1}{\delta t_0}\sum_{j=1}^k\sum_{i=t_0(j-1)+1}^{t_0j}\mathbf{e}_i\cdot D^t(\delta \mathbf{M}_j^{\circ}-\Lambda^{\circ}\mathbf{N}_j^{\circ}-t_0A\mathbf{N}_{t_0j}^{\bullet}).\nonumber
\end{align}
It remains to show that expression \eqref{eq:4.22} is equal to $Q_k(\mathbf{m},\mathbf{n})$. To this end, we first observe from equation \eqref{eq:4.16} that we have 
\begin{equation*}
\sum_{i=t_0(j-1)+1}^{t_0j}\mathbf{e}_i=0 
\end{equation*}
for all $1\leq j\leq k$. As $D^t(\delta \mathbf{M}_j^{\circ}-\Lambda^{\circ}\mathbf{N}_j^{\circ}-t_0A\mathbf{N}_{t_0j}^{\bullet})$ is independent of $i$ for all $t_0(j-1)+1\leq i\leq t_0j$, it follows that we have
\begin{equation}\label{eq:4.23}
\sum_{j=1}^k\sum_{i=t_0(j-1)+1}^{t_0j}\mathbf{e}_i\cdot D^t(\delta \mathbf{M}_j^{\circ}-\Lambda^{\circ}\mathbf{N}_j^{\circ}-t_0A\mathbf{N}_{t_0j}^{\bullet})=0.
\end{equation}
Next, we make use of the fact that
\begin{equation*}
\min(j,t_0j')=\sum_{s=0}^{t_0-1}\min\left(\left\lfloor\frac{j+s}{t_0}\right\rfloor,j'\right)
\end{equation*}
to deduce that
\begin{align}
Q_k(\mathbf{m},\mathbf{n})
&=\frac{1}{2}\sum_{j,j'=1}^k\min(j,j')\mathbf{m}_j^{\circ}\cdot (C^{\circ}\mathbf{m}_{j'}^{\circ}-2\mathbf{n}_{j'}^{\circ})+\frac{1}{2}\sum_{j,j'=1}^{t_0k}\min(j,j')\mathbf{m}_j^{\bullet}\cdot (C^{\bullet}\mathbf{m}_{j'}^{\bullet}-2\mathbf{n}_{j'}^{\bullet})\label{eq:4.24}\\
&\quad+\sum_{j=1}^{t_0k}\sum_{j'=1}^k\min(j,t_0j')\mathbf{m}_{j'}^{\circ}\cdot D\mathbf{m}_j^{\bullet}\nonumber\\
&=\frac{1}{2}\sum_{j,j'=1}^k\sum_{i=1}^{\min(j,j')}\mathbf{m}_j^{\circ}\cdot (C^{\circ}\mathbf{m}_{j'}^{\circ}-2\mathbf{n}_{j'}^{\circ})+\frac{1}{2}\sum_{j,j'=1}^{t_0k}\sum_{i=1}^{\min(j,j')}\mathbf{m}_j^{\bullet}\cdot (C^{\bullet}\mathbf{m}_{j'}^{\bullet}-2\mathbf{n}_{j'}^{\bullet})\nonumber\\
&\quad+\sum_{j=1}^{t_0k}\sum_{j'=1}^k\sum_{s=0}^{t_0-1}\sum_{i=1}^{\min\left(\left\lfloor\frac{j+s}{t_0}\right\rfloor,j'\right)}\mathbf{m}_{j'}^{\circ}\cdot D\mathbf{m}_j^{\bullet}\nonumber\\
&=\frac{1}{2}\sum_{i=1}^k\sum_{j,j'=i}^k\mathbf{m}_j^{\circ}\cdot (C^{\circ}\mathbf{m}_{j'}^{\circ}-2\mathbf{n}_{j'}^{\circ})+\frac{1}{2}\sum_{i=1}^{t_0k}\sum_{j,j'=i}^{t_0k}\mathbf{m}_j^{\bullet}\cdot (C^{\bullet}\mathbf{m}_{j'}^{\bullet}-2\mathbf{n}_{j'}^{\bullet})\nonumber\\
&\quad+\sum_{s=0}^{t_0-1}\sum_{i=1}^k\sum_{j'=i}^k\sum_{j=t_0i-s}^{t_0k}\mathbf{m}_{j'}^{\circ}\cdot D\mathbf{m}_j^{\bullet}\nonumber\\
&=\frac{1}{2}\sum_{j=1}^k\left[\mathbf{M}_j^{\circ}\cdot(C^{\circ}\mathbf{M}_j^{\circ}-2\mathbf{N}_j^{\circ})+\sum_{i=t_0(j-1)+1}^{t_0j}(\mathbf{M}_i^{\bullet}\cdot(C^{\bullet}\mathbf{M}_i^{\bullet}-2\mathbf{N}_i^{\bullet})+2\mathbf{M}_j^{\circ}\cdot D\mathbf{M}_i^{\bullet})\right]\nonumber.
\end{align}
Likewise, by a similar argument, we have
\begin{equation}\label{eq:4.25}
L_k(\mathbf{n})
=\sum_{j=1}^k\left[\mathbf{N}_j^{\circ}\cdot\Lambda^{\circ}\mathbf{N}_j^{\circ}+\sum_{i=t_0(j-1)+1}^{t_0j}(\mathbf{N}_i^{\bullet}\cdot\Lambda^{\bullet}\mathbf{N}_i^{\bullet}+2\mathbf{N}_j^{\circ}\cdot A\mathbf{N}_i^{\bullet})\right].
\end{equation}
By combining equations \eqref{eq:4.23}-\eqref{eq:4.25}, it follows from expression \eqref{eq:4.22} that \eqref{eq:4.16} holds, and we are done. 
\end{proof}

We are now ready to define our quantum generating functions that specialize to the restricted $N$-sums via a constant term evaluation. To make our notations more compact, we will introduce some shorthand notations. For any vector $\mathbf{v}^{\circ}=(v_{\alpha})_{\alpha\in\Pi_{\circ}}$ (respectively $\mathbf{v}^{\bullet}=(v_{\alpha})_{\alpha\in\Pi_{\bullet}}$) of integers and $k\in\mathbb{Z}$, we shall denote the product $\prod_{\alpha\in\Pi_{\circ}}\widehat{Q}_{\alpha,k}^{v_{\alpha}}$ (respectively $\prod_{\alpha\in\Pi_{\bullet}}\widehat{Q}_{\alpha,k}^{v_{\alpha}}$) of the commuting variables $\widehat{Q}_{\alpha,k}^{v_{\alpha}}$ by $\widehat{Q}_{\circ,k}^{\mathbf{v}^{\circ}}$ (respectively $\widehat{Q}_{\bullet,k}^{\mathbf{v}^{\bullet}}$). Likewise, we will also use the shorthand notations $\widehat{Y}_{\circ,k}^{\mathbf{v}^{\circ}}$, $\widehat{Y}_{\bullet,k}^{\mathbf{v}^{\bullet}}$, $\widehat{Z}_{\circ,k}^{\mathbf{v}^{\circ}}$ and $\widehat{Z}_{\bullet,k}^{\mathbf{v}^{\bullet}}$. 

Let us fix the quantum parameter to be $q=\nu^{\delta}$, $u=\nu^{\frac{1}{2}}=q^{\frac{1}{2\delta}}$, and $\mathbb{Z}_u=\mathbb{Z}[u^{\pm1}]$. Let $\widehat{\mathbf{Y}}_{\vec{s}_i}$ denote the quantum torus $\{\widehat{Q}_{\alpha,t_{\alpha}i},\widehat{Q}_{\alpha,t_{\alpha}i+1}\}_{\alpha\in I_r}$, and let $\widehat{\mathbf{Q}}_i=\{\widehat{Q}_{\alpha,i}\}_{\alpha\in I_r}$ for all $i\in\mathbb{Z}$. For any ring $R$ and a set of variables $x=\{x_1,\ldots,x_n\}$, we let $R((x))$ denote the ring of formal Laurent series of the variables $x_1,\ldots,x_n$ with coefficients in $R$. We define the quantum generating function for multiplicities $Z_{\lambda,\mathbf{n}}^{(k)}(\widehat{\mathbf{Y}}_{\vec{s}_0})\in\mathbb{Z}_u[\widehat{\mathbf{Q}}_0^{\pm1}]((\widehat{\mathbf{Q}}_1^{-1}))$ in the quantum torus $\widehat{\mathbf{Y}}_{\vec{s}_0}$, subject to the commutation relations in Lemma \ref{3.5} as follows:
\begin{equation}\label{eq:4.26}
Z_{\lambda,\mathbf{n}}^{(k)}(\widehat{\mathbf{Y}}_{\vec{s}_0})
=\sum_{\mathbf{m}}q^{\overline{Q}_k(\mathbf{m},\mathbf{n})}\prod_{(\alpha,i)\in J_{\mathfrak{g}}^{(0,0)}}\begin{bmatrix}m_{\alpha,i}+q_{\alpha,i}\\m_{\alpha,i}\end{bmatrix}_q\widehat{Q}_{\bullet,1}^{-\mathbf{q}_0^{\bullet}}\widehat{Q}_{\circ,1}^{-\mathbf{q}_0^{\circ}}\widehat{Q}_{\bullet,0}^{\mathbf{q}_1^{\bullet}}\widehat{Q}_{\circ,0}^{\mathbf{q}_1^{\circ}+D\mathbf{e}_1}.
\end{equation}
Here, the sum is over all vectors $\mathbf{m}=(m_{\alpha,i})_{(\alpha,i)\in J_{\mathfrak{g}}^{(0,0)}}$ of non-negative integers, and the modified quadratic form $\overline{Q}_k(\mathbf{m},\mathbf{n})$ is defined by
\begin{align}
\quad\overline{Q}_k(\mathbf{m},\mathbf{n})
&=\frac{1}{2\delta}\left[\sum_{j=1}^{k-1}[(\mathbf{q}_j^{\circ}-\mathbf{q}_{j+1}^{\circ})\cdot\Lambda^{\circ}(\mathbf{q}_j^{\circ}-\mathbf{q}_{j+1}^{\circ})+2(\mathbf{q}_j^{\circ}-\mathbf{q}_{j+1}^{\circ})\cdot A(\mathbf{q}_{t_0j}^{\bullet}-\mathbf{q}_{t_0(j+1)}^\bullet)]\right.\label{eq:4.27}\\
&\qquad\quad+\sum_{j=1}^{t_0k-1}(\mathbf{q}_j^{\bullet}-\mathbf{q}_{j+1}^{\bullet})\cdot\Lambda^{\bullet}(\mathbf{q}_j^{\bullet}-\mathbf{q}_{j+1}^{\bullet})-\sum_{j=1}^k U_j+\mathbf{q}_1^{\circ}\cdot\Lambda^{\circ}\mathbf{q}_1^{\circ}+\mathbf{q}_1^{\bullet}\cdot\Lambda^{\bullet}\mathbf{q}_1^{\bullet}\nonumber\\
&\qquad\quad\left.-2(\mathbf{q}_0^{\circ}-\mathbf{q}_1^{\circ})\cdot A\mathbf{q}_{t_0}^{\bullet}+2t_0\mathbf{q}_0^{\circ}\cdot A\mathbf{q}_1^{\bullet}+2(\Lambda^{\circ}\mathbf{q}_0^{\circ}+A\mathbf{q}_0^{\bullet})\cdot D\mathbf{e}_1\vphantom{\sum_{i=1}}\right].\nonumber
\end{align}
From equation \eqref{eq:4.27}, we have
\begin{align*}
\overline{Q}_k(\mathbf{m},\mathbf{n})-Q_k(\mathbf{m},\mathbf{n})
&=-\frac{1}{2\delta}\left[\vphantom{\sum_{i=1}}\mathbf{q}_0^{\circ}\cdot\Lambda^{\circ}(\mathbf{q}_0^{\circ}-2\mathbf{q}_1^{\circ})+\mathbf{q}_0^{\bullet}\cdot\Lambda^{\bullet}(\mathbf{q}_0^{\bullet}-2\mathbf{q}_1^{\bullet})+2(\mathbf{q}_0^{\circ}-\mathbf{q}_1^{\circ})\cdot A\mathbf{q}_0^{\bullet}\right.\\
&\qquad\qquad\left.-2t_0\mathbf{q}_0^{\circ}\cdot A\mathbf{q}_1^{\bullet}-2(\Lambda^{\circ}\mathbf{q}_0^{\circ}+A\mathbf{q}_0^{\bullet})\cdot D\mathbf{e}_1-L_k({\mathbf{n}})\vphantom{\sum_{i=1}}\right].
\end{align*}
This implies that the modified quadratic form $\overline{Q}_k(\mathbf{m},\mathbf{n})$ is equal to $Q_k(\mathbf{m},\mathbf{n})+\frac{1}{2\delta}L_k({\mathbf{n}})$ when $\mathbf{q}_0^{\circ}=\mathbf{0}$ and $\mathbf{q}_0^{\bullet}=\mathbf{0}$.

In addition to the generating function $Z_{\lambda,\mathbf{n}}^{(k)}(\widehat{\mathbf{Y}}_{\vec{s}_0})$ that we have defined above, we would also need to define intermediate quantum generating functions $Z_{\lambda,\mathbf{n}}^{(k,p)}(\widehat{\mathbf{Y}}_{\vec{s}_{0,p}})$, where $0<p<t_0$, that arise from partial summations over $m_{\alpha,i}$, with $\alpha\in\Pi_{\bullet}$ and $t_0\not\divides i$. To this end, we will need to define truncations of a given vector $\mathbf{v}=(v_{\alpha,i})_{(\alpha,i)\in J_{\mathfrak{g}}^{(0,0)}}$. For any vector $\mathbf{v}=(v_{\alpha,i})_{(\alpha,i)\in J_{\mathfrak{g}}^{(0,0)}}$, $0\leq j\leq k$ and $0\leq p<t_0$, we define $\mathbf{v}^{(j,p)}:=(v_{\alpha,i})_{(\alpha,i)\in J_{\mathfrak{g}}^{(j,p)}}$. In particular, we have $\mathbf{v}^{(0,0)}=\mathbf{v}$.

We shall first define the quantum generating function $Z_{\lambda,\mathbf{n}}^{(k,t_0-1)}(\widehat{\mathbf{Y}}_{\vec{s}_{0,t_0-1}})$. To this end, we let $\vec{s}_{i,t_0-1}=((i+1)t_{\alpha}-1)_{\alpha\in I_r}$ for all $i\in\mathbb{Z}$. Then it is easy to verify that $\vec{s}_{i,t_0-1}$ is a generalized Motzkin path for all $i\in\mathbb{Z}$, so the components of the vector 
\begin{equation*}
\widehat{\mathbf{Y}}_{\vec{s}_{i,t_0-1}}=(\widehat{Q}_{\alpha,(i+1)t_{\alpha}-1},\widehat{Q}_{\alpha,(i+1)t_{\alpha}})_{\alpha\in I_r}
\end{equation*} 
form a valid set of initial data for the quantum $Q$-system. Hence, we define the quantum generating function $Z_{\lambda,\mathbf{n}}^{(k,t_0-1)}(\widehat{\mathbf{Y}}_{\vec{s}_{0,t_0-1}})$ in the quantum torus $\widehat{\mathbf{Y}}_{\vec{s}_{0,t_0-1}}$ as follows:
\begin{align}
&\quad Z_{\lambda,\mathbf{n}}^{(k,t_0-1)}(\widehat{\mathbf{Y}}_{\vec{s}_{0,t_0-1}})\label{eq:4.28}\\
&=\sum_{\mathbf{m}^{(0,t_0-1)}}q^{\overline{Q}_k^{(t_0-1)}(\mathbf{m},\mathbf{n})}\prod_{(\alpha,i)\in J_{\mathfrak{g}}^{(0,t_0-1)}}\begin{bmatrix}m_{\alpha,i}+q_{\alpha,i}\\m_{\alpha,i}\end{bmatrix}_q\widehat{Q}_{\bullet,t_0}^{-\mathbf{q}_{t_0-1}^{\bullet}}\widehat{Q}_{\circ,1}^{-\mathbf{q}_0^{\circ,(t_0-1)}}\widehat{Q}_{\bullet,t_0-1}^{\mathbf{q}_{t_0}^{\bullet}}\widehat{Q}_{\circ,0}^{\mathbf{q}_1^{\circ}}.\nonumber
\end{align}
Here, the sum is over all vectors $\mathbf{m}^{(0,t_0-1)}=(m_{\alpha,i})_{(\alpha,i)\in J_{\mathfrak{g}}^{(0,t_0-1)}}$ of non-negative integers, and the quadratic form $\overline{Q}_k^{(t_0-1)}(\mathbf{m},\mathbf{n})$ is defined by
\begin{align}
&\qquad\overline{Q}_k^{(t_0-1)}(\mathbf{m},\mathbf{n})\label{eq:4.29}\\
&=\frac{1}{2\delta}\left[\sum_{j=1}^{k-1}[(\mathbf{q}_j^{\circ}-\mathbf{q}_{j+1}^{\circ})\cdot\Lambda^{\circ}(\mathbf{q}_j^{\circ}-\mathbf{q}_{j+1}^{\circ})+2(\mathbf{q}_j^{\circ}-\mathbf{q}_{j+1}^{\circ})\cdot A(\mathbf{q}_{t_0j}^{\bullet}-\mathbf{q}_{t_0(j+1)}^\bullet)-U_{j+1}]\right.\nonumber\\
&\qquad\left.+\sum_{j=t_0}^{t_0k-1}(\mathbf{q}_j^{\bullet}-\mathbf{q}_{j+1}^{\bullet})\cdot\Lambda^{\bullet}(\mathbf{q}_j^{\bullet}-\mathbf{q}_{j+1}^{\bullet})+\mathbf{q}_1^{\circ}\cdot\Lambda^{\circ}\mathbf{q}_1^{\circ}+\mathbf{q}_{t_0}^{\bullet}\cdot\Lambda^{\bullet}\mathbf{q}_{t_0}^{\bullet}+2\mathbf{q}_1^{\circ}\cdot A\mathbf{q}_{t_0}^{\bullet}\right].\nonumber
\end{align}
One important property of the quadratic form $\overline{Q}_k^{(t_0-1)}(\mathbf{m},\mathbf{n})$ is that it is independent of $m_{\alpha,i}$ and $n_{\alpha,i}$ for all $\alpha\in I_r$ and $1\leq i\leq t_{\alpha}$.

If $t_0=2$, that is, $\mathfrak{g}$ is of type $BCF$, then we are done with defining the intermediate quantum generating functions. Else, if $t_0=3$, that is, $\mathfrak{g}$ is of type $G_2$, then we need to define the intermediate quantum generating function $Z_{\lambda,\mathbf{n}}^{(k,1)}(\widehat{\mathbf{Y}}_{\vec{s}_{0,1}})$. We let $\vec{s}_{i,1}=(i,3i+1)$ for all $i\in\mathbb{Z}$. Then it is easy to verify that $\vec{s}_{i,1}$ is a generalized Motzkin path for all $i\in\mathbb{Z}$, so the components of the vector 
\begin{equation*}
\widehat{\mathbf{Y}}_{\vec{s}_{i,1}}=(\widehat{Q}_{1,i},\widehat{Q}_{2,3i+1},\widehat{Q}_{1,i+1},\widehat{Q}_{2,3i+2})
\end{equation*} 
form a valid set of initial data for the quantum $Q$-system. Hence, we define the quantum generating function $Z_{\lambda,\mathbf{n}}^{(k,1)}(\widehat{\mathbf{Y}}_{\vec{s}_{0,1}})$ in the quantum torus $\widehat{\mathbf{Y}}_{\vec{s}_{0,1}}$ as follows:
\begin{equation}\label{eq:4.30}
Z_{\lambda,\mathbf{n}}^{(k,1)}(\widehat{\mathbf{Y}}_{\vec{s}_{0,1}})
=\sum_{\mathbf{m}^{(0,1)}}q^{\overline{Q}_k^{(1)}(\mathbf{m},\mathbf{n})}\prod_{(\alpha,i)\in J_{\mathfrak{g}}^{(0,1)}}\begin{bmatrix}m_{\alpha,i}+q_{\alpha,i}\\m_{\alpha,i}\end{bmatrix}_q\widehat{Q}_{2,2}^{-q_{2,1}}\widehat{Q}_{1,1}^{-q_{1,0}^{(1)}}\widehat{Q}_{2,1}^{q_{2,2}}\widehat{Q}_{1,0}^{q_{1,1}+m_{2,2}}.
\end{equation}
Here, the sum is over all vectors $\mathbf{m}^{(0,1)}=(m_{\alpha,i})_{(\alpha,i)\in J_{\mathfrak{g}}^{(0,1)}}$ of non-negative integers, and the quadratic form $\overline{Q}_k^{(1)}(\mathbf{m},\mathbf{n})$ is defined by
\begin{align}
&\qquad\overline{Q}_k^{(1)}(\mathbf{m},\mathbf{n})\label{eq:4.31}\\
&=\sum_{j=1}^{k-1}\left[(q_{1,j}-q_{1,j+1})^2+(q_{1,j}-q_{1,j+1})(q_{2,3j}-q_{2,3j+3})\right]-\frac{1}{2}\sum_{j=1}^{k-1}U_{j+1}+\sum_{j=2}^{3k-1}(q_{2,j}-q_{2,j+1})^2\nonumber\\
&\quad+q_{1,1}^2+q_{2,2}^2+q_{1,1}q_{2,3}+q_{1,0}^{(1)}(q_{2,1}+n_{2,2})+m_{2,2}(2q_{2,1}-q_{2,2}-2m_{2,2}+2n_{2,2}).\nonumber
\end{align}
Here, we note that in the case of type $G_2$, we have $\delta=1$. It is easy to see that the quadratic form $\overline{Q}_k^{(1)}(\mathbf{m},\mathbf{n})$ is independent of $m_{2,2}$.

The generating function $Z_{\lambda,\mathbf{n}}^{(k)}(\widehat{\mathbf{Y}}_{\vec{s}_0})$ is related to the restricted $N$-sum $N_{\lambda,\mathbf{n}}^{(k)}(q^{-1})$ via a constant term and an evaluation. For any $f\in\mathbb{Z}_u[\widehat{\mathbf{Q}}_0^{\pm1}]((\widehat{\mathbf{Q}}_1^{-1}))$, we define the multiple constant term $\mathrm{CT}_{\{\widehat{Q}_{\alpha,1}\}_{\alpha\in I_r}}(f)$ to be the term of total degree $0$ in each of the variables $\widehat{Q}_{1,1},\cdots,\widehat{Q}_{r,1}$. In particular, if we have a normal-ordered expansion
\begin{equation*}
f=\sum_{\mathbf{a}^{\bullet},\mathbf{b}^{\bullet}\in\mathbb{Z}^s,\mathbf{a}^{\circ},\mathbf{b}^{\circ}\in\mathbb{Z}^d}f_{\mathbf{a}^{\bullet},\mathbf{a}^{\circ},\mathbf{b}^{\bullet},\mathbf{b}^{\circ}}\widehat{Q}_{\bullet,0}^{\mathbf{a}^{\bullet}}\widehat{Q}_{\circ,0}^{\mathbf{a}^{\circ}}\widehat{Q}_{\circ,1}^{\mathbf{b}^{\circ}}\widehat{Q}_{\bullet,1}^{\mathbf{b}^{\bullet}},
\end{equation*}
where $s=|\Pi_{\bullet}|,d=|\Pi_{\circ}|$, then we have
\begin{equation*}
\mathrm{CT}_{\widehat{\mathbf{Q}}_1}(f)=\sum_{\mathbf{a}^{\bullet}\in\mathbb{Z}^s,\mathbf{a}^{\circ}\in\mathbb{Z}^d}f_{\mathbf{a}^{\bullet},\mathbf{a}^{\circ},\mathbf{0},\mathbf{0}}\widehat{Q}_{\bullet,0}^{\mathbf{a}^{\bullet}}\widehat{Q}_{\circ,0}^{\mathbf{a}^{\circ}},
\end{equation*}
where $\widehat{\mathbf{Q}}_k=\{\widehat{Q}_{\alpha,k}\}_{\alpha\in I_r}$ for all $k\in\mathbb{Z}$. Likewise, we define the multiple evaluation of $f$ at $\widehat{Q}_{1,0},\cdots,\widehat{Q}_{r,0}=1$ to be the following Laurent series
\begin{equation*}
f|_{\widehat{\mathbf{Q}}_0=1}=\sum_{\mathbf{a}^{\bullet},\mathbf{b}^{\bullet}\in\mathbb{Z}^s,\mathbf{a}^{\circ},\mathbf{b}^{\circ}\in\mathbb{Z}^d}f_{\mathbf{a}^{\bullet},\mathbf{a}^{\circ},\mathbf{b}^{\bullet},\mathbf{b}^{\circ}}\widehat{Q}_{\circ,1}^{\mathbf{b}^{\circ}}\widehat{Q}_{\bullet,1}^{\mathbf{b}^{\bullet}}.
\end{equation*}
The constant term and evaluation maps commute, and their composition gives:
\begin{equation}\label{eq:4.32}
\mathrm{CT}_{\widehat{\mathbf{Q}}_1}(f)|_{\widehat{\mathbf{Q}}_0=1}=\sum_{\mathbf{a}^{\bullet}\in\mathbb{Z}^s,\mathbf{a}^{\circ}\in\mathbb{Z}^d}f_{\mathbf{a}^{\bullet},\mathbf{a}^{\circ},\mathbf{0},\mathbf{0}}.
\end{equation}
To simplify our notation, we will denote the LHS of equation \eqref{eq:4.32} by $\phi(f)$.

We may now express the $N$-sum in terms of $Z_{\lambda,\mathbf{n}}^{(k)}(\widehat{\mathbf{Y}}_{\vec{s}_0})$ as follows:
\begin{equation}\label{eq:4.33}
N_{\lambda,\mathbf{n}}^{(k)}(q^{-1})
=q^{-\frac{1}{2\delta}L_k(\mathbf{n})}\phi(Z_{\lambda,\mathbf{n}}^{(k)}(\widehat{\mathbf{Y}}_{\vec{s}_0})),
\end{equation} 
where the constant term ensures the condition that $\mathbf{q}_0^{\circ}=\mathbf{0}$ and $\mathbf{q}_0^{\bullet}=\mathbf{0}$, and the result agrees with the definition of the restricted $N$-sum $N_{\lambda,\mathbf{n}}^{(k)}(q^{-1})$, as $Q_k(\mathbf{m},\mathbf{n})=\overline{Q}_k(\mathbf{m},\mathbf{n})-\frac{1}{2\delta}L_k(\mathbf{n})$ when $\mathbf{q}_0^{\circ}=\mathbf{0}$ and $\mathbf{q}_0^{\bullet}=\mathbf{0}$.

\subsection{Factorization properties of the quantum generating functions}

In this subsection, we will state how the quantum generating function $Z_{\lambda,\mathbf{n}}^{(k)}(\widehat{\mathbf{Y}}_{\vec{s}_0})$ fully factorizes (up to a scalar constant dependent on $\mathbf{n}$ and $\lambda$) into a product of the quantum $Q$-system variables $\widehat{Q}_{\alpha,k}$ and their inverses. In order to make our notations more compact, we will denote the product $y_1y_2\cdots y_j$ of variables $y_1,y_2,\ldots,y_j$ by $\prod_{i=1}^j y_i$. Next, we define
\begin{equation}\label{eq:4.34}
\widehat{P}_{\mathbf{v},j}=\prod_{i=1}^j\left(\left(\prod_{j'=1}^{t_0}\widehat{Q}_{\bullet,t_0(i-1)+j'}^{\mathbf{v}_{t_0(i-1)+j'}^{\bullet}}\right)\widehat{Q}_{\circ,i}^{\mathbf{v}_i^{\circ}}\right),
\end{equation}
and $\widehat{P}_{\mathbf{v},0}:=1$ for all vectors $\mathbf{v}=(v_{\alpha,i})_{\alpha\in I_r,i\in\mathbb{N}}$ of integers and all positive integers $j$. 

When $k=1$, we have the following full factorization of $Z_{\lambda,\mathbf{n}}^{(1)}(\widehat{\mathbf{Y}}_{\vec{s}_0})$ into a product of the quantum $Q$-system variables $\widehat{Q}_{\alpha,k}$ and their inverses:

\begin{lemma}\label{4.3}
The quantum generating function $Z_{\lambda,\mathbf{n}}^{(1)}(\widehat{\mathbf{Y}}_{\vec{s}_0})$ can be expressed as a product of $\widehat{Q}_{\alpha,i}$ ($\alpha\in I_r$, $0\leq i\leq t_{\alpha}+1$) as follows:
\begin{equation*}
Z_{\lambda,\mathbf{n}}^{(1)}(\widehat{\mathbf{Y}}_{\vec{s}_0})
=q^{-\frac{1}{2\delta}\sum_{\alpha,\beta\in I_r}\left(\Lambda_{\alpha\alpha}\ell_{\alpha}+2\sum_{i=1}^{t_{\alpha}}\Lambda_{\alpha\beta}n_{\alpha,i}\right)}\widehat{Z}_{\bullet,0}^{-1}\widehat{Z}_{\circ,0}^{-1}\widehat{P}_{\mathbf{n},1}\widehat{Z}_{\circ,1}^{\ell^{\circ}+1}\widehat{Z}_{\bullet,t_0}^{\ell^{\bullet}+1},
\end{equation*}
where $\ell^{\circ}+1$ (respectively $\ell^{\bullet}+1$) is the vector $(\ell_{\alpha}+1)_{\alpha\in\Pi_{\circ}}$ (respectively $(\ell_{\alpha}+1)_{\alpha\in\Pi_{\bullet}}$).
\end{lemma}

In order to prove Lemma \ref{4.3}, we will follow the strategy as in the proof of \cite[Lemma 5.3]{DFK08}, with the appropriate modifications in order to take into account the effects of quantum commutativity. As the form of the factorization formulas in Lemma \ref{4.3} suggests, we would need to sum over $m_{\alpha,k}$ for all $\alpha\in I_r$ and $1\leq k\leq t_{\alpha}$, and use the quantum $Q$-system relations to simplify the expressions. We will first sum over $m_{\alpha,p}$ for all $\alpha\in\Pi_{\bullet}$ and $0<p<t_0$, as $q_{\gamma',p}$ depends on $m_{\gamma,1}$ for all $0<p<t_0$. This leads to the following partial factorization of $Z_{\lambda,\mathbf{n}}^{(k)}(\widehat{\mathbf{Y}}_{\vec{s}_0})$ as a product of the quantum $Q$-system variables $\widehat{Q}_{\alpha,k}$ (along with their inverses), and its intermediate generating function $Z_{\lambda,\mathbf{n}}^{(k,p)}(\widehat{\mathbf{Y}}_{\vec{s}_0})$, where $0<p<t_0$:

\begin{lemma}\label{4.4}
The quantum generating function $Z_{\lambda,\mathbf{n}}^{(k)}(\widehat{\mathbf{Y}}_{\vec{s}_0})$ has the following partial factorization in terms of $Z_{\lambda,\mathbf{n}}^{(k,p)}(\widehat{\mathbf{Y}}_{\vec{s}_{0,p}})$ ($0<p<t_0$) and $\widehat{Q}_{\alpha,i}$ ($\alpha\in\Pi_{\bullet}$, $0\leq i\leq p+1$) as follows:
\begin{equation*}
Z_{\lambda,\mathbf{n}}^{(k)}(\mathbf{Y}_{\vec{s}_0})
=q^{-\frac{1}{\delta}\sum_{\alpha,\beta\in\Pi_{\bullet}}\sum_{i=1}^p\Lambda_{\alpha\beta}n_{\alpha,i}}\widehat{Z}_{\bullet,0}^{-1}\left(\widehat{Q}_{\bullet,1}^{\mathbf{n}_1^{\bullet}}\cdots\widehat{Q}_{\bullet,p}^{\mathbf{n}_p^{\bullet}}\right)\widehat{Z}_{\bullet,p} Z_{\lambda,\mathbf{n}}^{(k,p)}(\widehat{\mathbf{Y}}_{\vec{s}_{0,p}}).
\end{equation*}
\end{lemma}

Subsequently, we will then sum over $m_{\alpha,t_{\alpha}}$ for all $\alpha\in I_r$ to obtain the full factorization formulas in Lemma \ref{4.3}. As the proofs of Lemmas \ref{4.3} and \ref{4.4} are by and large similar, we will only prove Lemma \ref{4.3} in this subsection, leaving the details of the proof of Lemma \ref{4.4} to Appendix \ref{Appendix A}. In order to prove Lemma \ref{4.3}, we will first need to compute $\widehat{Y}_{\alpha,t_{\alpha}}$ explicitly for all $\alpha\in I_r$. It is easy to see from Lemma \ref{3.5} that we have
\begin{equation}\label{eq:4.35}
\widehat{Y}_{\alpha,t_{\alpha}}=\prod_{\beta\in I_r}\widehat{Q}_{\beta,t_{\beta}}^{-C_{\alpha\beta}}.
\end{equation}

Next, we need the following technical lemmas that will be helpful in proving Lemmas \ref{4.3} and \ref{4.4}:

\begin{lemma}\label{4.5}
For any $\alpha\in I_r$ and $i\in\mathbb{Z}$, we have
\begin{equation*}
\widehat{Z}_{\alpha,i}(1-\widehat{Y}_{\alpha,i+1})^{-1}=\widehat{Z}_{\alpha,i+1}.
\end{equation*}
\end{lemma}

\begin{proof}
By Lemma \ref{3.5} and \eqref{eq:3.22}, we have
\begin{equation*}
1-\widehat{Y}_{\alpha,i+1}
=\nu^{-\Lambda_{\alpha\alpha}}\widehat{Q}_{\alpha,i+1}^{-2}\widehat{Q}_{\alpha,i+2}\widehat{Q}_{\alpha,i}
=\widehat{Z}_{\alpha,i+1}^{-1}\widehat{Z}_{\alpha,i},
\end{equation*}
or equivalently, $(1-\widehat{Y}_{\alpha,i+1})^{-1}=\widehat{Z}_{\alpha,i}^{-1}\widehat{Z}_{\alpha,i+1}$. The desired statement now follows.
\end{proof}

\begin{lemma}\label{4.6}
For any $\alpha\in I_r$ and $i,b\in\mathbb{Z}$, we have 
\begin{equation*}
\sum_{a\geq0}\begin{bmatrix}a+b\\a\end{bmatrix}_q\widehat{Y}_{\alpha,i+1}^a\widehat{Z}_{\alpha,i}^b=\widehat{Z}_{\alpha,i}^{-1}\widehat{Z}_{\alpha,i+1}^{b+1}.
\end{equation*}
\end{lemma}

\begin{proof}
Firstly, we note that 
\begin{equation*}
\sum_{a=0}^{\infty}\begin{bmatrix}a+b\\a\end{bmatrix}_vx^a=\frac{(v^{b+1}x;v)_{\infty}}{(x;v)_{\infty}}
=
\begin{cases}
\prod_{i=0}^b(1-v^ix)^{-1} & \text{if }b\geq0\\
\prod_{i=0}^{-b-1}(1-v^{i+b+1}x) & \text{if }b<0.
\end{cases}
\end{equation*}
This implies that for any integer $b$ and any two variables $x,y$ satisfying the quasi-commutation relation $yx=vxy$, we have
\begin{equation*}
\sum_{a\geq0}\begin{bmatrix}a+b\\a\end{bmatrix}_vx^ay^b=y^{-1}(y(1-x)^{-1})^{b+1},
\end{equation*}
where the right hand side is considered as a formal power series in the variable $x$. 

Now, it follows from Lemma \ref{3.5} that we have $\widehat{Z}_{\alpha,i}\widehat{Y}_{\alpha,i+1}=\nu^{\delta}\widehat{Y}_{\alpha,i+1}\widehat{Z}_{\alpha,i}=q\widehat{Y}_{\alpha,i+1}\widehat{Z}_{\alpha,i}$. In particular, by setting $v=q$, $x=\widehat{Y}_{\alpha,i+1}$ and $y=\widehat{Z}_{\alpha,i}$, we see that the desired statement follows from the above equation, along with Lemma \ref{4.5}.
\end{proof}

\begin{proof}[Proof of Lemma \ref{4.3}]
For convenience, let us set $p=t_0-1$. When $k=1$, we see that $q_{\alpha,t_{\alpha}}=\ell_{\alpha}$ for all $\alpha\in I_r$. Thus, it follows from equations \eqref{eq:4.8} and \eqref{eq:4.11} that we have
\begin{align}
\mathbf{q}_p^{\bullet}
&=\ell^{\bullet}+C^{\bullet}\mathbf{m}_{t_0}^{\bullet}-\mathbf{n}_{t_0}^{\bullet}+D^t\mathbf{m}_1^{\circ},\label{eq:4.36}\\
\mathbf{q}_0^{\circ,(p)}
&=\ell^{\circ}+C^{\circ}\mathbf{m}_1^{\circ}-\mathbf{n}_1^{\circ}+t_0D\mathbf{m}_{t_0}^{\bullet}.\label{eq:4.37}
\end{align}
Next, we have
\begin{equation}\label{eq:4.38}
\overline{Q}^{(p)}(\mathbf{m},\mathbf{n})=\frac{1}{2\delta}\left(\ell^{\circ}\cdot\Lambda^{\circ}\ell^{\circ}+\ell^{\bullet}\cdot\Lambda^{\bullet}\ell^{\bullet}+2\ell^{\bullet}\cdot A^t\ell^{\circ}\right).
\end{equation}
We now use Lemmas \ref{3.5}, along with equations \eqref{eq:4.35}, \eqref{eq:4.36} and \eqref{eq:4.37} to rearrange the terms involving the $\widehat{Q}_{\alpha,i}$'s in the quantum generating function $Z_{\lambda,\mathbf{n}}^{(1,p)}(\widehat{\mathbf{Y}}_{\vec{s}_{0,p}})$ as follows:
\begin{equation}\label{eq:4.39}
\widehat{Q}_{\bullet,t_0}^{-\mathbf{q}_p^{\bullet}}\widehat{Q}_{\circ,1}^{-\mathbf{q}_0^{\circ,(p)}}\widehat{Q}_{\bullet,p}^{\mathbf{q}_{t_0}^{\bullet}}\widehat{Q}_{\circ,0}^{\mathbf{q}_1^{\circ}}
=\widehat{Q}_{\bullet,t_0}^{-\mathbf{q}_p^{\bullet}}\widehat{Q}_{\circ,1}^{-\mathbf{q}_0^{\circ,(p)}}\widehat{Q}_{\bullet,p}^{\ell^{\bullet}}\widehat{Q}_{\circ,0}^{\ell^{\circ}}
=\widehat{Q}_{\bullet,1}^{\mathbf{n}_{t_0}^{\bullet}}\widehat{Q}_{\circ,1}^{\mathbf{n}_1^{\circ}}\widehat{Y}_{\bullet,t_0}^{\mathbf{m}_{t_0}^{\bullet}}\widehat{Y}_{\circ,1}^{\mathbf{m}_1^{\circ}}\widehat{Q}_{\bullet,t_0}^{-\ell^{\bullet}}\widehat{Q}_{\circ,1}^{-\ell^{\circ}}\widehat{Q}_{\bullet,p}^{\ell^{\bullet}}\widehat{Q}_{\circ,0}^{\ell^{\circ}}.
\end{equation}
Next, we let $\Pi_{\bullet}=\{\alpha_1,\alpha_2,\cdots,\alpha_s\}$ and $\Pi_{\circ}=\{\beta_1,\beta_2,\cdots,\beta_d\}$. We use Lemma \ref{3.5} and equation \eqref{eq:4.38} to deduce that
\begin{align}\label{eq:4.40}
\widehat{Q}_{\bullet,t_0}^{-\ell^{\bullet}}\widehat{Q}_{\circ,1}^{-\ell^{\circ}}\widehat{Q}_{\bullet,p}^{\ell^{\bullet}}\widehat{Q}_{\circ,0}^{\ell^{\circ}}
&=q^{-\overline{Q}^{(p)}(\mathbf{m},\mathbf{n})-\frac{1}{2\delta}\sum_{\alpha\in I_r}\Lambda_{\alpha\alpha}\ell_{\alpha}^2}\prod_{i=1}^s\left(\widehat{Q}_{\alpha_i,p}^{\ell_{\alpha_i}}\widehat{Q}_{\alpha_i,t_0}^{-\ell_{\alpha_i}}\right)\prod_{j=1}^d\left(\widehat{Q}_{\beta_j,0}^{\ell_{\beta_j}}\widehat{Q}_{\beta_j,1}^{-\ell_{\beta_j}}\right)\nonumber\\
&=q^{-\overline{Q}^{(p)}(\mathbf{m},\mathbf{n})-\frac{1}{2\delta}\sum_{\alpha\in I_r}\Lambda_{\alpha\alpha}\ell_{\alpha}}\widehat{Z}_{\bullet,p}^{\ell^{\bullet}}\widehat{Z}_{\circ,0}^{\ell^{\circ}}.
\end{align}
Finally, we use Lemma \ref{3.6}(6) and (7) to deduce that
\begin{equation}\label{eq:4.41}
\widehat{Y}_{\bullet,t_0}^{\mathbf{m}_{t_0}^{\bullet}}\widehat{Y}_{\circ,1}^{\mathbf{m}_1^{\circ}}\widehat{Z}_{\bullet,p}^{\ell^{\bullet}}\widehat{Z}_{\circ,0}^{\ell^{\circ}}
=\prod_{i=1}^s\left(\widehat{Y}_{\alpha_i,t_0}^{m_{\alpha_i,t_0}}\widehat{Z}_{\alpha_i,p}^{\ell_{\alpha_i}}\right)\prod_{j=1}^d\left(\widehat{Y}_{\beta_j,1}^{m_{\beta_j,1}}\widehat{Z}_{\beta_j,0}^{\ell_{\beta_j}}\right).
\end{equation}
We now use equations \eqref{eq:4.39}-\eqref{eq:4.41} to deduce that 
\begin{align}
Z_{\lambda,\mathbf{n}}^{(1,p)}(\widehat{\mathbf{Y}}_{\vec{s}_{0,p}})
&=q^{-\frac{1}{2\delta}\sum_{\alpha\in I_r}\Lambda_{\alpha\alpha}\ell_{\alpha}}\widehat{Q}_{\bullet,1}^{\mathbf{n}_{t_0}^{\bullet}}\widehat{Q}_{\circ,1}^{\mathbf{n}_1^{\circ}}\left(\prod_{i=1}^s\sum_{m_{\alpha_i,t_0}\geq0}\begin{bmatrix}m_{\alpha_i,t_0}+\ell_{\alpha_i}\\m_{\alpha_i,t_0}\end{bmatrix}_q\widehat{Y}_{\alpha_i,t_0}^{m_{\alpha_i,t_0}}\widehat{Z}_{\alpha_i,p}^{\ell_{\alpha_i}}\right)\label{eq:4.42}\\
&\qquad\times\left(\prod_{j=1}^d\sum_{m_{\beta_j,1}\geq0}\begin{bmatrix}m_{\beta_j,1}+\ell_{\beta_j}\\m_{\beta_j,1}\end{bmatrix}_q\widehat{Y}_{\beta_j,1}^{m_{\beta_j,1}}\widehat{Z}_{\beta_j,0}^{\ell_{\beta_j}}\right).\nonumber
\end{align}
As $q^{-\frac{1}{2\delta}\sum_{\alpha\in I_r}\Lambda_{\alpha\alpha}\ell_{\alpha}}$ is a constant, we may sum over each $m_{\alpha_i,t_0}$ and $m_{\beta_j,1}$ for all $i=1,\cdots,s$ and $j=1,\cdots,d$, and apply Lemmas \ref{3.5}, \ref{3.6}(1), (2) and \ref{4.6}, to rewrite the RHS of equation \eqref{eq:4.42} as 
\begin{align}\label{eq:4.43}
&\qquad q^{-\frac{1}{2\delta}\sum_{\alpha\in I_r}\Lambda_{\alpha\alpha}\ell_{\alpha}}\widehat{Q}_{\bullet,1}^{\mathbf{n}_{t_0}^{\bullet}}\widehat{Q}_{\circ,1}^{\mathbf{n}_1^{\circ}}\prod_{i=1}^s\left(\widehat{Z}_{\alpha_i,p}^{-1}\widehat{Z}_{\alpha_i,t_0}^{\ell_{\alpha_i}+1}\right)\prod_{j=1}^d\left(\widehat{Z}_{\beta_j,0}^{-1}\widehat{Z}_{\beta_j,1}^{\ell_{\beta_j}+1}\right)\nonumber\\
&=q^{-\frac{1}{2\delta}\sum_{\alpha\in I_r}\Lambda_{\alpha\alpha}\ell_{\alpha}}\widehat{Q}_{\bullet,1}^{\mathbf{n}_{t_0}^{\bullet}}\widehat{Q}_{\circ,1}^{\mathbf{n}_1^{\circ}}\widehat{Z}_{\bullet,p}^{-1}\widehat{Z}_{\circ,0}^{-1}\widehat{Z}_{\circ,1}^{\ell^{\circ}+1}\widehat{Z}_{\bullet,1}^{\ell^{\bullet}+1}\nonumber\\
&=q^{-\frac{1}{2\delta}\sum_{\alpha,\beta\in I_r}\left(\Lambda_{\alpha\alpha}\ell_{\alpha}+2\Lambda_{\alpha\beta}n_{\alpha,t_{\alpha}}\right)}\widehat{Z}_{\bullet,p}^{-1}\widehat{Z}_{\circ,0}^{-1}\widehat{Q}_{\bullet,1}^{\mathbf{n}_{t_0}^{\bullet}}\widehat{Q}_{\circ,1}^{\mathbf{n}_1^{\circ}}\widehat{Z}_{\circ,1}^{\ell^{\circ}+1}\widehat{Z}_{\bullet,1}^{\ell^{\bullet}+1}.
\end{align}
By Lemma \ref{4.4} and equation \eqref{eq:4.43}, it follows that we have
\begin{align*}
&\qquad Z_{\lambda,\mathbf{n}}^{(1)}(\widehat{\mathbf{Y}}_{\vec{s}_0})\\
&=q^{-\frac{1}{2\delta}\sum_{\alpha,\beta\in I_r}\left(\Lambda_{\alpha\alpha}\ell_{\alpha}+2\Lambda_{\alpha\beta}n_{\alpha,t_{\alpha}}\right)-\frac{1}{\delta}\sum_{\alpha,\beta\in\Pi_{\bullet}}\sum_{i=1}^{p}\Lambda_{\alpha,\beta}n_{\alpha,i}}\widehat{Z}_{\bullet,0}^{-1}\left(\prod_{i=1}^p\widehat{Q}_{\bullet,i}^{\mathbf{n}_i^{\bullet}}\right)\widehat{Z}_{\circ,0}^{-1}\widehat{Q}_{\bullet,1}^{\mathbf{n}_{t_0}^{\bullet}}\widehat{Q}_{\circ,1}^{\mathbf{n}_1^{\circ}}\widehat{Z}_{\circ,1}^{\ell^{\circ}+1}\widehat{Z}_{\bullet,1}^{\ell^{\bullet}+1}.
\end{align*}
By commuting $\widehat{Z}_{\circ,0}^{-1}$ to the immediate left of $\left(\prod_{i=1}^p\widehat{Q}_{\bullet,i}^{\mathbf{n}_i^{\bullet}}\right)$ using Lemma \ref{3.5}, we obtain the remaining factor of $q^{-\frac{1}{\delta}\sum_{\alpha\in\Pi_{\bullet},\beta\in\Pi_{\circ}}\sum_{i=1}^p\Lambda_{\alpha\beta}n_{\alpha,i}}$, and Lemma \ref{4.3} follows.
\end{proof}

In the general $k>1$ case, we have:

\begin{lemma}\label{4.7}
For $k>1$, the quantum generating function $Z_{\lambda,\mathbf{n}}^{(k)}(\widehat{\mathbf{Y}}_{\vec{s}_0})$ can be expressed as a product of $\widehat{Q}_{\alpha,i}$ ($\alpha\in I_r$, $0\leq i\leq t_{\alpha}+1$) and $Z_{\lambda,\mathbf{n}^{(1,0)}}^{(k-1)}(\widehat{\mathbf{Y}}_{\vec{s}_1})$ as follows:
\begin{equation*}
Z_{\lambda,\mathbf{n}}^{(1)}(\widehat{\mathbf{Y}}_{\vec{s}_0})
=q^{-\frac{1}{\delta}\sum_{\alpha,\beta\in I_r}\sum_{i=1}^{t_{\alpha}}\Lambda_{\alpha\beta}n_{\alpha,i}}\widehat{Z}_{\bullet,0}^{-1}\widehat{Z}_{\circ,0}^{-1}\widehat{P}_{\mathbf{n},1}\widehat{Z}_{\circ,1}\widehat{Z}_{\bullet,t_0}Z_{\lambda,\mathbf{n}^{(1)}}^{(k-1)}(\widehat{\mathbf{Y}}_{\vec{s}_1}).
\end{equation*}
\end{lemma}

The proof proceeds in a similar fashion as that of Lemma \ref{4.3}, which we will describe in detail in Appendix \ref{Appendix B}.

Writing the solution of the quantum $Q$-system as $\widehat{Q}_{\alpha,n}(\widehat{\mathbf{Y}}_{\vec{s}_j})$ to display its dependence on the initial data $\widehat{\mathbf{Y}}_{\vec{s}_j}$, we shall now use the following translational invariance property of the quantum $Q$-system:

\begin{lemma}\label{4.8}
For any solution $\widehat{Q}_{\alpha,n}(\widehat{\mathbf{Y}}_{\vec{0}})$ of the quantum $Q$-system \eqref{eq:3.20}, we have
\begin{equation*}
\widehat{Q}_{\alpha,n}(\widehat{\mathbf{Y}}_{\vec{s}_j})=\widehat{Q}_{\alpha,n+t_{\alpha}j}(\widehat{\mathbf{Y}}_{\vec{s}_0})
\end{equation*}
for all $\alpha\in I_r$, $n\in\mathbb{Z}$ and $j\in\mathbb{Z}_+$.
\end{lemma} 

\begin{proof}
Firstly, we would like to mention that Lemma \ref{4.8} is proved for the simply-laced case \cite[Lemma 5.5]{DFK14}. We will focus on the case where $\mathfrak{g}$ is of type $BCF$, with the arguments for the case where $\mathfrak{g}$ is of type $G_2$ being similar. In this case, we have $t_{\alpha}=1$ if $\alpha\in\Pi_{\circ}$ and $t_{\alpha}=2$ if $\alpha\in\Pi_{\bullet}$. We will prove by induction on $k\in\mathbb{Z}_+$ that the above equation holds for all $\alpha\in I_r$ and $n\in\mathbb{Z}_+$ satisfying $n\leq t_{\alpha}k+1$, with the base case $k=0$ following immediately from the definition of $\widehat{\mathbf{Y}}_{\vec{s}_j}$ and $\widehat{Q}_{\alpha,n}(\widehat{\mathbf{Y}}_{\vec{s}_j})$. 

Suppose that the statement holds for $k=m$, where $m\in\mathbb{Z}_+$. We would like to show that the statement holds for $k=m+1$ as well. By induction hypothesis, we have 
\begin{equation}\label{eq:4.44}
\widehat{Q}_{\alpha,m}(\widehat{\mathbf{Y}}_{\vec{s}_j})
=\widehat{Q}_{\alpha,m+j}(\widehat{\mathbf{Y}}_{\vec{s}_0}),\quad
\widehat{Q}_{\alpha,m+1}(\widehat{\mathbf{Y}}_{\vec{s}_j})
=\widehat{Q}_{\alpha,m+j+1}(\widehat{\mathbf{Y}}_{\vec{s}_0})
\end{equation}
for all $\alpha\in\Pi_{\circ}$, and 
\begin{equation}\label{eq:4.45}
\widehat{Q}_{\alpha,2m}(\widehat{\mathbf{Y}}_{\vec{s}_j})
=\widehat{Q}_{\alpha,2m+2j}(\widehat{\mathbf{Y}}_{\vec{s}_0}),\quad
\widehat{Q}_{\alpha,2m+1}(\widehat{\mathbf{Y}}_{\vec{s}_j})
=\widehat{Q}_{\alpha,2m+2j+1}(\widehat{\mathbf{Y}}_{\vec{s}_0})
\end{equation}
for all $\alpha\in\Pi_{\bullet}$. Let us first pick some $\alpha_{\bullet}\in\Pi_{\bullet}$. We let $\Pi_{\alpha_{\bullet}}=\{\beta\in\Pi_{\bullet}:\beta\sim\alpha_{\bullet}\}$. By \eqref{eq:3.21}, we have 
\begin{equation}\label{eq:4.46}
\widehat{Y}_{\alpha_{\bullet},2m+1}(\widehat{\mathbf{Y}}_{\vec{s}_j})
=\widehat{Q}_{\alpha_{\bullet},2m+1}(\widehat{\mathbf{Y}}_{\vec{s}_j})^{-2}\normord{\widehat{Q}_{\gamma,m}(\widehat{\mathbf{Y}}_{\vec{s}_j})^{\delta_{\alpha_{\bullet},\gamma'}}\widehat{Q}_{\gamma,m+1}(\widehat{\mathbf{Y}}_{\vec{s}_j})^{\delta_{\alpha_{\bullet},\gamma'}}\prod_{\beta\in\Pi_{\alpha_{\bullet}}}\widehat{Q}_{\beta,2m+1}(\widehat{\mathbf{Y}}_{\vec{s}_j})}.
\end{equation}
By combining \eqref{eq:3.21}-\eqref{eq:3.22} and \eqref{eq:4.45}-\eqref{eq:4.46}, we deduce that
\begin{equation}\label{eq:4.47}
\widehat{Q}_{\alpha,2m+2}(\widehat{\mathbf{Y}}_{\vec{s}_j})
=\widehat{Q}_{\alpha,2m+2j+2}(\widehat{\mathbf{Y}}_{\vec{s}_0}).
\end{equation}
Likewise, by observing that
\begin{equation*}
\widehat{Y}_{\alpha_{\bullet},2m+2}(\widehat{\mathbf{Y}}_{\vec{s}_j})
=\widehat{Q}_{\alpha_{\bullet},2m+2}(\widehat{\mathbf{Y}}_{\vec{s}_j})^{-2}\normord{\widehat{Q}_{\gamma,m+1}(\widehat{\mathbf{Y}}_{\vec{s}_j})^{2\delta_{\alpha_{\bullet},\gamma'}}\prod_{\beta\in\Pi_{\alpha_{\bullet}}}\widehat{Q}_{\beta,2m+2}(\widehat{\mathbf{Y}}_{\vec{s}_j})},
\end{equation*}
we may deduce in a similar fashion as before that we have $\widehat{Q}_{\alpha,2m+3}(\widehat{\mathbf{Y}}_{\vec{s}_j})
=\widehat{Q}_{\alpha,2m+2j+3}(\widehat{\mathbf{Y}}_{\vec{s}_0})$.

Next, let us pick some $\alpha_{\circ}\in\Pi_{\circ}$. We let $\Pi_{\alpha_{\circ}}=\{\beta\in\Pi_{\circ}:\beta\sim\alpha_{\circ}\}$. By \eqref{eq:3.21}, we have 
\begin{equation}\label{eq:4.48}
\widehat{Y}_{\alpha_{\circ},m+1}(\widehat{\mathbf{Y}}_{\vec{s}_j})
=\widehat{Q}_{\alpha_{\bullet},m+1}(\widehat{\mathbf{Y}}_{\vec{s}_j})^{-2}\normord{\widehat{Q}_{\gamma',2m+2}(\widehat{\mathbf{Y}}_{\vec{s}_j})^{\delta_{\alpha_{\circ},\gamma}}\prod_{\beta\in\Pi_{\alpha_{\circ}}}\widehat{Q}_{\beta,m+1}(\widehat{\mathbf{Y}}_{\vec{s}_j})}.
\end{equation}
By combining \eqref{eq:3.21}-\eqref{eq:3.22}, \eqref{eq:4.44} and \eqref{eq:4.48}, we deduce that
\begin{equation*}
\widehat{Q}_{\alpha,m+2}(\widehat{\mathbf{Y}}_{\vec{s}_j})
=\widehat{Q}_{\alpha,m+j+2}(\widehat{\mathbf{Y}}_{\vec{s}_0}).
\end{equation*}
This completes the induction step, and we are done.
\end{proof}

Using the translational invariance property of the quantum $Q$-system, along with Lemmas \ref{4.3} and \ref{4.7}, we get

\begin{theorem}\label{4.9}
\begin{equation*}
Z_{\lambda,\mathbf{n}}^{(k)}(\widehat{\mathbf{Y}}_{\vec{s}_0})
=q^{-\frac{1}{2\delta}\sum_{\alpha,\beta\in I_r}\left(\Lambda_{\alpha\alpha}\ell_{\alpha}+2\sum_{i=1}^{t_{\alpha}k}\Lambda_{\alpha\beta}n_{\alpha,i}\right)}\widehat{Z}_{\bullet,0}^{-1}\widehat{Z}_{\circ,0}^{-1}\widehat{P}_{\mathbf{n},k}\widehat{Z}_{\circ,k}^{\ell^{\circ}+1}\widehat{Z}_{\bullet,t_0k}^{\ell^{\bullet}+1}.
\end{equation*}
\end{theorem}

As an immediate consequence of Lemmas \ref{4.7}, \ref{4.8} and Theorem \ref{4.9}, we have:

\begin{corollary}\label{4.10}
For all $0<j<k$, we have
\begin{equation*}
Z_{\lambda,\mathbf{n}}^{(k)}(\widehat{\mathbf{Y}}_{\vec{s}_0})=Z_{\mathbf{0},\mathbf{n}}^{(j)}(\widehat{\mathbf{Y}}_{\vec{s}_0})Z_{\lambda,\mathbf{n}^{(j)}}^{(k-j)}(\widehat{\mathbf{Y}}_{\vec{s}_j}).
\end{equation*}
\end{corollary}

As an immediate consequence of Corollary \ref{4.10}, along with Lemmas \ref{4.3} and \ref{4.4}, we have:

\begin{corollary}\label{4.11}
For all $0\leq j\leq k$, and $0<p<t_0$, we have
\begin{equation*}
Z_{\lambda,\mathbf{n}}^{(k)}(\widehat{\mathbf{Y}}_{\vec{s}_0})
=q^{-\frac{1}{\delta}\sum_{\alpha,\beta\in I_r}\sum_{i=1}^{t_{\alpha}j}\Lambda_{\alpha\beta}n_{\alpha,i}}\widehat{Z}_{\bullet,0}^{-1}\widehat{Z}_{\circ,0}^{-1}\widehat{P}_{\mathbf{n},j}\widehat{Z}_{\circ,j}\widehat{Z}_{\bullet,t_0j}Z_{\lambda,\mathbf{n}^{(j)}}^{(k-j)}(\widehat{\mathbf{Y}}_{\vec{s}_j}),
\end{equation*}
and
\begin{align*}
Z_{\lambda,\mathbf{n}}^{(k)}(\widehat{\mathbf{Y}}_{\vec{s}_0})
&=q^{-\frac{1}{\delta}\sum_{\beta\in I_r}\left(\sum_{\alpha\in\Pi_{\circ}}\sum_{i=1}^j\Lambda_{\alpha\beta}n_{\alpha,i}+\sum_{\omega\in\Pi_{\bullet}}\sum_{i=1}^{t_0j+p}\Lambda_{\omega\beta}n_{\omega,i}\right)}\\
&\quad\times\widehat{Z}_{\bullet,0}^{-1}\widehat{Z}_{\circ,0}^{-1}\widehat{P}_{\mathbf{n},j}\widehat{Z}_{\circ,j}\widehat{Z}_{\bullet,t_0j+p}Z_{\mathbf{\ell},\mathbf{n}^{(j)}}^{(k-j,p)}(\widehat{\mathbf{Y}}_{\vec{s}_{j,p}}).
\end{align*}
\end{corollary}

\subsection{Proof of Theorem \ref{4.1}}

We are now ready to prove Theorem \ref{4.1}. To this end, we need a few auxiliary lemmas from \cite[Section 5.5]{DFK14}. We will omit the proofs of these lemmas as they follow immediately from the proofs of the analogous lemmas in \cite[Section 5.5]{DFK14} with minor modifications. We will let $A$ denote the ring $\mathbb{Z}_u[\widehat{\mathbf{Q}}_1,\widehat{\mathbf{Q}}_{-1},\widehat{\mathbf{Q}}_0^{\pm1}]$, and $A_{\alpha}$ denote the ring $\mathbb{Z}_u[\{\widehat{Q}_{\beta,\pm1}\}_{\beta\neq\alpha},\widehat{\mathbf{Q}}_0^{\pm1}]$ for all $\alpha\in I_r$.

\begin{lemma}[Lemma 5.9, \cite{DFK14}]\label{4.12}
Let $\alpha_1,\cdots,\alpha_n\in I_r$, $i_1,\cdots,i_n\in\mathbb{Z}$ and $m_1,\cdots,m_n\in\mathbb{Z}$. Then $\prod_{j=1}^n\widehat{Q}_{\alpha_j,i_j}^{m_j}\in A$.
\end{lemma}

The proof of \cite[Lemma 5.9]{DFK14} would still be applicable here, since $\widehat{Q}_{\alpha,1}$ and $\widehat{Q}_{\beta,-1}$ are on the same quantum torus for all distinct $\alpha,\beta\in I_r$ by Lemma \ref{3.5}.

\begin{lemma}[Lemma 5.12, \cite{DFK14}]\label{4.13}
For any $\beta\in I_r$ and $f\in\mathbb{Z}_q[\widehat{\mathbf{Q}}_0^{\pm1}]((\widehat{\mathbf{Q}}_1^{-1}))$, we have  
\begin{equation*}
\left(\widehat{Z}_{\bullet,0}^{-1}\widehat{Z}_{\circ,0}^{-1}\widehat{Q}_{\beta,-1}\times f\right)\biggl|_{\widehat{\mathbf{Q}}_0=1}=0.
\end{equation*}
\end{lemma}

\begin{lemma}[Lemma 5.14, \cite{DFK14}]\label{4.14}
For all $\alpha\in I_r$ and $n\in\mathbb{Z}_+$, we have $\widehat{Q}_{\alpha,n}^{-1}\in A_{\alpha}((\widehat{Q}_{\alpha,1}^{-1}))$.
\end{lemma}

\begin{proof}[Proof of Theorem \ref{4.1}]
We shall prove by induction on $j=k,\cdots,0$ that the sum in equation \eqref{eq:4.5} is unchanged if we restrict the sum to sets of vectors $\mathbf{m}$ of non-negative integers such that $q_{\alpha,i}\geq0$ for all $\alpha\in I_r$ and $i\geq t_{\alpha}j$. The base case $j=k$ holds since we have $q_{\alpha,t_{\alpha}k}=\ell_{\alpha}\geq0$ for all $\alpha\in I_r$. Next, let us assume that the statement holds for $j$, where $j\geq1$. Let $p=t_0-1$. By Corollary \ref{4.11}, we have 
\begin{align}\label{eq:4.49}
&\quad Z_{\lambda,\mathbf{n}}^{(k)}(\widehat{\mathbf{Y}}_{\vec{s}_0})\\
&=q^{-\frac{1}{\delta}\sum_{\beta\in I_r}\left(\sum_{\alpha\in\Pi_{\circ}}\sum_{i=1}^{j-1}\Lambda_{\alpha\beta}n_{\alpha,i}+\sum_{\omega\in\Pi_{\bullet}}\sum_{i=1}^{t_0j-1}\Lambda_{\omega\beta}n_{\omega,i}\right)}\widehat{Z}_{\bullet,0}^{-1}\widehat{Z}_{\circ,0}^{-1}\widehat{P}_{\mathbf{n},j-1}\left(\prod_{i=t_0(j-1)+1}^{t_0j-1}\widehat{Q}_{\bullet,i}^{\mathbf{n}_i^{\bullet}}\right)\widehat{Z}_{\circ,j}\widehat{Z}_{\bullet,t_0j-1}\nonumber\\
&\quad\times\sum_{\mathbf{m}^{(j-1,p)}}q^{\overline{Q}_{k-j+1}^{(p)}(\mathbf{m}^{(j-1,0)},\mathbf{n}^{(j-1)})}\prod_{(\alpha,i)\in J_{\mathfrak{g}}^{(j-1,p)}}\begin{bmatrix}m_{\alpha,i}+q_{\alpha,i}\\m_{\alpha,i}\end{bmatrix}_q\widehat{Q}_{\bullet,t_0j}^{-\mathbf{q}_{t_0j-1}^{\bullet}}\widehat{Q}_{\circ,j}^{-\mathbf{q}_0^{\circ,(p)}}\widehat{Q}_{\bullet,t_0j-1}^{\mathbf{q}_{t_0j}^{\bullet}}\widehat{Q}_{\circ,j-1}^{\mathbf{q}_j^{\circ}}.\nonumber
\end{align}
Let us fix any $\beta\in\Pi_{\bullet}$. By induction hypothesis, we may restrict the sum in equation \eqref{eq:4.5} to $q_{\alpha,i}\geq0$ for all $\alpha\in I_r$ and $i\geq t_{\alpha}j$. A generic term of the sum in equation \eqref{eq:4.49} (apart from the coefficients involving $q$) has the form 
\begin{equation*}
S
=\widehat{Z}_{\bullet,0}^{-1}\widehat{Z}_{\circ,0}^{-1}\widehat{P}_{\mathbf{n},j-1}\left(\prod_{i=t_0(j-1)+1}^{t_0j-1}\widehat{Q}_{\bullet,i}^{\mathbf{n}_i^{\bullet}}\right)\widehat{Z}_{\circ,j}\widehat{Z}_{\bullet,t_0j-1}\widehat{Q}_{\bullet,t_0j}^{-\mathbf{q}_{t_0j-1}^{\bullet}}\widehat{Q}_{\circ,j}^{-\mathbf{q}_0^{\circ,(p)}}\widehat{Q}_{\bullet,t_0j-1}^{\mathbf{q}_{t_0j}^{\bullet}}\widehat{Q}_{\circ,j-1}^{\mathbf{q}_j^{\circ}}.
\end{equation*}
When $q_{\beta,j}<0$, it follows from Lemmas \ref{4.4} and \ref{4.14} that $S$ has only non-negative powers of $\widehat{Q}_{\beta,\pm1}$. By Lemma \ref{4.13}r, it follows that $S|_{\widehat{\mathbf{Q}}_0=1}$ is a Laurent series of some variables $\widehat{Q}_{\alpha,1}^{-1}$ with $\alpha\neq\beta$, with coefficients that are in particular polynomials of $\widehat{Q}_{\beta,1}$. Due to the prefactor $\widehat{Z}_{\bullet,0}^{-1}\widehat{Z}_{\circ,0}^{-1}$, it follows that the exponent of $\widehat{Q}_{\beta,1}$ in all terms of $S|_{\widehat{\mathbf{Q}}_0=1}$ are positive. Consequently, we have $\phi(S)=0$. So this shows that the sum in equation \eqref{eq:4.5} is unchanged if we impose the restriction $q_{\beta,t_0j-1}\geq0$. We now repeat the above argument once more, to deduce that the sum in equation \eqref{eq:4.5} is unchanged if we impose the restriction $q_{\beta,i}\geq0$, where $\beta\in I_r$ and $i\geq t_{\beta}(j-1)$. This completes the induction step, and we are done.
\end{proof}

In particular, for a fixed vector $\mathbf{n}=(n_{\alpha,i})_{\alpha\in I_r,i\in\mathbb{N}}$ of nonnegative integers that parameterizes a finite set of KR-modules over $\mathfrak{g}[t]$, a dominant integral weight $\lambda$ of $\mathfrak{g}$ and a sufficiently large integer $k$, we have that $M_{\lambda,\mathbf{n}}^{(k)}(q^{-1})=M_{\lambda,\mathbf{n}}(q^{-1})$ and $N_{\lambda,\mathbf{n}}^{(k)}(q^{-1})=N_{\lambda,\mathbf{n}}(q^{-1})$. Thus, it follows that Theorem \ref{1.2} holds. Together with \cite[Theorem 5.1]{DFK14}, we have

\begin{theorem}\label{4.15}
Let $\mathfrak{g}$ be a simple Lie algebra, $\mathbf{n}=(n_{\alpha,i})_{\alpha\in I_r,i\in\mathbb{N}}$ be a vector of nonnegative integers that parameterizes a finite set of KR-modules over $\mathfrak{g}[t]$, $\lambda$ be a dominant integral weight, and $k$ be a positive integer. Then we have 
\begin{equation*}
M_{\lambda,\mathbf{n}}(q^{-1})=N_{\lambda,\mathbf{n}}(q^{-1}).
\end{equation*}
\end{theorem}

In particular, Theorem \ref{4.15} gives a complete characterization of the graded characters of the fusion products of Kirillov-Reshetikhin modules in terms of the quantum $Q$-system via equation \eqref{eq:4.33}. More precisely, it follows from equation \eqref{eq:4.33} and Theorem \ref{4.9} (along with \cite[Theorem 5.17]{DFK14} in the simply-laced case) that we have 
\begin{equation}\label{eq:4.50}
M_{\lambda,\mathbf{n}}(q^{-1})=q^{-\frac{1}{2\delta}\left(L(\mathbf{n})+\sum_{\alpha,\beta\in I_r}\left(\Lambda_{\alpha\alpha}\ell_{\alpha}+2\sum_{i=1}^{\infty}\Lambda_{\alpha\beta}n_{\alpha,i}\right)\right)}\phi\left(\left(\prod_{\alpha\in I_r}\widehat{Z}_{\alpha,0}^{-1}\right)\widehat{P}_{\mathbf{n}}\prod_{\beta\in I_r}\widehat{Z}_{\beta}^{\ell_{\beta}+1}\right), 
\end{equation}
where 
{\allowdisplaybreaks
\begin{align}
\widehat{P}_{\mathbf{n}}
&=\prod_{\alpha\in \Pi_{\bullet}}\widehat{Q}_{\alpha,1}^{n_{\alpha,1}}\cdots\prod_{\alpha\in \Pi_{\bullet}}\widehat{Q}_{\alpha,t_{\alpha}-1}^{n_{\alpha,t_{\alpha}-1}}\prod_{\alpha\in I_r}\widehat{Q}_{\alpha,t_{\alpha}}^{n_{\alpha,t_{\alpha}}}\prod_{\alpha\in \Pi_{\bullet}}\widehat{Q}_{\alpha,t_{\alpha}+1}^{n_{\alpha,t_{\alpha}+1}}\cdots\prod_{\alpha\in \Pi_{\bullet}}\widehat{Q}_{\alpha,2t_{\alpha}-1}^{n_{\alpha,2t_{\alpha}-1}}\prod_{\alpha\in I_r}\widehat{Q}_{\alpha,2t_{\alpha}}^{n_{\alpha,2t_{\alpha}}}\cdots,\label{eq:4.51}\\
L(\mathbf{n})
&=\sum_{\alpha,\beta\in I_r,i,j\in\mathbb{N}}\frac{\Lambda_{\alpha,\beta}}{t_{\alpha}}\min(t_{\alpha}j,t_{\beta}i)n_{\alpha,i}n_{\beta,j},\label{eq:4.52}\quad\text{and}\\
\widehat{Z}_{\beta}
&=\lim_{k\to\infty}\widehat{Z}_{\beta,k}.\label{eq:4.53}
\end{align}
Here, we briefly remark that a similar reasoning as in the proof of \cite[Theorem 5.17]{DFK14} shows that $\widehat{Z}_{\beta}$ is well-defined as a formal power series of $\widehat{Q}_{\beta,1}^{-1}$ with coefficients Laurent polynomial of the remaining initial data of $\widehat{\mathbf{Y}}_{\vec{s}_0}$.
}
\section{An identity satisfied by the graded characters of \texorpdfstring{$\KR$}{KR}-modules}\label{Section 5}

Our goal in this section is to prove Theorem \ref{1.3}. To this end, let us fix $\alpha\in I_r$ and $m\in\mathbb{N}$. In order to derive Theorem \ref{1.3} without considering too many cases, we would need to a slight generalization of equation \eqref{eq:4.50}. To facilitate this generalization, we will allow ourselves to consider vectors $\widehat{\mathbf{n}}=(n_{\alpha,i})_{\alpha\in I_r,i\in\mathbb{Z}_+}$ of nonnegative integers that parameterizes a finite set of KR-modules over $\mathfrak{g}[t]$. We could view $\widehat{\mathbf{n}}$ as the extension of the vector $\mathbf{n}=(n_{\alpha,i})_{\alpha\in I_r,i\in\mathbb{N}}$ by the vector $\mathbf{n}_0=(n_{\alpha,0})_{\alpha\in I_r}$. Also, we shall make the following definitions: 
{\allowdisplaybreaks
\begin{align}
\widehat{L}(\widehat{\mathbf{n}})
&=\sum_{\alpha,\beta\in I_r,i,j\in\mathbb{Z}_+}\frac{\Lambda_{\alpha\beta}}{t_{\alpha}}\min(t_{\alpha}j,t_{\beta}i)n_{\alpha,i}n_{\beta,j},\label{eq:5.1}\\
\widehat{P}_{\widehat{\mathbf{n}}}
&=\prod_{\alpha\in I_r}\widehat{Q}_{\alpha,0}^{n_{\alpha,0}}\widehat{P}_{\mathbf{n}},\quad\text{and}\label{eq:5.2}\\
M_{\lambda,\widehat{\mathbf{n}}}(q^{-1})
&=q^{-\frac{1}{2\delta}\left(\sum_{\alpha\in I_r}\Lambda_{\alpha\alpha}\ell_{\alpha}+2\widehat{F}(\widehat{\mathbf{n}})+\widehat{L}(\widehat{\mathbf{n}})\right)}\phi\left(\left(\prod_{\alpha\in I_r}\widehat{Z}_{\alpha,0}^{-1}\right)\widehat{P}_{\widehat{\mathbf{n}}}\prod_{\beta\in I_r}\widehat{Z}_{\beta}^{\ell_{\beta}+1}\right),\label{eq:5.3}
\end{align}
where
}
\begin{equation}\label{eq:5.4}
\widehat{F}(\widehat{\mathbf{n}})=\sum_{\alpha,\beta\in I_r,i,j\in\mathbb{Z}_+}\Lambda_{\alpha\beta}n_{\alpha,i}.
\end{equation}
It is easy to see from the definitions of $L(\mathbf{n})$ and $\widehat{L}(\widehat{\mathbf{n}})$ that we have $\widehat{L}(\widehat{\mathbf{n}})=L(\mathbf{n})$.

We claim that $M_{\lambda,\widehat{\mathbf{n}}}(q^{-1})=M_{\lambda,\mathbf{n}}(q^{-1})$. Indeed, we may regard $\left(\prod_{\alpha\in I_r}\widehat{Z}_{\alpha,0}^{-1}\right)\widehat{P}_{\widehat{\mathbf{n}}}\prod_{\beta\in I_r}\widehat{Z}_{\beta}^{\ell_{\beta}+1}$ as an element of $\mathbb{Z}_u[\widehat{\mathbf{Q}}_0^{\pm1}]((\widehat{\mathbf{Q}}_1^{-1}))$ by Theorem \ref{4.11}. Thus, it follows from the definition of the function $\phi$ that we have
{\allowdisplaybreaks
\begin{align*}
M_{\lambda,\mathbf{n}}(q^{-1})
&=q^{-\frac{1}{2\delta}\left[\sum_{\alpha,\beta\in I_r}\left(\Lambda_{\alpha\alpha}\ell_{\alpha}+2\sum_{i=1}^{\infty}\Lambda_{\alpha\beta}n_{\alpha,i}\right)+L(\mathbf{n})\right]}\phi\left(\left(\prod_{\alpha\in I_r}\widehat{Z}_{\alpha,0}^{-1}\right)\widehat{P}_{\mathbf{n}}\prod_{\beta\in I_r}\widehat{Z}_{\beta}^{\ell_{\beta}+1}\right)\\
&=q^{-\frac{1}{2\delta}\left[\sum_{\alpha\beta\in I_r}\left(\Lambda_{\alpha\alpha}\ell_{\alpha}+2\sum_{i=1}^{\infty}\Lambda_{\alpha\beta}n_{\alpha,i}\right)+\widehat{L}(\widehat{\mathbf{n}})\right]}\phi\left(\prod_{\alpha\in I_r}\widehat{Q}_{\alpha,0}^{n_{\alpha,0}}\left(\prod_{\alpha\in I_r}\widehat{Z}_{\alpha,0}^{-1}\right)\widehat{P}_{\mathbf{n}}\prod_{\beta\in I_r}\widehat{Z}_{\beta}^{\ell_{\beta}+1}\right)\\
&=q^{-\frac{1}{2\delta}\left[\sum_{\alpha,\beta\in I_r}\left(\Lambda_{\alpha\alpha}\ell_{\alpha}+2\sum_{i=0}^{\infty}\Lambda_{\alpha\beta}n_{\alpha,i}\right)+\widehat{L}(\widehat{\mathbf{n}})\right]}\phi\left(\left(\prod_{\alpha\in I_r}\widehat{Z}_{\alpha,0}^{-1}\right)\left(\prod_{\alpha\in I_r}\widehat{Q}_{\alpha,0}^{n_{\alpha,0}}\right)\widehat{P}_{\mathbf{n}}\prod_{\beta\in I_r}\widehat{Z}_{\beta}^{\ell_{\beta}+1}\right)\\
&=M_{\lambda,\widehat{\mathbf{n}}}(q^{-1}),
\end{align*}
where the third equality follows from Lemma \ref{3.5}. The equality $M_{\lambda,\mathbf{n}}(q^{-1})=M_{\lambda,\widehat{\mathbf{n}}}(q^{-1})$ is consistent with the fact that the fusion product of a cyclic $\mathfrak{g}[t]$-module $V$ with the trivial $\mathfrak{g}[t]$-module is precisely $V$ itself. Hence, we may regard $M_{\lambda,\widehat{\mathbf{n}}}(q^{-1})$ as the graded multiplicity of the irreducible component $V(\lambda)$ in $\mathcal{F}_{\widehat{\mathbf{n}}}^*$, where $\mathcal{F}_{\widehat{\mathbf{n}}}^*$ is the corresponding fusion product of KR-modules parameterized by $\widehat{\mathbf{n}}$. More precisely, we have
}
\begin{equation*}
M_{\lambda,\widehat{\mathbf{n}}}(q)=\sum_{m=0}^{\infty}\dim\Hom_{\mathfrak{g}}(\mathcal{F}_{\widehat{\mathbf{n}}}^*[m],V(\lambda))q^m.
\end{equation*}
Our next step is to express the terms (without the coefficients) that appear in the quantum $Q$-system relation \eqref{eq:3.20} as some $\widehat{P}_{\widehat{\mathbf{n}}}$. To begin, we need to rewrite our quantum $Q$-system relation \eqref{eq:3.20}. We first observe from Lemma \ref{3.5} and \eqref{eq:3.22} that we have 
\begin{equation*}
\widehat{Q}_{\alpha,m+1}
=\nu^{\Lambda_{\alpha\alpha}}\widehat{Q}_{\alpha,m}^2\left(1-\widehat{Y}_{\alpha,m}\right)\widehat{Q}_{\alpha,m-1}^{-1}
=\nu^{-\Lambda_{\alpha\alpha}}\widehat{Q}_{\alpha,m-1}^{-1}\widehat{Q}_{\alpha,m}^2\left(1-\nu^{\delta}\widehat{Y}_{\alpha,m}\right),
\end{equation*}
or equivalently, 
\begin{equation}\label{eq:5.5}
\nu^{\Lambda_{\alpha\alpha}}\widehat{Q}_{\alpha,m-1}\widehat{Q}_{\alpha,m+1}=\widehat{Q}_{\alpha,m}^2-\nu^{\delta}\normord{\prod_{\beta\sim\alpha}\prod_{i=0}^{|C_{\alpha\beta}|-1}\widehat{Q}_{\beta,\lfloor t_{\beta}(m+i)/t_{\alpha}\rfloor}}.
\end{equation}
Next, we let $\widehat{\mathbf{d}}=(d_{\beta,i})_{\beta\in I_r,i\in\mathbb{Z}_+}$ be the vector that corresponds to the term $\widehat{Q}_{\alpha,m-1}\widehat{Q}_{\alpha,m+1}$, where 
\begin{equation}\label{eq:5.6}
d_{\beta,i}=\delta_{\alpha\beta}(\delta_{i,m-1}+\delta_{i,m+1}).
\end{equation}
Next, we let $\widehat{\mathbf{s}}=(s_{\beta,i})_{\beta\in I_r,i\in\mathbb{Z}_+}$ be the vector that corresponds to the term $\widehat{Q}_{\alpha,m}^2$, where 
\begin{equation}\label{eq:5.7}
s_{\beta,i}=2\delta_{\alpha\beta}\delta_{i,m}.
\end{equation}
Finally, we will let $\widehat{\mathbf{k}}=(k_{\beta,i})_{\beta\in I_r,i\in\mathbb{Z}_+}$ to be the vector that corresponds to the term 
\begin{equation*}
\normord{\prod_{\beta\sim\alpha}\prod_{i=0}^{|C_{\alpha\beta}|-1}\widehat{Q}_{\beta,\lfloor t_{\beta}(m+i)/t_{\alpha}\rfloor}}. 
\end{equation*}
In order to define the numbers $k_{\beta,i}$, we would need to consider cases:

Case 1: Either $\alpha\neq\gamma'$ or $\alpha=\gamma'$ and $t_0\divides m$. In this case, we define
\begin{equation}\label{eq:5.8}
k_{\beta,i}=-\delta_{\beta\sim\alpha}\delta_{i,t_{\beta}m/t_{\alpha}}C_{\alpha\beta},
\end{equation}
where the function $\delta_{\beta\sim\alpha}$ is equal to $1$ if $\beta\sim\alpha$, and is equal to $0$ otherwise.

Case 2: $\alpha=\gamma'$ and $t_0\not\divides m$. Let us write $m=t_0n+p$, where $0<p<t_0$. In this case, we define
\begin{equation}\label{eq:5.9}
k_{\beta,i}=\delta_{\beta\sim\alpha}\delta_{i,m},
\end{equation}
for all roots $\beta\neq\gamma$, and 
\begin{equation}\label{eq:5.10}
k_{\gamma,i}=(t_0-p)\delta_{i,n}+p\delta_{i,n+1}.
\end{equation}
It is easy to see from the definition of $\widehat{\mathbf{d}}$ and $\widehat{\mathbf{s}}$ that we have $\widehat{Q}_{\alpha,m-1}\widehat{Q}_{\alpha,m+1}=\widehat{P}_{\widehat{\mathbf{d}}}$ and $\widehat{Q}_{\alpha,m}^2=\widehat{P}_{\widehat{\mathbf{s}}}$. 

It remains to write $\normord{\prod_{\beta\sim\alpha}\prod_{i=0}^{|C_{\alpha\beta}|-1}\widehat{Q}_{\beta,\lfloor t_{\beta}(m+i)/t_{\alpha}\rfloor}}$ as a scalar multiple of $\widehat{P}_{\widehat{\mathbf{k}}}$. To this end, we first observe from Lemma \ref{3.5} that in the case where either $\alpha\neq\gamma'$ or $\alpha=\gamma'$ and $t_0\divides m$, we have
\begin{equation*}
\normord{\prod_{\beta\sim\alpha}\prod_{i=0}^{|C_{\alpha\beta}|-1}\widehat{Q}_{\beta,\lfloor t_{\beta}(m+i)/t_{\alpha}\rfloor}}=\widehat{P}_{\widehat{\mathbf{k}}}.
\end{equation*}
It remains to handle the case where $\alpha=\gamma'$ and $t_0\not\divides m$. To this end, we would need to consider subcases. Let us first consider the subcase where $\mathfrak{g}$ is of type $BCF$, in which case we have $t_0=2$ and $p=1$. By Lemma \ref{3.5}, we have
\begin{equation*}
\normord{\prod_{\beta\sim\alpha}\prod_{i=0}^{|C_{\alpha\beta}|-1}\widehat{Q}_{\beta,\lfloor t_{\beta}(m+i)/t_{\alpha}\rfloor}}
=\nu^{\frac{\Lambda_{\gamma\gamma}}{2}}\widehat{Q}_{\gamma,n}\widehat{Q}_{\gamma,n+1}\prod_{\beta\in\Pi_{\bullet}}\widehat{Q}_{\beta,2n+1}^{-C_{\alpha\beta}}
=\nu^{\frac{1}{2}\sum_{\beta\sim\gamma'}\Lambda_{\beta\gamma}}\widehat{P}_{\widehat{\mathbf{k}}}.
\end{equation*}
Next, in the subcase where $\mathfrak{g}$ is of type $G$, in which case we have $t_0=3$ and $p=1,2$. By Lemma \ref{3.5}, we have
\begin{equation*}
\normord{\prod_{\beta\sim\alpha}\prod_{i=0}^{|C_{\alpha\beta}|-1}\widehat{Q}_{\beta,\lfloor t_{\beta}(m+i)/t_{\alpha}\rfloor}}
=\nu^2\widehat{Q}_{1,n}^{3-p}\widehat{Q}_{1,n+1}^p
=\nu^2\widehat{P}_{\widehat{\mathbf{k}}}.
\end{equation*}
We conclude that 
\begin{equation}\label{eq:5.11}
\normord{\prod_{\beta\sim\alpha}\prod_{i=0}^{|C_{\alpha\beta}|-1}\widehat{Q}_{\beta,\lfloor t_{\beta}(m+i)/t_{\alpha}\rfloor}}=\nu^{\sigma_{\alpha,m}}\widehat{P}_{\widehat{\mathbf{k}}},
\end{equation}
where
\begin{equation}\label{eq:5.12}
\sigma_{\alpha,m}
=
\begin{cases}
0 & \text{if } \alpha\neq\gamma' \text{ or } \alpha=\gamma' \text{ and } t_0\divides m;\\
\frac{1}{2}\sum_{\beta\sim\gamma'}\Lambda_{\beta\gamma} & \text{if }\alpha=\gamma', t_0\not\divides m,\text{ and }\mathfrak{g}\text{ is of type }BCF;\\
2 & \text{if }\alpha=\gamma', t_0\not\divides m,\text{ and }\mathfrak{g}\text{ is of type }G,
\end{cases}
\end{equation}
and hence \eqref{eq:5.5} is equivalent to
\begin{equation}\label{eq:5.13}
\nu^{\Lambda_{\alpha\alpha}}\widehat{P}_{\widehat{\mathbf{d}}}=\widehat{P}_{\widehat{\mathbf{s}}}-\nu^{\delta+\sigma_{\alpha,m}}\widehat{P}_{\widehat{\mathbf{k}}}.
\end{equation}
As the map $\phi$ is linear, it follows that we have
\begin{align*}
&\qquad\nu^{\Lambda_{\alpha\alpha}}\phi\left(\left(\prod_{\alpha\in I_r}\widehat{Z}_{\alpha,0}^{-1}\right)\widehat{P}_{\widehat{\mathbf{d}}}\prod_{\beta\in I_r}\widehat{Z}_{\beta}^{\ell_{\beta}+1}\right)\\
&=\phi\left(\left(\prod_{\alpha\in I_r}\widehat{Z}_{\alpha,0}^{-1}\right)\widehat{P}_{\widehat{\mathbf{s}}}\prod_{\beta\in I_r}\widehat{Z}_{\beta}^{\ell_{\beta}+1}\right)-\nu^{\delta+\sigma_{\alpha,m}}\phi\left(\left(\prod_{\alpha\in I_r}\widehat{Z}_{\alpha,0}^{-1}\right)\widehat{P}_{\widehat{\mathbf{k}}}\prod_{\beta\in I_r}\widehat{Z}_{\beta}^{\ell_{\beta}+1}\right),
\end{align*}
or equivalently,
\begin{align}
&\qquad q^{\frac{1}{2\delta}\left(2\widehat{F}(\widehat{\mathbf{d}})+\widehat{L}(\widehat{\mathbf{d}})+2\Lambda_{\alpha\alpha}\right)}M_{\lambda,\widehat{\mathbf{d}}}(q^{-1})\label{eq:5.14}\\
&=q^{\frac{1}{2\delta}\left(2\widehat{F}(\widehat{\mathbf{s}})+\widehat{L}(\widehat{\mathbf{s}})\right)}M_{\lambda,\widehat{\mathbf{s}}}(q^{-1})+q^{\frac{1}{2\delta}\left(2\widehat{F}(\widehat{\mathbf{k}})+\widehat{L}(\widehat{\mathbf{k}})+2\sigma_{\alpha,m}+2\delta\right)}M_{\lambda,\widehat{\mathbf{k}}}(q^{-1}),\nonumber
\end{align}
where we have factored out $q^{\frac{1}{2\delta}\sum_{\alpha\in I_r}\Lambda_{\alpha\alpha}\ell_{\alpha}}$ on both sides of the equation.

It remains to simplify equation \eqref{eq:5.14}. To this end, we would need to compute the exponents that appear in equation \eqref{eq:5.14} explicitly. Our first technical lemma involves the relation between $\widehat{F}(\widehat{\mathbf{d}})$, $\widehat{F}(\widehat{\mathbf{s}})$ and $\widehat{F}(\widehat{\mathbf{k}})$:

\begin{lemma}\label{5.1}
For all $\alpha\in I_r$ and $m\in\mathbb{N}$, we have
\begin{equation*}
\widehat{F}(\widehat{\mathbf{d}})=\widehat{F}(\widehat{\mathbf{s}})=\delta+\widehat{F}(\widehat{\mathbf{k}}).
\end{equation*}
\end{lemma}

\begin{proof}
The statement follows immediately from the following equalities:
\begin{align*}
\widehat{F}(\widehat{\mathbf{d}})
&=\sum_{\omega,\beta\in I_r}\sum_{i=0}^{\infty}\Lambda_{\beta\omega}d_{\beta,i}
=2\sum_{\omega\in I_r}\Lambda_{\alpha,\omega};\\
\widehat{F}(\widehat{\mathbf{s}})
&=\sum_{\omega,\beta\in I_r}\sum_{i=0}^{\infty}\Lambda_{\beta\omega}s_{\beta,i}
=2\sum_{\omega\in I_r}\Lambda_{\alpha,\omega};\\
\widehat{F}(\widehat{\mathbf{k}})
&=\sum_{\omega,\beta\in I_r}\sum_{i=0}^{\infty}\Lambda_{\beta\omega}k_{\beta,i}
=-\sum_{\omega\in I_r,\beta\sim\alpha}C_{\alpha\beta}\Lambda_{\beta\omega}
=-\delta+2\sum_{\omega\in I_r}\Lambda_{\alpha\omega},
\end{align*}
where we used $C\Lambda=\delta I$ in the last equality.
\end{proof}

Our next technical lemma involves the relation $\widehat{L}(\widehat{\mathbf{d}})$ and $\widehat{L}(\widehat{\mathbf{s}})$:

\begin{lemma}\label{5.2}
For all $\alpha\in I_r$ and $m\in\mathbb{N}$, we have
\begin{equation*}
\widehat{L}(\widehat{\mathbf{d}})=\widehat{L}(\widehat{\mathbf{s}})-2\Lambda_{\alpha\alpha}.
\end{equation*}
\end{lemma}

\begin{proof}
The statement follows immediately from the following equalities:
\begin{align*}
\widehat{L}(\widehat{\mathbf{d}})
&=\sum_{\beta,\omega\in I_r,i,j\in\mathbb{Z}_+}\frac{\Lambda_{\beta\omega}}{t_{\beta}}\min(t_{\beta}j,t_{\omega}i)d_{\beta,i}d_{\omega,j}\\
&=\Lambda_{\alpha\alpha}[\min(m-1,m-1)+2\min(m-1,m+1)+\min(m+1,m+1)]\\
&=(4m-2)\Lambda_{\alpha\alpha};\\
\widehat{L}(\widehat{\mathbf{s}})
&=\sum_{\beta,\omega\in I_r,i,j\in\mathbb{Z}_+}\frac{\Lambda_{\beta\omega}}{t_{\beta}}\min(t_{\beta}j,t_{\omega}i)s_{\beta,i}s_{\omega,j}
=2^2\Lambda_{\alpha\alpha}\min(m,m)
=4m\Lambda_{\alpha\alpha}.
\end{align*}
\end{proof}

Our final technical lemma involves relating $\widehat{L}(\widehat{\mathbf{s}})$ and $\widehat{L}(\widehat{\mathbf{k}})+2\sigma_{\alpha,m}$:

\begin{lemma}\label{5.3}
For all $\alpha\in I_r$ and $m\in\mathbb{N}$, we have
\begin{equation*}
\widehat{L}(\widehat{\mathbf{k}})+2\sigma_{\alpha,m}=\widehat{L}(\widehat{\mathbf{s}})-2m\delta.
\end{equation*}
\end{lemma}

We are now ready to prove Theorem \ref{1.3}.

\begin{proof}[Proof of Theorem \ref{1.3}]
By Lemmas \ref{5.1}-\ref{5.3}, we have
\begin{equation*}
2\widehat{F}(\widehat{\mathbf{d}})+\widehat{L}(\widehat{\mathbf{d}})+2\Lambda_{\alpha\alpha}
=2\widehat{F}(\widehat{\mathbf{s}})+\widehat{L}(\widehat{\mathbf{s}})
=2\widehat{F}(\widehat{\mathbf{k}})+\widehat{L}(\widehat{\mathbf{k}})+2\sigma_{\alpha,m}+2\delta+2m\delta,
\end{equation*}
so equation \eqref{eq:5.14} reduces to
\begin{equation*}
M_{\lambda,\widehat{\mathbf{d}}}(q^{-1})
=M_{\lambda,\widehat{\mathbf{s}}}(q^{-1})
-q^{-m}M_{\lambda,\widehat{\mathbf{k}}}(q^{-1}),
\end{equation*}
or equivalently,
\begin{equation}\label{eq:5.15}
M_{\lambda,\widehat{\mathbf{d}}}(q)
=M_{\lambda,\widehat{\mathbf{s}}}(q)
-q^{-m}M_{\lambda,\widehat{\mathbf{k}}}(q).
\end{equation}
As equation \eqref{eq:5.15} holds for all dominant weights $\lambda$, and $M_{\lambda,\widehat{\mathbf{n}}}(q)$ is the graded multiplicity of the irreducible component $V(\lambda)$ in $\mathcal{F}_{\widehat{\mathbf{n}}}^*$, we have 
\begin{equation*}
\ch_q\mathcal{F}_{\widehat{\mathbf{d}}}^*
=\ch_q\mathcal{F}_{\widehat{\mathbf{s}}}^*
-q^m\ch_q\mathcal{F}_{\widehat{\mathbf{k}}}^*.
\end{equation*}
Theorem \ref{1.3} now follows from the above equation, along with the definitions of $\widehat{\mathbf{d}}$, $\widehat{\mathbf{s}}$ and $\widehat{\mathbf{k}}$. 
\end{proof}

\begin{remark}
In a related work, Chari and Venkatesh \cite[Theorem 4]{CV15} showed that when $m$ is a multiple of $t_{\alpha}$, there exists a short exact sequence of fusion product of KR-modules that extends the $Q$-system relations \eqref{eq:1.1}: 
\begin{equation*}
0\longrightarrow\tau_mK_{\alpha,m}^{\star}\longrightarrow\KR_{\alpha,m}\star\KR_{\alpha,m}\longrightarrow\KR_{\alpha,m+1}\star\KR_{\alpha,m-1}\longrightarrow0,
\end{equation*}
where $\tau_k$ is the grading shift operator on the set of graded $\mathfrak{g}[t]$-modules by $k$. By applying the character map to the above exact sequence, we see that the identity of graded characters of the fusion product of KR-modules is the same as that as the identity stated in Theorem \ref{1.3} in the aforementioned case. This implies the short exact sequences of fusion products of KR-modules obtained in \cite{CV15} are consistent with our quantum $Q$-system relations. In light of Theorem~\ref{1.3}, we expect that the above short exact sequence of fusion product of KR-modules should exist in the remaining cases.
\end{remark}

\section{Conclusion}\label{Section 6}

In this paper, we have proved the combinatorial identity $M_{\lambda,\mathbf{n}}(q^{-1})=N_{\lambda,\mathbf{n}}(q^{-1})$ \ref{1.2} for all non-simply laced simple Lie algebras $\mathfrak{g}$, which together with \cite[Theorem 5.1]{DFK14}, shows that the identity holds for all simple Lie algebras $\mathfrak{g}$. This was done by defining appropriate quantum generating functions that satisfy some factorization properties, and writing the $q$-graded sum $N_{\lambda,\mathbf{n}}(q^{-1})$ as the constant term evaluation of these generating functions, which would then allow us to prove the desired identity. As an application, we obtained an identity of graded characters of the fusion product of KR-modules over the current algebra $\mathfrak{g}[t]$, and showed that the short exact sequences of fusion products of KR-modules over $\mathfrak{g}[t]$ obtained by Chari and Venkatesh in \cite{CV15} are consistent with our quantum $Q$-system relations. 

Here, we would like to remark that there is a generalization of the conjecture of the $q$-graded fermionic sums, conjectured by Hatayama \emph{et al.} in the sequel \cite{HKOTT02}, to the twisted case as well. In the same paper, Hatayama \emph{et al.} gave a combinatorial definition of the KR-modules for the twisted quantum affine algebras, and showed that if the restricted characters of these KR-modules over the twisted quantum affine algebras satisfy the twisted $Q$-system relations (which first appeared in \cite{KS95}), along with some other asymptotic conditions, then the multiplicities arising in the tensor product of KR-modules over the twisted quantum affine algebras could expressed in terms of an extended fermionic sum $\tilde{M}_{\lambda,\mathbf{n}}(1)$ \cite[4.20]{HKOTT02} (here, $\tilde{M}_{\lambda,\mathbf{n}}(q^{-1})$ plays the role of $N_{\lambda,\mathbf{n}}(q^{-1})$ in the twisted case).

Subsequently, Hernandez \cite{Hernandez10} showed that the restricted characters of these KR-modules over the twisted quantum affine algebras do satisfy the twisted $Q$-system relations, while Okado \emph{et al.} \cite{OSS18} established a bijection between rigged configurations and crystal paths in the non-exceptional cases. These two results (along with earlier results by Hernandez \cite{Hernandez06}) together shows that conjectural identity \cite[Conjecture 4.3]{HKOTT02} of the $q$-graded fermionic sums $M_{\lambda,\mathbf{n}}(q^{-1})=\tilde{M}_{\lambda,\mathbf{n}}(q^{-1})$ holds at $q=1$ in the non-exceptional cases.

In another development, Williams \cite{Williams15} extended the results of \cite{DFK09}, and showed that the twisted $Q$-system relations of type $\neq A_{2n}^{(2)}$ could be interpreted as cluster algebra mutations as well. Similar to the untwisted case, these cluster algebras admit natural deformations as well, so we could obtain the quantum twisted $Q$-system relations of type $\neq A_{2n}^{(2)}$. Using these quantum twisted $Q$-system relations and the methods developed in \cite{DFK14} and in this paper, we believe that the conjectural identity of $q$-fermionic sums $M_{\lambda,\mathbf{n}}(q^{-1})=\tilde{M}_{\lambda,\mathbf{n}}(q^{-1})$ holds in all twisted cases of type $\neq A_{2n}^{(2)}$, and we will address this conjectural identity $M_{\lambda,\mathbf{n}}(q^{-1})=\tilde{M}_{\lambda,\mathbf{n}}(q^{-1})$ in these cases in a future publication \cite{Lin19-1}. We will also address the $A_{2n}^{(2)}$ case as well in a separate future publication \cite{Lin19-2}.

\appendix

\section{Proof of Lemma \ref{4.4}}\label{Appendix A}

As before in the proof of Lemma \ref{4.3}, we would first need to compute $\widehat{Y}_{\alpha,p}$ explicitly for all $\alpha\in\Pi_{\bullet}$ and $1\leq p\leq t_0-1$. Firstly, we observe that in the case where $\alpha\neq\gamma'$, it follows from Lemma \ref{3.5} that we have
\begin{equation}\label{eq:A.1}
\widehat{Y}_{\alpha,p}=\prod_{\beta\in I_r}\widehat{Q}_{\beta,k}^{-C_{\alpha\beta}}.
\end{equation}
Our next goal is to compute $\widehat{Y}_{\gamma',p}$ explicitly for all $1\leq p\leq t_0-1$. To this end, we need to consider cases. Firstly, when $\mathfrak{g}$ is of type $BCF$, we have $t_0=2$, in which case the only integral value of $p$ for which $1\leq p\leq t_0-1$ is $p=1$. By observing from Lemma \ref{3.5} that we have 
\begin{equation*}
C\left(\widehat{Q}_{\alpha,1}^{-2},\normord{\prod_{\beta\sim\alpha}\prod_{i=0}^{|C_{\alpha\beta}|-1}\widehat{Q}_{\beta,\lfloor t_{\beta}(1+i)/t_{\alpha}\rfloor}}\right)=0,
\end{equation*}
it follows that we have
\begin{equation}\label{eq:A.2}
\widehat{Y}_{\gamma',1}=q^{\frac{\Lambda_{\gamma\gamma}}{2\delta}}\widehat{Q}_{\gamma,0}\widehat{Q}_{\gamma,1}\prod_{\beta\in\Pi_{\bullet}}\widehat{Q}_{\beta,1}^{-C_{\alpha\beta}}.
\end{equation}
Next, when $\mathfrak{g}$ is of type $G_2$, we have $t_0=3$, in which case the only integral values of $k$ for which $1\leq p\leq t_0-1$ is $p=1,2$. For both values of $p$, it is easy to check that we have
\begin{equation}\label{eq:A.3}
\widehat{Y}_{2,p}=q^2\widehat{Q}_{1,0}^{3-p}\widehat{Q}_{1,1}^p\widehat{Q}_{2,k}^{-2}.
\end{equation}

\subsection{Proof of Lemma \ref{4.4} in the case where \texorpdfstring{$\mathfrak{g}$}{g} is of type \texorpdfstring{$BCF$}{BCF}}

Firstly, it follows from equations \eqref{eq:4.7}, \eqref{eq:4.8} and \eqref{eq:4.11} that we have
\begin{align}
\mathbf{q}_0^{\bullet}-2\mathbf{q}_1^{\bullet}+\mathbf{q}_2^{\bullet}&=C^{\bullet}\mathbf{m}_1^{\bullet}-\mathbf{n}_1^{\bullet},\label{eq:A.4}\\
\mathbf{q} _0^{\circ}-\mathbf{q}_0^{\circ,(1)}&=-D\mathbf{e}_1.\label{eq:A.5}
\end{align}
Next, by equations \eqref{eq:4.27} and \eqref{eq:4.29}, we have
\begin{align}
&\qquad\overline{Q}_k(\mathbf{m},\mathbf{n})-\overline{Q}_k^{(1)}(\mathbf{m},\mathbf{n})\label{eq:A.6}\\
&=\frac{1}{2\delta}\left(2\mathbf{q}_1^{\bullet}\cdot\Lambda^{\bullet}(\mathbf{q}_1^{\bullet}-\mathbf{q}_2^{\bullet})+2\mathbf{q}_0^{\circ}\cdot A(2\mathbf{q}_1^{\bullet}-\mathbf{q}_2^{\bullet})-U_1+2(\Lambda^{\circ}\mathbf{q}_0^{\circ}+A\mathbf{q}_0^{\bullet})\cdot D\mathbf{e}_1\right).\nonumber
\end{align}
By letting $c_1=2(\Lambda^{\circ}\mathbf{q}_0^{\circ}+A\mathbf{q}_0^{\bullet})\cdot D\mathbf{e}_1$, $c_2=2\mathbf{q}_1^{\bullet}\cdot\Lambda^{\bullet}(2\mathbf{q}_1^{\bullet}-\mathbf{q}_2^{\bullet})+2\mathbf{q}_0^{\circ}\cdot A(2\mathbf{q}_1^{\bullet}-\mathbf{q}_2^{\bullet})$, $c_3=-2\mathbf{q}_1^{\bullet}\cdot\Lambda^{\bullet}\mathbf{q}_1^{\bullet}+\sum_{\alpha\in\Pi_{\bullet}}\Lambda_{\alpha\alpha}q_{\alpha,1}(q_{\alpha,1}+1)$ and $c_4=-U_1$, it follows that from \eqref{eq:A.6} that we have 
\begin{equation}\label{eq:A.7}
\frac{1}{2\delta}(c_1+c_2+c_3+c_4)=\overline{Q}_k(\mathbf{m},\mathbf{n})-\overline{Q}_k^{(1)}(\mathbf{m},\mathbf{n})+\frac{1}{2\delta}\sum_{\alpha\in\Pi_{\bullet}}\Lambda_{\alpha\alpha}q_{\alpha,1}(q_{\alpha,1}+1).
\end{equation}
We first use Lemma \ref{3.5}, along with equations \eqref{eq:A.2} and \eqref{eq:A.7}, to rearrange the terms involving the $\widehat{Q}_{\alpha,i}$'s in the quantum generating function $Z_{\lambda,\mathbf{n}}^{(k)}(\widehat{\mathbf{Y}}_{\vec{s}_0})$ as follows:
\begin{align}
\widehat{Q}_{\bullet,1}^{-\mathbf{q}_0^{\bullet}}\widehat{Q}_{\circ,1}^{-\mathbf{q}_0^{\circ}}\widehat{Q}_{\bullet,0}^{\mathbf{q}_1^{\bullet}}\widehat{Q}_{\circ,0}^{\mathbf{q}_1^{\circ}+D\mathbf{e}_1}
&=q^{-\frac{1}{2\delta}c_1}\widehat{Q}_{\circ,0}^{D\mathbf{e}_1}\widehat{Q}_{\bullet,1}^{-\mathbf{q}_0^{\bullet}}\widehat{Q}_{\circ,1}^{-\mathbf{q}_0^{\circ}}\widehat{Q}_{\bullet,0}^{\mathbf{q}_1^{\bullet}}\widehat{Q}_{\circ,0}^{\mathbf{q}_1^{\circ}},\label{eq:A.8}\\
\widehat{Q}_{\bullet,1}^{-\mathbf{q}_0^{\bullet}}
&=(q^{\frac{\Lambda_{\gamma\gamma}}{2\delta}}\widehat{Q}_{\gamma,0}\widehat{Q}_{\gamma,1})^{-m_{\gamma',1}}\widehat{Y}_{\bullet,1}^{\mathbf{m}_1^{\bullet}}\widehat{Q}_{\bullet,1}^{\mathbf{n}_1^{\bullet}-2\mathbf{q}_1^{\bullet}+\mathbf{q}_2^{\bullet}},\label{eq:A.9}\\
\widehat{Q}_{\bullet,1}^{-2\mathbf{q}_1^{\bullet}+\mathbf{q}_2^{\bullet}}\widehat{Q}_{\circ,1}^{-\mathbf{q}_0^{\circ}}\widehat{Q}_{\bullet,0}^{\mathbf{q}_1^{\bullet}}
&=q^{-\frac{1}{2\delta}c_2}\widehat{Q}_{\circ,1}^{-\mathbf{q}_0^{\circ}}\widehat{Q}_{\bullet,0}^{\mathbf{q}_1^{\bullet}}\widehat{Q}_{\bullet,1}^{-2\mathbf{q}_1^{\bullet}+\mathbf{q}_2^{\bullet}},\label{eq:A.10}\\
\widehat{Q}_{\bullet,0}^{\mathbf{q}_1^{\bullet}}\widehat{Q}_{\bullet,1}^{-2\mathbf{q}_1^{\bullet}}
&=q^{\frac{1}{\delta}(\mathbf{q}_1^{\bullet}\cdot\Lambda^{\bullet}\mathbf{q}_1^{\bullet}-\sum_{\alpha\in\Pi_{\bullet}}\Lambda_{\alpha\alpha}q_{\alpha,1}^2)}\prod_{i=1}^s\widehat{Q}_{\alpha_i,0}^{q_{\alpha_i,1}}\widehat{Q}_{\alpha_i,1}^{-2q_{\alpha_i,1}}\nonumber\\
&=q^{-\frac{1}{2\delta}c_3}\prod_{i=1}^s\widehat{Z}_{\alpha_i,0}^{q_{\alpha_i,1}}\widehat{Q}_{\alpha_i,1}^{-q_{\alpha_i,1}}.\label{eq:A.11}
\end{align}
As we have $(D\mathbf{e}_1)_{\alpha}=\delta_{\alpha\gamma}m_{\gamma',1}$, we deduce from Lemmas \ref{3.5} and \ref{3.6}(8) that we have
{\allowdisplaybreaks
\begin{align}
&\qquad\widehat{Q}_{\circ,0}^{D\mathbf{e}_1}(q^{\frac{\Lambda_{\gamma\gamma}}{2\delta}}\widehat{Q}_{\gamma,0}\widehat{Q}_{\gamma,1})^{-m_{\gamma',1}}\widehat{Y}_{\bullet,1}^{\mathbf{m}_1^{\bullet}}\widehat{Q}_{\bullet,1}^{\mathbf{n}_1^{\bullet}}\widehat{Q}_{\circ,1}^{-\mathbf{q}_0^{\circ}}\label{eq:A.12}\\
&=q^{-\frac{1}{2\delta}(c_4-\Lambda_{\gamma\gamma}m_{\gamma',1}^2)}\widehat{Q}_{\circ,0}^{D\mathbf{e}_1}\widehat{Q}_{\circ,1}^{D\mathbf{e}_1}(q^{\frac{\Lambda_{\gamma,\gamma}}{2\delta}}\widehat{Q}_{\gamma,0}\widehat{Q}_{\gamma,1})^{-m_{\gamma',1}}\widehat{Y}_{\bullet,1}^{\mathbf{m}_1^{\bullet}}\widehat{Q}_{\bullet,1}^{\mathbf{n}_1^{\bullet}}\widehat{Q}_{\circ,1}^{-\mathbf{q}_0^{\circ,(1)}}\nonumber\\
&=q^{-\frac{1}{2\delta}c_4}\widehat{Y}_{\bullet,1}^{\mathbf{m}_1^{\bullet}}\widehat{Q}_{\bullet,1}^{\mathbf{n}_1^{\bullet}}\widehat{Q}_{\circ,1}^{-\mathbf{q}_0^{\circ,(1)}}\nonumber\\
&=q^{-\frac{1}{2\delta}c_4}\widehat{Q}_{\bullet,1}^{\mathbf{n}_1^{\bullet}}\widehat{Q}_{\circ,1}^{-\mathbf{q}_0^{\circ,(1)}}\widehat{Y}_{\bullet,1}^{\mathbf{m}_1^{\bullet}}.\nonumber
\end{align}
Finally, we use Lemma \ref{3.6}(6) to deduce that
}
\begin{equation}\label{eq:A.13}
\widehat{Y}_{\bullet,1}^{\mathbf{m}_1^{\bullet}}\prod_{i=1}^{s}\left(\widehat{Z}_{\alpha_i,0}^{q_{\alpha_i,1}}\widehat{Q}_{\alpha_i,1}^{-q_{\alpha_i,1}}\right)
=\prod_{i=1}^{s}\left(\widehat{Y}_{\alpha_i,1}^{m_{\alpha_i,1}}\widehat{Z}_{\alpha_i,0}^{q_{\alpha_i,1}}\widehat{Q}_{\alpha_i,1}^{-q_{\alpha_i,1}}\right).
\end{equation}
By combining equations \eqref{eq:A.7}-\eqref{eq:A.13}, it follows from the definition of $Z_{\lambda,\mathbf{n}}^{(k)}(\widehat{\mathbf{Y}}_{\vec{s}_0})$ that we have
\begin{align}
Z_{\lambda,\mathbf{n}}^{(k)}(\widehat{\mathbf{Y}}_{\vec{s}_0})
&=\sum_{\mathbf{m}^{(0,1)}\geq\mathbf{0}}\left[q^{\overline{Q}_k^{(1)}(\mathbf{m},\mathbf{n})-\frac{1}{2\delta}\sum_{\alpha\in\Pi_{\bullet}}\Lambda_{\alpha\alpha}q_{\alpha,1}(q_{\alpha,1}+1)}\prod_{(\alpha,i)\in J_{\mathfrak{g}}^{(0,1)}}\begin{bmatrix}m_{\alpha,i}+q_{\alpha,i}\\m_{\alpha,i}\end{bmatrix}_q\widehat{Q}_{\bullet,1}^{\mathbf{n}_1^{\bullet}}\widehat{Q}_{\circ,1}^{-\mathbf{q}_0^{\circ,(1)}}\right.\label{eq:A.14}\\
&\qquad\qquad\quad\left.\times\left(\prod_{i=1}^{s}\sum_{m_{\alpha_i,1}\geq0}\begin{bmatrix}m_{\alpha_i,1}+q_{\alpha_i,1}\\m_{\alpha_i,1}\end{bmatrix}_q\widehat{Y}_{\alpha_i,1}^{m_{\alpha_i,1}}\widehat{Z}_{\alpha_i,0}^{q_{\alpha_i,1}}\widehat{Q}_{\alpha_i,1}^{-q_{\alpha_i,1}}\right)\widehat{Q}_{\bullet,1}^{\mathbf{q}_2^{\bullet}}\widehat{Q}_{\circ,0}^{\mathbf{q}_1^{\circ}}\right].\nonumber
\end{align}
As $\overline{Q}_k^{(1)}(\mathbf{m},\mathbf{n})-\frac{1}{2\delta}\sum_{\alpha\in\Pi_{\bullet}}\Lambda_{\alpha\alpha}q_{\alpha,1}(q_{\alpha,1}+1)$ is independent of $m_{\alpha_i,1}$ for all $i=1,\cdots,s$, we may sum over each $m_{\alpha_i,1}$ for all $i=1,\cdots,s$, and apply Lemma \ref{4.6} to write the RHS of \eqref{eq:A.14} as 
\begin{equation}\label{eq:A.15}
\sum_{\mathbf{m}^{(0,1)}\geq\mathbf{0}}q^{\overline{Q}_k^{(1)}(\mathbf{m},\mathbf{n})}\prod_{(\alpha,i)\in J_{\mathfrak{g}}^{(0,1)}}\begin{bmatrix}m_{\alpha,i}+q_{\alpha,i}\\m_{\alpha,i}\end{bmatrix}_q\widehat{Q}_{\bullet,1}^{\mathbf{n}_1^{\bullet}}
\widehat{Q}_{\circ,1}^{-\mathbf{q}_0^{\circ,(1)}}\left(\prod_{i=1}^{s}\widehat{Z}_{\alpha_i,0}^{-1}\widehat{Z}_{\alpha_i,1}\widehat{Q}_{\alpha_i,2}^{-q_{\alpha_i,1}}\right)\widehat{Q}_{\bullet,1}^{\mathbf{q}_2^{\bullet}}\widehat{Q}_{\circ,0}^{\mathbf{q}_1^{\circ}}.
\end{equation}
Finally, we rewrite expression \eqref{eq:A.15} using Lemmas \ref{3.5} and \ref{3.6}(1), (3) and (5) to deduce that $Z_{\lambda,\mathbf{n}}^{(k)}(\widehat{\mathbf{Y}}_{\vec{s}_0})$ is equal to
\begin{equation*}
q^{-\frac{1}{\delta}\sum_{\alpha,\beta\in\Pi_{\bullet}}\Lambda_{\alpha\beta}n_{\alpha,1}}\widehat{Z}_{\bullet,0}^{-1}\widehat{Q}_{\bullet,1}^{\mathbf{n}_1^{\bullet}}\widehat{Z}_{\bullet,1}\sum_{\mathbf{m}^{(0,1)}}q^{\overline{Q}_k^{(1)}(\mathbf{m},\mathbf{n})}\prod_{(\alpha,i)\in J_{\mathfrak{g}}^{(0,1)}}\begin{bmatrix}m_{\alpha,i}+q_{\alpha,i}\\m_{\alpha,i}\end{bmatrix}_q\widehat{Q}_{\bullet,2}^{-\mathbf{q}_1^{\bullet}}\widehat{Q}_{\circ,1}^{-\mathbf{q}_0^{\circ,(1)}}\widehat{Q}_{\bullet,1}^{\mathbf{q}_2^{\bullet}}\widehat{Q}_{\circ,0}^{\mathbf{q}_1^{\circ}}.
\end{equation*}
Lemma \ref{4.4} for the case where $\mathfrak{g}$ is of type $BCF$ now follows from the definition of $Z_{\lambda,\mathbf{n}}^{(k,1)}(\widehat{\mathbf{Y}}_{\vec{s}_{0,1}})$.

\subsection{Proof of Lemma \ref{4.4} in the case where \texorpdfstring{$\mathfrak{g}$}{g} is of type \texorpdfstring{$G_2$}{G}}

In this case, we have $t_0=3$. We will first handle the case when $p=1$. Firstly, it follows from equations \eqref{eq:4.7}, \eqref{eq:4.8} and \eqref{eq:4.11} that we have
\begin{align}
q_{1,0}-q_{1,0}^{(1)}&=-m_{2,1},\quad\text{and}\label{eq:A.16}\\
q_{2,j-1}-2q_{2,j}+q_{2,j+1}&=2m_{2,j}-n_{2,j},\quad j=1,2.\label{eq:A.17}
\end{align}
This implies that $q_{1,0}(q_{2,1}-2q_{2,2}+q_{2,3})=(q_{1,0}^{(1)}-m_{2,1})(2m_{2,2}-n_{2,2})$, from which we deduce that
\begin{equation}\label{eq:A.18}
q_{1,0}q_{2,3}=-q_{1,0}q_{2,1}+2q_{1,0}q_{2,2}+2q_{1,0}^{(1)}m_{2,2}-q_{1,0}^{(1)}n_{2,2}-2m_{2,1}m_{2,2}+m_{2,1}n_{2,2}.
\end{equation}
Our next step is to compute $U_1$ explicitly. To this end, we observe that $\mathbf{e}_1=-2m_{2,1}-m_{2,2}$, $\mathbf{e}_2=m_{2,1}-m_{2,2}$, $\mathbf{e}_3=m_{2,1}+2m_{2,2}$, $\mathbf{f}_1=n_{2,1}+n_{2,2}$, $\mathbf{f}_2=n_{2,2}$ and $\mathbf{f}_3=0$, $t_0D^t\Lambda^{\circ}D=2$. This implies that
\begin{align}
U_1
&=\frac{2}{3}[(-2m_{2,1}-m_{2,2})^2+(m_{2,1}-m_{2,2})^2+(m_{2,1}+2m_{2,2})^2]\nonumber\\
&\quad+2[(-2m_{2,1}-m_{2,2})(n_{2,1}+n_{2,2})+(m_{2,1}-m_{2,2})n_{2,2}]\nonumber\\
&=2(2m_{2,1}^2+2m_{2,1}m_{2,2}+2m_{2,2}^2-2m_{2,1}n_{2,1}-m_{2,1}n_{2,2}-m_{2,2}n_{2,1}-2m_{2,2}n_{2,2}).\label{eq:A.19}
\end{align}
Using the fact that $D\mathbf{e}_1=2m_{2,1}+m_{2,2}$, it follows from combining equations \eqref{eq:4.27}, \eqref{eq:4.31}, \eqref{eq:A.18} and \eqref{eq:A.19} that we have
{\allowdisplaybreaks
\begin{align}
\overline{Q}_k(\mathbf{m},\mathbf{n})-\overline{Q}_k^{(1)}(\mathbf{m},\mathbf{n})
&= 2q_{2,1}(q_{2,1}-q_{2,2})+2q_{1,0}(2q_{2,1}-q_{2,2})+\sum_{j,\alpha=1}^2(3-\alpha)(3-j)m_{2,j}q_{\alpha,0}\label{eq:A.20}\\
&\qquad+ 2m_{2,1}(m_{2,1}-n_{2,1})+m_{2,2}(n_{2,1}-2q_{1,0}^{(1)}-2q_{2,1}+q_{2,2})-q_{1,0}^{(1)}q_{2,1}.\nonumber
\end{align}
By letting $c_1=\sum_{j,\alpha=1}^2(3-\alpha)(3-j)m_{2,j}q_{\alpha,0}$, $c_2= 2q_{2,1}(q_{2,1}-q_{2,2})+2q_{1,0}(2q_{2,1}-q_{2,2})+q_{2,1}(q_{2,1}+1)$, $c_3=2m_{2,1}(m_{2,1}-n_{2,1})$, $c_4=m_{2,2}(n_{2,1}-2q_{1,0}^{(1)}-2q_{2,1}+q_{2,2})$ and $c_5=-q_{1,0}^{(1)}q_{2,1}$, it follows that from equation \eqref{eq:A.20} that we have 
}
\begin{equation}\label{eq:A.21}
c_1+c_2+c_3+c_4+c_5=\overline{Q}_k(\mathbf{m},\mathbf{n})-\overline{Q}_k^{(1)}(\mathbf{m},\mathbf{n})+q_{2,1}(q_{2,1}+1).
\end{equation}
We first use Lemma \ref{3.5}, along with equations \eqref{eq:A.3} and \eqref{eq:A.17} to rearrange the terms involving the $\widehat{Q}_{\alpha,i}$'s in the quantum generating function $Z_{\lambda,\mathbf{n}}^{(k)}(\widehat{\mathbf{Y}}_{\vec{s}_0})$ as follows:
\begin{align}
\widehat{Q}_{2,1}^{-q_{2,0}}\widehat{Q}_{1,1}^{-q_{1,0}}\widehat{Q}_{2,0}^{q_{2,1}}\widehat{Q}_{1,0}^{2m_{2,1}+m_{2,2}}
&=q^{-c_1}\widehat{Q}_{1,0}^{2m_{2,1}+m_{2,2}}\widehat{Q}_{2,1}^{-q_{2,0}}\widehat{Q}_{1,1}^{-q_{1,0}}\widehat{Q}_{2,0}^{q_{2,1}},\label{eq:A.22}\\
\widehat{Q}_{2,1}^{-q_{2,0}}
&=(q^2\widehat{Q}_{1,0}^2\widehat{Q}_{1,1})^{-m_{2,1}}\widehat{Y}_{2,1}^{m_{2,1}}\widehat{Q}_{2,1}^{n_{2,1}-2q_{2,1}+q_{2,2}},\label{eq:A.23}\\
\widehat{Q}_{2,1}^{-2q_{2,1}+q_{2,2}}\widehat{Q}_{1,1}^{-q_{1,0}}\widehat{Q}_{2,0}^{q_{2,1}}\widehat{Q}_{2,1}^{-2q_{2,1}}
&=q^{2(q_{1,0}+q_{2,1})(-2q_{2,1}+q_{2,2})}\widehat{Q}_{1,1}^{-q_{1,0}}\widehat{Q}_{2,0}^{q_{2,1}}\widehat{Q}_{2,1}^{-2q_{2,1}+q_{2,2}}\nonumber\\
&=q^{-c_2}\widehat{Q}_{1,1}^{-q_{1,0}}\widehat{Z}_{2,0}^{q_{2,1}}\widehat{Q}_{2,1}^{-q_{2,1}}\widehat{Q}_{2,1}^{q_{2,2}},\label{eq:A.24}
\end{align}
Next, we deduce from Lemmas \ref{3.5} and \ref{3.6}(8), along with equation \eqref{eq:A.16}, that we have
\begin{align}
&\qquad\widehat{Q}_{1,0}^{2m_{2,1}}(q^2\widehat{Q}_{1,0}^2\widehat{Q}_{1,1})^{-m_{2,1}}\widehat{Y}_{2,1}^{m_{2,1}}\widehat{Q}_{2,1}^{n_{2,1}}\widehat{Q}_{1,1}^{-q_{1,0}}\label{eq:A.25}\\
&= q^{-c_3+2m_{2,1}^2}\widehat{Q}_{1,0}^{2m_{2,1}}\widehat{Q}_{1,1}^{m_{2,1}}(q^2\widehat{Q}_{1,0}^2\widehat{Q}_{1,1})^{-m_{2,1}}\widehat{Y}_{2,1}^{m_{2,1}}\widehat{Q}_{2,1}^{n_{2,1}}\widehat{Q}_{1,1}^{-q_{1,0}^{(1)}}\nonumber\\
&=q^{-c_3}\widehat{Y}_{2,1}^{m_{2,1}}\widehat{Q}_{2,1}^{n_{2,1}}\widehat{Q}_{1,1}^{-q_{1,0}^{(1)}}\nonumber\\
&=q^{-c_3}\widehat{Q}_{2,1}^{n_{2,1}}\widehat{Q}_{1,1}^{-q_{1,0}^{(1)}}\widehat{Y}_{2,1}^{m_{2,1}}.\nonumber
\end{align}
Finally, it follows from Lemmas \ref{3.5} and Lemma \ref{3.6}(8) that we have
\begin{align}
\widehat{Q}_{1,0}^{m_{2,2}}\widehat{Q}_{2,1}^{n_{2,1}}\widehat{Q}_{1,1}^{-q_{1,0}^{(1)}}\widehat{Y}_{2,1}^{m_{2,1}}\widehat{Z}_{2,0}^{q_{2,1}}\widehat{Q}_{2,1}^{-q_{2,1}+q_{2,2}}
&=q^{-c_4}\widehat{Q}_{2,1}^{n_{2,1}}\widehat{Q}_{1,1}^{-q_{1,0}^{(1)}}\widehat{Y}_{2,1}^{m_{2,1}}\widehat{Z}_{2,0}^{q_{2,1}}\widehat{Q}_{2,1}^{-q_{2,1}+q_{2,2}}\widehat{Q}_{1,0}^{m_{2,2}},\label{eq:A.26}\\
\widehat{Q}_{1,1}^{-q_{1,0}^{(1)}},\widehat{Y}_{2,1}^{m_{2,1}}\widehat{Z}_{2,0}^{q_{2,1}}\widehat{Q}_{2,1}^{-q_{2,1}}
&=q^{-c_5}\widehat{Y}_{2,1}^{m_{2,1}}\widehat{Z}_{2,0}^{q_{2,1}}\widehat{Q}_{2,1}^{-q_{2,1}}\widehat{Q}_{1,1}^{-q_{1,0}^{(1)}}.\label{eq:A.27}
\end{align}
By combining equations \eqref{eq:A.21}-\eqref{eq:A.27}, it follows from the definition of $Z_{\lambda,\mathbf{n}}^{(k)}(\widehat{\mathbf{Y}}_{\vec{s}_0})$ that we have
\begin{align}
Z_{\lambda,\mathbf{n}}^{(k)}(\widehat{\mathbf{Y}}_{\vec{s}_0})
&=\sum_{\mathbf{m}^{(0,1)}\geq\mathbf{0}}\left[q^{\overline{Q}_k^{(1)}(\mathbf{m},\mathbf{n})-q_{2,1}(q_{2,1}+1)}\prod_{(\alpha,i)\in J_{\mathfrak{g}}^{(0,1)}}\begin{bmatrix}m_{\alpha,i}+q_{\alpha,i}\\m_{\alpha,i}\end{bmatrix}_q\widehat{Q}_{2,1}^{n_{2,1}}\right.\label{eq:A.28}\\
&\qquad\qquad\left.\times\left(\sum_{m_{2,1}\geq0}\begin{bmatrix}m_{2,1}+q_{2,1}\\m_{2,1}\end{bmatrix}_q\widehat{Y}_{2,1}^{m_{2,1}}\widehat{Z}_{2,0}^{q_{2,1}}\widehat{Q}_{2,1}^{-q_{2,1}}\right)\widehat{Q}_{1,1}^{-q_{1,0}^{(1)}}\widehat{Q}_{2,1}^{q_{2,2}}\widehat{Q}_{1,0}^{q_{1,1}+m_{2,2}}\right].\nonumber
\end{align}
As $\overline{Q}_k^{(1)}(\mathbf{m},\mathbf{n})-q_{2,1}(q_{2,1}+1)$ is independent of $m_{2,1}$, we may sum over $m_{2,1}$, and apply Lemma \ref{4.6} to write the RHS of \eqref{eq:A.28} as

\begin{equation*}
\widehat{Q}_{2,1}^{n_{2,1}}\widehat{Z}_{2,0}^{-1}\widehat{Z}_{2,1}\sum_{\mathbf{m}^{(0,1)}\geq\mathbf{0}}q^{\overline{Q}_k^{(1)}(\mathbf{m},\mathbf{n})}\prod_{(\alpha,i)\in J_{\mathfrak{g}}^{(0,1)}}\begin{bmatrix}m_{\alpha,i}+q_{\alpha,i}\\m_{\alpha,i}\end{bmatrix}_q\widehat{Q}_{2,2}^{-q_{2,1}}\widehat{Q}_{1,1}^{-q_{1,0}^{(1)}}\widehat{Q}_{2,1}^{q_{2,2}}\widehat{Q}_{1,0}^{q_{1,1}+m_{2,2}}.
\end{equation*}
Lemma \ref{4.4} for the case $p=1$ now follows from the definition of $Z_{\lambda,\mathbf{n}}^{(k,1)}(\widehat{\mathbf{Y}}_{\vec{s}_{0,1}})$, with the remaining factor of $q^{-\Lambda_{22}n_{2,1}}$ obtained by commuting $\widehat{Z}_{2,0}^{-1}$ to the left of $\widehat{Q}_{2,1}^{n_{2,1}}$ using Lemma \ref{3.5}.

Finally, we will deal with the case $p=2$. Bearing in mind that we have $q_{2,1}+n_{2,2}=2q_{2,2}-q_{2,3}+2m_{2,2}$, it follows from equations \eqref{eq:4.27} and \eqref{eq:4.29} that we have
\begin{align}
\overline{Q}_k^{(1)}(\mathbf{m},\mathbf{n})-\overline{Q}_k^{(2)}(\mathbf{m},\mathbf{n})
&=2q_{2,2}(q_{2,2}-q_{2,3})+q_{1,0}^{(1)}(2q_{2,2}-q_{2,3})\label{eq:A.29}\\
&\quad+m_{2,2}(2q_{2,1}-q_{2,2}+2q_{1,0}^{(1)}-2m_{2,2}+2n_{2,2}).\nonumber
\end{align}
By letting $c_6=m_{2,2}(2q_{2,1}-q_{2,2}+2q_{1,0}^{(1)})$, $c_7=2q_{2,2}(q_{2,2}-q_{2,3})+q_{1,0}^{(1)}(2q_{2,2}-q_{2,3})+q_{2,2}(q_{2,2}+1)$ and $c_8=-2m_{2,2}(m_{2,2}-n_{2,2})$, it follows from equation \eqref{eq:A.29} that we have
\begin{equation}\label{eq:A.30}
c_6+c_7+c_8=\overline{Q}_k^{(1)}(\mathbf{m},\mathbf{n})-\overline{Q}_k^{(2)}(\mathbf{m},\mathbf{n})+q_{2,2}(q_{2,2}+1).
\end{equation}
We first use Lemma \ref{3.5}, along with equations \eqref{eq:A.3} and \eqref{eq:A.17} to rearrange the terms involving the $\widehat{Q}_{\alpha,i}$'s in the quantum generating function $Z_{\lambda,\mathbf{n}}^{(k,1)}(\widehat{\mathbf{Y}}_{\vec{s}_{0,1}})$ as follows:
\begin{align}
\widehat{Q}_{2,2}^{-q_{2,1}}\widehat{Q}_{1,1}^{-q_{1,0}^{(1)}}\widehat{Q}_{2,1}^{q_{2,2}}\widehat{Q}_{1,0}^{m_{2,2}}
&=q^{-c_6}\widehat{Q}_{1,0}^{m_{2,2}}\widehat{Q}_{2,2}^{-q_{2,1}}\widehat{Q}_{1,1}^{-q_{1,0}^{(1)}}\widehat{Q}_{2,1}^{q_{2,2}},\label{eq:A.31}\\
\widehat{Q}_{2,2}^{-q_{2,1}}
&=(q^2\widehat{Q}_{1,0}\widehat{Q}_{1,1}^2)^{-m_{2,2}}\widehat{Y}_{2,2}^{m_{2,2}}\widehat{Q}_{2,2}^{n_{2,2}-2q_{2,2}+q_{2,3}},\label{eq:A.32}\\
\widehat{Q}_{2,2}^{-2q_{2,2}+q_{2,3}}\widehat{Q}_{1,1}^{-q_{1,0}^{(1)}}\widehat{Q}_{2,1}^{q_{2,2}}
&=q^{(q_{1,0}^{(1)}+2q_{2,2})(-2q_{2,2}+q_{2,3})}\widehat{Q}_{1,1}^{-q_{1,0}^{(1)}}\widehat{Q}_{2,1}^{q_{2,2}}\widehat{Q}_{2,2}^{-2q_{2,2}+q_{2,3}}\nonumber\\
&=q^{-c_7}\widehat{Q}_{1,1}^{-q_{1,0}^{(1)}}\widehat{Z}_{2,1}^{q_{2,2}}\widehat{Q}_{2,2}^{-q_{2,2}}\widehat{Q}_{2,2}^{q_{2,3}}.\label{eq:A.33}
\end{align}
Bearing in mind that $q_{1,0}^{(1)}-q_{1,0}^{(2)}=-2m_{2,2}$, we apply Lemmas \ref{3.5} and \ref{3.6}(8) to deduce that
\begin{align}
&\qquad\widehat{Q}_{1,0}^{m_{2,2}}(q^2\widehat{Q}_{1,0}\widehat{Q}_{1,1}^2)^{-m_{2,2}}\widehat{Y}_{2,2}^{m_{2,2}}\widehat{Q}_{2,2}^{n_{2,2}}\widehat{Q}_{1,1}^{-q_{1,0}^{(1)}}\label{eq:A.34}\\
&= q^{-c_8+2m_{2,2}^2}\widehat{Q}_{1,0}^{m_{2,2}}\widehat{Q}_{1,1}^{2m_{2,2}}(q^2\widehat{Q}_{1,0}\widehat{Q}_{1,1}^2)^{-m_{2,2}}\widehat{Y}_{2,2}^{m_{2,2}}\widehat{Q}_{2,2}^{n_{2,2}}\widehat{Q}_{1,1}^{-q_{1,0}^{(2)}}\nonumber\\
&=q^{-c_8}\widehat{Y}_{2,2}^{m_{2,2}}\widehat{Q}_{2,2}^{n_{2,2}}\widehat{Q}_{1,1}^{-q_{1,0}^{(2)}}\nonumber\\
&=q^{-c_8}\widehat{Q}_{2,2}^{n_{2,2}}\widehat{Q}_{1,1}^{-q_{1,0}^{(2)}}\widehat{Y}_{2,2}^{m_{2,2}}.\nonumber
\end{align}
By combining equations \eqref{eq:A.30}-\eqref{eq:A.34}, it follows from the definition of $Z_{\lambda,\mathbf{n}}^{(k,1)}(\widehat{\mathbf{Y}}_{\vec{s}_{0,1}})$ that we have
\begin{align}
Z_{\lambda,\mathbf{n}}^{(k,1)}(\widehat{\mathbf{Y}}_{\vec{s}_{0,1}})
&=\sum_{\mathbf{m}^{(0,2)}\geq\mathbf{0}}\left[q^{\overline{Q}_k^{(2)}(\mathbf{m},\mathbf{n})-q_{2,2}(q_{2,2}+1)}\prod_{(\alpha,i)\in J_{\mathfrak{g}}^{(0,2)}}\begin{bmatrix}m_{\alpha,i}+q_{\alpha,i}\\m_{\alpha,i}\end{bmatrix}_q\widehat{Q}_{2,2}^{n_{2,2}}\widehat{Q}_{1,1}^{-q_{1,0}^{(2)}}\right.\label{eq:A.35}\\
&\qquad\qquad\quad\left.\times\left(\sum_{m_{2,2}\geq0}\begin{bmatrix}m_{2,2}+q_{2,2}\\m_{2,2}\end{bmatrix}_q\widehat{Y}_{2,2}^{m_{2,2}}\widehat{Z}_{2,1}^{q_{2,2}}\widehat{Q}_{2,2}^{-q_{2,2}}\right)\widehat{Q}_{2,2}^{q_{2,3}}\widehat{Q}_{1,0}^{q_{1,1}}\right].\nonumber
\end{align}
As $\overline{Q}_k^{(2)}(\mathbf{m},\mathbf{n})-q_{2,2}(q_{2,2}+1)$ is independent of $m_{2,2}$, we may sum over $m_{2,2}$, and apply Lemma \ref{4.6} to rewrite the RHS of equation \eqref{eq:A.35} as 
\begin{equation}\label{eq:A.36}
\sum_{\mathbf{m}^{(0,2)}\geq\mathbf{0}}q^{\overline{Q}_k^{(2)}(\mathbf{m},\mathbf{n})}\prod_{(\alpha,i)\in J_{\mathfrak{g}}^{(0,2)}}\begin{bmatrix}m_{\alpha,i}+q_{\alpha,i}\\m_{\alpha,i}\end{bmatrix}_q\widehat{Q}_{2,2}^{n_{2,2}}
\widehat{Q}_{1,1}^{-q_{1,0}^{(2)}}\widehat{Z}_{2,1}^{-1}\widehat{Z}_{2,2}\widehat{Q}_{2,3}^{-q_{2,2}}\widehat{Q}_{2,2}^{q_{2,3}}\widehat{Q}_{1,0}^{q_{1,1}}.
\end{equation}
Finally, we rewrite expression \eqref{eq:A.36} using Lemmas \ref{3.5} and \ref{3.6}(1), (3) and (5) to deduce that $Z_{\lambda,\mathbf{n}}^{(k,1)}(\widehat{\mathbf{Y}}_{\vec{s}_{0,1}})$ is equal to
\begin{equation*}
q^{-2n_{2,2}}\widehat{Z}_{2,1}^{-1}\widehat{Q}_{2,2}^{n_{2,2}}\widehat{Z}_{2,2}\sum_{\mathbf{m}^{(0,2)}}q^{\overline{Q}_k^{(2)}(\mathbf{m},\mathbf{n})}\prod_{(\alpha,i)\in J_{\mathfrak{g}}^{(0,2)}}\begin{bmatrix}m_{\alpha,i}+q_{\alpha,i}\\m_{\alpha,i}\end{bmatrix}_q\widehat{Q}_{2,3}^{-q_{2,2}}\widehat{Q}_{1,1}^{-q_{1,0}^{(2)}}\widehat{Q}_{2,2}^{q_{2,3}}\widehat{Q}_{1,0}^{q_{1,1}}.
\end{equation*}
Lemma \ref{4.4} for the case $p=2$ now follows from the previous case $p=1$, as well as the definition of $Z_{\lambda,\mathbf{n}}^{(k,2)}(\widehat{\mathbf{Y}}_{\vec{s}_{0,2}})$.

\section{Proof of Lemma \ref{4.9}}\label{Appendix B}

For convenience, let us set $p=t_0-1$. Firstly, we observe from equations \eqref{eq:4.8} and \eqref{eq:4.11} that we have
\begin{align}
\mathbf{q}_p^{\bullet}-2\mathbf{q}_{t_0}^{\bullet}+\mathbf{q}_{t_0+1}^{\bullet}
&=C^{\bullet}\mathbf{m}_{t_0}^{\bullet}-\mathbf{n}_{t_0}^{\bullet}+D^t\mathbf{m}_1^{\circ},\label{eq:B.1}\\
\mathbf{q}_0^{(p),\circ}-2\mathbf{q}_1^{\circ}+\mathbf{q}_2^{\circ}
&=C^{\circ}\mathbf{m}_1^{\circ}-\mathbf{n}_1^{\circ}+t_0D\mathbf{m}_{t_0}^{\bullet}-D\mathbf{e}_{t_0+1}.\label{eq:B.2}
\end{align}
Next, we deduce from equations \eqref{eq:4.27} and \eqref{eq:4.29} that we have
\begin{align}
&\qquad\overline{Q}_k^{(p)}(\mathbf{m},\mathbf{n})-\overline{Q}_{k-1}(\mathbf{m}^{(1,0)},\mathbf{n}^{(1)})\label{eq:B.3}\\
&=\frac{1}{\delta}\left[\mathbf{q}_1^{\circ}\cdot\Lambda^{\circ}(\mathbf{q}_1^{\circ}-\mathbf{q}_2^{\circ})+\mathbf{q}_{t_0}^{\bullet}\cdot\Lambda^{\bullet}(\mathbf{q}_{t_0}^{\bullet}-\mathbf{q}_{t_0+1}^{\bullet})+(2\mathbf{q}_1^{\circ}-\mathbf{q}_2^{\circ})\cdot A\mathbf{q}_{t_0}^{\bullet}\right.\nonumber\\
&\quad\qquad\left.-t_0\mathbf{q}_1^{\circ}\cdot A\mathbf{q}_{t_0+1}^{\bullet}-(\Lambda^{\circ}\mathbf{q}_1^{\circ}+A\mathbf{q}_{t_0}^{\bullet})\cdot D\mathbf{e}_{t_0+1}\right].\nonumber
\end{align}
Let $c_1=-(\Lambda^{\circ}\mathbf{q}_1^{\circ}+A\mathbf{q}_{t_0}^{\bullet})\cdot D\mathbf{e}_{t_0+1}$, $c_2=-(\mathbf{q}_1^{\circ}\cdot\Lambda^{\circ}\mathbf{q}_2^{\circ}+\mathbf{q}_{t_0}^{\bullet}\cdot\Lambda^{\bullet}\mathbf{q}_{t_0+1}^{\bullet}+\mathbf{q}_2^{\circ}\cdot A\mathbf{q}_{t_0}^{\bullet}+t_0\mathbf{q}_1^{\circ}\cdot A\mathbf{q}_{t_0+1}^{\bullet})$ and $c_3=\mathbf{q}_1^{\circ}\cdot\Lambda^{\circ}\mathbf{q}_1^{\circ}+\mathbf{q}_{t_0}^{\bullet}\cdot\Lambda^{\bullet}\mathbf{q}_{t_0}^{\bullet}+2\mathbf{q}_1^{\circ}\cdot A\mathbf{q}_{t_0}^{\bullet}+\frac{1}{2}\sum_{\alpha\in I_r}\Lambda_{\alpha\alpha}q_{\alpha,t_{\alpha}}(q_{\alpha,t_{\alpha}}+1)$. Then it follows from equation \eqref{eq:B.3} that we have
\begin{equation}\label{eq:B.4}
\frac{1}{\delta}(c_1+c_2+c_3)=\overline{Q}_k^{(p)}(\mathbf{m},\mathbf{n})-\overline{Q}_{k-1}(\mathbf{m}^{(1,0)},\mathbf{n}^{(1)})+\frac{1}{2\delta}\sum_{\alpha\in I_r}\Lambda_{\alpha\alpha}q_{\alpha,t_{\alpha}}(q_{\alpha,t_{\alpha}}+1).
\end{equation}
We first use Lemma \ref{3.5}, along with equations \eqref{eq:4.35}, \eqref{eq:B.1} and \eqref{eq:B.2} to rearrange the terms involving the $\widehat{Q}_{\alpha,i}$'s in the quantum generating function $Z_{\lambda,\mathbf{n}}^{(k,p)}(\widehat{\mathbf{Y}}_{\vec{s}_{0,p}})$ as follows:
\begin{align}
\widehat{Q}_{\bullet,t_0}^{-\mathbf{q}_{p}^{\bullet}}\widehat{Q}_{\circ,1}^{-\mathbf{q}_0^{\circ,(p)}}
&=\widehat{Q}_{\bullet,1}^{\mathbf{n}_{t_0}^{\bullet}}\widehat{Q}_{\circ,1}^{\mathbf{n}_1^{\circ}}\widehat{Y}_{\bullet,t_0}^{\mathbf{m}_{t_0}^{\bullet}}\widehat{Y}_{\circ,1}^{\mathbf{m}_1^{\circ}}\widehat{Q}_{\bullet,t_0}^{-2\mathbf{q}_{t_0}^{\bullet}}\widehat{Q}_{\circ,1}^{-2\mathbf{q}_1^{\circ}}\widehat{Q}_{\bullet,t_0}^{\mathbf{q}_{t_0+1}^{\bullet}}\widehat{Q}_{\circ,1}^{\mathbf{q}_2^{\circ}+D\mathbf{e}_{t_0+1}},\label{eq:B.5}\\
\widehat{Q}_{\circ,1}^{D\mathbf{e}_{t_0+1}}\widehat{Q}_{\bullet,p}^{\mathbf{q}_{t_0}^{\bullet}}\widehat{Q}_{\circ,0}^{\mathbf{q}_1^{\circ}}
&=q^{-\frac{1}{\delta}c_1}\widehat{Q}_{\bullet,p}^{\mathbf{q}_{t_0}^{\bullet}}\widehat{Q}_{\circ,0}^{\mathbf{q}_1^{\circ}}\widehat{Q}_{\circ,1}^{D\mathbf{e}_{t_0+1}},\label{eq:B.6}\\
\widehat{Q}_{\bullet,t_0}^{\mathbf{q}_{t_0+1}^{\bullet}}\widehat{Q}_{\circ,1}^{\mathbf{q}_2^{\circ}}\widehat{Q}_{\bullet,p}^{\mathbf{q}_{t_0}^{\bullet}}\widehat{Q}_{\circ,0}^{\mathbf{q}_1^{\circ}}
&=q^{-\frac{1}{\delta}c_2}\widehat{Q}_{\bullet,p}^{\mathbf{q}_{t_0}^{\bullet}}\widehat{Q}_{\circ,0}^{\mathbf{q}_1^{\circ}}\widehat{Q}_{\bullet,t_0}^{\mathbf{q}_{t_0+1}^{\bullet}}\widehat{Q}_{\circ,1}^{\mathbf{q}_2^{\circ}},\label{eq:B.7}
\end{align}
and
{\allowdisplaybreaks
\begin{align}
&\qquad\widehat{Q}_{\bullet,t_0}^{-2\mathbf{q}_{t_0}^{\bullet}}\widehat{Q}_{\circ,1}^{-2\mathbf{q}_1^{\circ}}\widehat{Q}_{\bullet,p}^{\mathbf{q}_{t_0}^{\bullet}}\widehat{Q}_{\circ,0}^{\mathbf{q}_1^{\circ}}\nonumber\\
&=q^{-\frac{1}{\delta}(2\mathbf{q}_1^{\circ}\cdot\Lambda^{\circ}\mathbf{q}_1^{\circ}+2\mathbf{q}_{t_0}^{\bullet}\cdot\Lambda^{\bullet}\mathbf{q}_{t_0}^{\bullet}+2\mathbf{q}_1^{\circ}\cdot A\mathbf{q}_{t_0}^{\bullet})}\widehat{Q}_{\bullet,p}^{\mathbf{q}_{t_0}^{\bullet}}\widehat{Q}_{\bullet,t_0}^{-2\mathbf{q}_{t_0}^{\bullet}}\widehat{Q}_{\circ,0}^{\mathbf{q}_1^{\circ}}\widehat{Q}_{\circ,1}^{-2\mathbf{q}_1^{\circ}}\nonumber\\
&=q^{-\frac{1}{\delta}\left(c_3+\frac{1}{2}\sum_{\alpha\in I_r}\Lambda_{\alpha\alpha}q_{\alpha,t_{\alpha}}(q_{\alpha,t_{\alpha}}-1)\right)}\prod_{i=1}^s\left(\widehat{Q}_{\alpha_i,p}^{q_{\alpha_i,t_0}}\widehat{Q}_{\alpha_i,t_0}^{-2q_{\alpha_i,t_0}}\right)\prod_{j=1}^d\left(\widehat{Q}_{\beta_j,0}^{q_{\beta_j,1}}\widehat{Q}_{\beta_j,1}^{-2q_{\beta_j,1}}\right)\nonumber\\
&=q^{-\frac{1}{\delta}c_3}\prod_{i=1}^s\left(\widehat{Z}_{\alpha_i,p}^{q_{\alpha_i,t_0}}\widehat{Q}_{\alpha_i,t_0}^{-q_{\alpha_i,t_0}}\right)\prod_{j=1}^d\left(\widehat{Z}_{\beta_j,0}^{q_{\beta_j,1}}\widehat{Q}_{\beta_j,1}^{-q_{\beta_j,1}}\right).\label{eq:B.8}
\end{align}
Finally, we use Lemma \ref{3.6}(6) and (7) to deduce that
}
\begin{align}
&\quad\widehat{Y}_{\bullet,t_0}^{\mathbf{m}_{t_0}^{\bullet}}\widehat{Y}_{\circ,1}^{\mathbf{m}_1^{\circ}}\prod_{i=1}^s\left(\widehat{Z}_{\alpha_i,p}^{q_{\alpha_i,t_0}}\widehat{Q}_{\alpha_i,t_0}^{-q_{\alpha_i,t_0}}\right)\prod_{j=1}^d\left(\widehat{Z}_{\beta_j,0}^{q_{\beta_j,1}}\widehat{Q}_{\beta_j,1}^{-q_{\beta_j,1}}\right)\nonumber\\
&=\prod_{i=1}^s\left(\widehat{Y}_{\alpha_i,t_0}^{m_{\alpha_i,t_0}}\widehat{Z}_{\alpha_i,p}^{q_{\alpha_i,t_0}}\widehat{Q}_{\alpha_i,t_0}^{-q_{\alpha_i,t_0}}\right)\prod_{j=1}^d\left(\widehat{Y}_{\beta_j,1}^{m_{\beta_j,1}}\widehat{Z}_{\beta_j,0}^{q_{\beta_j,1}}\widehat{Q}_{\beta_j,1}^{-q_{\beta_j,1}}\right).\label{eq:B.9}
\end{align}
We now combine equations \eqref{eq:B.4}-\eqref{eq:B.9} to deduce that
{\allowdisplaybreaks
\begin{align}
Z_{\lambda,\mathbf{n}}^{(k,p)}(\widehat{\mathbf{Y}}_{\vec{s}_{0,p}})
&=\widehat{Q}_{\bullet,1}^{\mathbf{n}_{t_0}^{\bullet}}\widehat{Q}_{\circ,1}^{\mathbf{n}_1^{\circ}}\sum_{\mathbf{m}^{(1,0)}}\left[q^{\overline{Q}_{k-1}(\mathbf{m}^{(1,0)},\mathbf{n}^{(1)})-\frac{1}{2\delta}\sum_{\alpha\in I_r}\Lambda_{\alpha\alpha}q_{\alpha,t_{\alpha}}(q_{\alpha,t_{\alpha}}+1)}\prod_{(\alpha,i)\in J_{\mathfrak{g}}^{(1,0)}}\begin{bmatrix}m_{\alpha,i}+q_{\alpha,i}\\m_{\alpha,i}\end{bmatrix}_q\right.\label{eq:B.10}\\
&\qquad\qquad\times\left(\prod_{i=1}^s\sum_{m_{\alpha_i,t_0}\geq0}\begin{bmatrix}m_{\alpha_i,t_0}+q_{\alpha_i,t_0}\\m_{\alpha_i,t_0}\end{bmatrix}_q\widehat{Y}_{\alpha_i,t_0}^{m_{\alpha_i,t_0}}\widehat{Z}_{\alpha_i,p}^{q_{\alpha_i,t_0}}\widehat{Q}_{\alpha_i,t_0}^{-q_{\alpha_i,t_0}}\right)\nonumber\\
&\qquad\qquad\times\left.\left(\prod_{j=1}^d\sum_{m_{\beta_j,1}\geq0}\begin{bmatrix}m_{\beta_j,1}+q_{\beta_j,1}\\m_{\beta_j,1}\end{bmatrix}_q\widehat{Y}_{\beta_j,1}^{m_{\beta_j,1}}\widehat{Z}_{\beta_j,0}^{q_{\beta_j,1}}\widehat{Q}_{\beta_j,1}^{-q_{\beta_j,1}}\right)\widehat{Q}_{\bullet,t_0}^{\mathbf{q}_{t_0+1}^{\bullet}}\widehat{Q}_{\circ,1}^{\mathbf{q}_2^{\circ}+D\mathbf{e}_{t_0+1}}\right].\nonumber
\end{align}
As $\overline{Q}_{k-1}(\mathbf{m}^{(1,0)},\mathbf{n}^{(1)})-\frac{1}{2\delta}\sum_{\alpha\in I_r}\Lambda_{\alpha\alpha}q_{\alpha,t_{\alpha}}(q_{\alpha,t_{\alpha}}+1)$ is independent of $m_{\beta_j,1}$ and $m_{\alpha_i,t_0}$ for all $i=1,\cdots,s$ and $j=1,\cdots,d$, we may sum over each $m_{\beta_j,1}$ and $m_{\alpha_i,t_0}$ for all $i=1,\cdots,s$ and $j=1,\cdots,d$, and apply Lemma \ref{4.6} to write the RHS of equation \eqref{eq:B.10} as 
}
\begin{align}
&\quad\widehat{Q}_{\bullet,1}^{\mathbf{n}_{t_0}^{\bullet}}\widehat{Q}_{\circ,1}^{\mathbf{n}_1^{\circ}}\sum_{\mathbf{m}^{(1,0)}}\left[q^{\overline{Q}_{k-1}(\mathbf{m}^{(1,0)},\mathbf{n}^{(1)})}\prod_{(\alpha,i)\in J_{\mathfrak{g}}^{(1,0)}}\begin{bmatrix}m_{\alpha,i}+q_{\alpha,i}\\m_{\alpha,i}\end{bmatrix}_q\left(\prod_{i=1}^s\widehat{Z}_{\alpha_i,p}^{-1}\widehat{Z}_{\alpha_i,t_0}\widehat{Q}_{\alpha_i,t_0+1}^{-q_{\alpha_i,t_0}}\right)\right.\label{eq:B.11}\\
&\qquad\qquad\qquad\qquad\left.\times\left(\prod_{j=1}^d\widehat{Z}_{\beta_j,0}^{-1}\widehat{Z}_{\beta_j,1}\widehat{Q}_{\beta_j,2}^{-q_{\beta_j,1}}\right)\widehat{Q}_{\bullet,t_0}^{\mathbf{q}_{t_0+1}^{\bullet}}\widehat{Q}_{\circ,1}^{\mathbf{q}_2^{\circ}+D\mathbf{e}_{t_0+1}}\right].\nonumber
\end{align}
Finally, we rewrite expression \eqref{eq:B.11} using Lemma \ref{3.6}(1), (2), (3) and (4) to get
\begin{equation*}
Z_{\lambda,\mathbf{n}}^{(k,p)}(\widehat{\mathbf{Y}}_{\vec{s}_{0,p}})
=\widehat{Q}_{\bullet,1}^{\mathbf{n}_{t_0}^{\bullet}}\widehat{Q}_{\circ,1}^{\mathbf{n}_1^{\circ}}\widehat{Z}_{\bullet,p}^{-1}\widehat{Z}_{\circ,0}^{-1}\widehat{Z}_{\circ,1}\widehat{Z}_{\bullet,t_0}Z_{\lambda,\mathbf{n}^{(1)}}^{(k-1)}(\widehat{\mathbf{Y}}_{\vec{s}_1}).
\end{equation*}
Lemma \ref{4.9} now follows from Lemma \ref{4.4}, along with a similar argument at the end of the proof of Lemma \ref{4.3}.

\sc{
\address{Department of Mathematics, University of Illinois MC-382, Urbana, IL 61821, U.S.A.}
\email{mlin39@illinois.edu}
}


\begin{thebibliography}{10}
\bibitem{AK07} E. Ardonne and R. Kedem. Fusion products of Kirillov-Reshetikhin modules and fermionic multiplicity formulas. \emph{J. Algebra}, 308(1):270--294, 2007.
\bibitem{BZ05} A. Berenstein and A. Zelevinsky. Quantum cluster algebras. \emph{Adv. Math.}, 195(2):405--455, 2005.
\bibitem{Chari01} V. Chari. On the fermionic formula and the Kirillov-Reshetikhin conjecture. \emph{Int. Math. Res. Not. IMRN}, (12):629--654, 2001.
\bibitem{CM06} V. Chari and A. Moura. The restricted Kirillov-Reshetikhin modules for the current and twisted current algebras. \emph{Comm. Math. Phys.}, 266(2):431--454, 2006.
\bibitem{CV15} V. Chari and R. Venkatesh. Demazure modules, Fusion products and \textit{Q}-systems. \emph{Comm. Math. Phys.}, 333(2):799--830, 2015.
\bibitem{DFK08} P. Di Francesco and R. Kedem. Proof of the combinatorial Kirillov-Reshetikhin conjecture. \emph{Int. Math. Res. Not. IMRN}, (7):Art. ID rnn006, 57, 2008.
\bibitem{DFK09} P. Di Francesco and R. Kedem. \textit{Q}-systems as cluster algebras. II. Cartan matrix of finite type and the polynomial property. \emph{Lett. Math. Phys.}, 89(3):183--216, 2009.
\bibitem{DFK11} P. Di Francesco and R. Kedem. Non-commutative integrability, paths and quasi-determinants. \emph{Adv. Math.}, 228:97--152, 2011.
\bibitem{DFK14} P. Di Francesco and R. Kedem. Quantum cluster algebras and fusion products. \emph{Int. Math. Res. Not. IMRN}, (10):2593--2642, 2014.
\bibitem{FL99} B. Feigin and S. Loktev. On generalized Kostka polynomials and the quantum Verlinde rule. In \emph{Differential topology, infinite-dimensional Lie algebras, and applications}, volume 194 of \emph{Amer. Math. Soc. Transl. Ser. 2}, pages 61--79. Amer. Math. Soc., Providence, RI, 1999.
\bibitem{FZ02} S. Fomin and A. Zelevinsky. Cluster algebras. I. Foundations. \emph{J. Amer. Math. Soc.}, 15(2):497--529 (electronic), 2002.
\bibitem{GSV03} M. Gekhtman, M. Shapiro, and A. Vainshtein. Cluster Algebras and Poisson Geometry. \emph{Mosc. Math. J.}, 3(3):899--934, 2003.
\bibitem{HKOTT02} G. Hatayama, A. Kuniba, M. Okado, T. Takagi, and Z. Tsuboi. Paths, crystals and fermionic formulae. In \emph{MathPhys odyssey, 2001}, volume 23 of \emph{Prog. Math. Phys.}, pages 205--272. Birkh\"{a}user Boston, Boston, MA, 2002.
\bibitem{HKOTY99} G. Hatayama, A. Kuniba, M. Okado, T. Takagi, and Y. Yamada. Remarks on fermionic formula. In \emph{Recent developments in quantum affine algebras and related topics (Raleigh, NC, 1998)}, volume 248 of \emph{Contemp. Math.}, pages 243--291. Amer. Math. Soc., Providence, RI, 1999.
\bibitem{Hernandez06} D. Hernandez. The Kirillov-Reshetikhin conjecture and solutions of \textit{T}-systems. \emph{J. Reine Angew. Math.}, 596:63--87, 2006.
\bibitem{Hernandez10} D. Hernandez. Kirillov-Reshetikhin conjecture: the general case. \emph{Int. Math. Res. Not. IMRN}, (1):149--193, 2010.
\bibitem{Kedem08} R. Kedem. Q-systems as cluster algebras. \emph{J. Phys. A}, 41(19):194011, 14, 2008.
\bibitem{Kedem11} R. Kedem. A pentagon of identities, fusion products and the Kirillov-Reshetikhin conjecture. In \emph{New Trends in Quantum Integrable Systems}, pages 173--193. World Scientific, Hackensack, NJ, 2011.
\bibitem{Kirillov89} A. N. Kirillov. Identities for the Rogers dilogarithm function connected with simple Lie algebras. \emph{J. Sov. Math.}, 47:2450--2459, 1989.
\bibitem{KR87} A. N. Kirillov and N. Yu. Reshetikhin. Representations of Yangians and multiplicities of the occurrence of the irreducible components of the tensor product of representations of simple lie algebras. \emph{Zap. Nauchn. Sem. Leningrad. Otdel. Mat. Inst. Steklov. (LOMI)}, 160(Anal. Teor. Chisel i Teor. Funktsii. 8):211--221, 301, 1987.
\bibitem{KSS02} A. N. Kirillov, A. Schilling and M. Shimozono. A bijection between Littlewood-Richardson tableaux and rigged configurations. \emph{Selecta Math. (N.S.)}, 8(1):67--135, 2002.
\bibitem{KNS94} A. Kuniba, T. Nakanishi and J. Suzuki. Functional relations in solvable lattice models I: functional relations and representation theory. \emph{Internat. J. Modern Phys. A}, 9(30):5215--5266, 1994.
\bibitem{KS95} A. Kuniba and J. Suzuki. Functional relations and analytic Bethe ansatz for twisted quantum affine algebras. \emph{J. Phys. A}, 28(3):711--722, 1995.
\bibitem{Lin19} M. S. Lin. Quantum \textit{Q}-systems and fusion products -- the twisted case, \emph{in preparation}. 
\bibitem{Nakajima03} H. Nakajima. \textit{t}-analogs of \textit{q}-characters of Kirillov-Reshetikhin modules of quantum affine algebras. \emph{Represent. Theory}, 7:259--274 (electronic), 2003.
\bibitem{OSS18} M. Okado, A. Schilling and T. Scrimshaw. Rigged configuration bijection and proof of the $X=M$ conjecture for nonexceptional affine types. \emph{J. Algebra}, 516(1):1--37, 2018.
\bibitem{Williams15} H. Williams. \textit{Q}-Systems, Factorization Dynamics, and the Twist Automorphism. \emph{Int. Math. Res. Not. IMRN}, (22):12042--12069, 2015.
\end{thebibliography}
\end{document}